%% file: main.tex
\documentclass[11pt,a4paper]{amsart}
\usepackage[T1]{fontenc}
\usepackage[utf8]{inputenc}
\usepackage{amsmath,amsfonts,amssymb,amsthm}
\usepackage{mathtools}
\usepackage{pgf,tikz}
\usepackage{bbm}
\usetikzlibrary{matrix,arrows,calc}
\usepackage{tikz-cd}
\usepackage{shuffle}
\usepackage{tensor}
\usepackage{hyperref}
\usepackage{cleveref}
\usepackage{enumitem}

\tikzset{
	symbol/.style={
		draw=none,
		every to/.append style={
			edge node={node [sloped, allow upside down, auto=false]{$#1$}}}
	}
}

\numberwithin{equation}{section}

\hypersetup{
	pdftitle={Chabauty--Kim, finite descent, and the Section Conjecture for locally geometric sections},
	pdfauthor={L. Alexander Betts, Theresa Kumpitsch, Martin Lüdtke},
	colorlinks=true,    %
}

\usepackage[colorinlistoftodos]{todonotes}
\definecolor{orange}{rgb}{0.9, 0.9, 0.1}
\definecolor{turquoise}{rgb}{0.19, 0.84, 0.78}

\newtheorem{thm}{Theorem}
\newtheorem{thmABC}{Theorem}
\newtheorem*{thmABC*}{Theorem}

\newtheorem{lem}[thm]{Lemma}
\newtheorem{prop}[thm]{Proposition}
\newtheorem{cor}[thm]{Corollary}
\newtheorem{conj}[thm]{Conjecture}

\theoremstyle{remark}
\newtheorem{rem}[thm]{Remark}

\theoremstyle{definition}
\newtheorem{defn}[thm]{Definition}

\numberwithin{thm}{section}

\AddToHook{env/defn/begin}{\crefalias{thm}{defn}}
\AddToHook{env/prop/begin}{\crefalias{thm}{prop}}
\AddToHook{env/lem/begin}{\crefalias{thm}{lem}}
\AddToHook{env/conj/begin}{\crefalias{thm}{conj}}
\AddToHook{env/cor/begin}{\crefalias{thm}{cor}}
\AddToHook{env/rem/begin}{\crefalias{thm}{rem}}
\AddToHook{env/example/begin}{\crefalias{thm}{example}}

\newcommand{\bA}{\mathbb A}

\newcommand{\Aut}{\mathrm{Aut}}

\newcommand{\bC}{\mathbb C}
\newcommand{\rd}{\mathrm d}

\newcommand{\sD}{\mathsf{D}}
\newcommand{\Ext}{\mathrm{Ext}}
\newcommand{\bF}{\mathbb F}

\newcommand{\Fbar}{{\overline{F}}}
\newcommand{\bG}{\mathbb G}
\newcommand{\dG}{\mathcal G}
\newcommand{\Gal}{\mathrm{Gal}}
\newcommand{\rH}{\mathrm H}

\newcommand{\Lie}{\mathrm{Lie}}

\newcommand{\fP}{\mathfrak P}
\newcommand{\cM}{\mathcal{M}}

\newcommand{\bP}{\mathbb P}

\newcommand{\QQ}{\mathbb Q}
\newcommand{\bQ}{\mathbb Q}
\newcommand{\bR}{\mathbb R}

\newcommand\Qbar{{\overline\bQ}}
\newcommand\Kbar{{\overline{K}}}

\newcommand{\cO}{\mathcal{O}}

\newcommand{\dY}{\mathcal{Y}}
\newcommand{\cY}{\mathcal{Y}}

\newcommand{\dT}{\mathcal{T}}

\newcommand{\ZZ}{\mathbb Z}
\newcommand{\bZ}{\mathbb Z}
\newcommand{\rZ}{\mathrm Z}

\newcommand\et{\mathrm{\acute et}}

\newcommand{\logBK}{\log_{\mathrm{BK}}}

\newcommand\lgeom{\mathrm{Sel}}

\newcommand{\DM}{\mathrm{DM}}

\newcommand{\MT}{\mathrm{MT}}
\newcommand{\MTR}{\mathrm{MTR}}

\newcommand\Zar{\mathrm{Zar}}

\mathchardef\mhyphen="2D %
\newcommand\fab{\mathrm{f\mhyphen ab}}
\newcommand\fgab{\mathrm{f\mhyphen gab}}

\newcommand\fcov{\mathrm{f\mhyphen cov}}

\newcommand\mot{\mathrm{mot}}

\newcommand{\ind}[1]{\mathrm{ind}\text{-}#1}
\DeclareMathOperator{\li}{li}

\DeclareMathOperator{\ab}{ab}

\DeclareMathOperator{\cris}{cris}

\DeclareMathOperator{\dR}{dR}
\DeclareMathOperator{\im}{im}

\DeclareMathOperator{\gr}{gr}

\DeclareMathOperator{\loc}{loc}
\DeclareMathOperator{\Li}{Li}

\DeclareMathOperator{\Res}{Res}

\DeclareMathOperator{\Hom}{Hom}
\DeclareMathOperator{\per}{per}
\newcommand{\gab}{{\mathrm{gab}}}

\newcommand{\redu}{{\mathrm{red}}}
\DeclareMathOperator{\PL}{PL}
\DeclareMathOperator{\Rep}{Rep}
\DeclareMathOperator{\Sel}{Sel}
\DeclareMathOperator{\Spec}{Spec}

\DeclareMathOperator{\Sec}{Sec}

\newcommand{\ur}{\mathrm{ur}}
\newcommand{\ev}{\mathrm{ev}}
\newcommand{\proj}{\mathrm{proj}}

\usepackage[
	backend=bibtex,
	style=alphabetic,
	minalphanames=3,
	maxnames=99,
	maxalphanames=4,
	giveninits=true,
	isbn=false,
	doi=false,
]{biblatex}
\title[Chabauty--Kim, finite descent, and the Section Conjecture]{Chabauty--Kim, finite descent, and the Section Conjecture for locally geometric sections}
\author{L. Alexander Betts}
\author{Theresa Kumpitsch}
\author{Martin Lüdtke}

\begin{document}

\begin{abstract}
	Let~$X$ be a smooth projective curve of genus~$\geq2$ over a number field. We consider the relationship between Kim's Conjecture on the sufficiency of the Chabauty--Kim method, a conjecture of Stoll on the sufficiency of the finite descent obstruction, and the locally geometric part of Grothendieck's Section Conjecture. We show that the latter two conjectures are equivalent to one another, and are a consequence of~$X$ satisfying Kim's Conjecture for almost all choices of auxiliary prime~$p$. This leads to a new ``computational'' strategy, based on Chabauty--Kim theory, for proving the locally geometric variant of the Section Conjecture. We give the appropriate generalisation for $S$-integral points on hyperbolic curves and, as a proof of concept, carry out the proposed strategy for the thrice-punctured line over~$\bZ[1/2]$, in the process establishing infinitely many new cases of Kim's Conjecture.
\end{abstract}

	\maketitle
	
	\thispagestyle{empty}
	\vspace{0.2cm}
	\tableofcontents
	
\input{intro}

\input{descent}

\input{motives}
\input{refined}
\input{calculations}
\appendix
\input{appendix}

\printbibliography

\end{document}

%% file: intro.tex
\section{Introduction}

Let $X$ be a smooth projective curve of genus $\geq 2$ over a number field $K$ with algebraic closure~$\Kbar$ and absolute Galois group $G_K = \Gal(\Kbar/K)$. Via the \textit{fundamental exact sequence}
\begin{equation}\label{eq:fes}
1 \to \pi^{\et}_1(X_{\Kbar}) \to \pi^{\et}_1(X) \to G_K \to 1
\end{equation}
one may study the $K$-rational points of $X$: each such point $x \in X(K)$ induces a section of this sequence, well-defined up to conjugation by~$\pi^{\et}_1(X_{\Kbar})$.\footnote{Strictly speaking, one should choose a geometric basepoint on $X_{\Kbar}$ to make sense of these fundamental groups; we suppress this choice for the purposes of the introduction.} Grothendieck's famous \textit{Section Conjecture} states that the set of rational points $X(K)$ should be in bijection with the set~$\Sec(X/K)$ of $\pi^{\et}_1(X_{\Kbar})$-conjugacy classes of splittings of~\eqref{eq:fes}. The injectivity is easy to show, but surjectivity is still a wide open problem, and there exist only a few examples where this is known to hold, all in the special case that~$X(K)=\emptyset$ \cite{harari-szamuely:no_abelian_sections,li-litt-salter-srinivasan:no_sections,stix:leereSC,stix:brauer-manin_for_sections}.

In this paper, we are interested in a variant of the Section Conjecture of local-to-global nature.

\begin{defn}\label{def:selmer_sc}
	A section~$s$ of~\eqref{eq:fes} is \emph{Selmer} just when for every place~$v$ of~$K$, the restriction $s|_{G_v}$ to a decomposition group at~$v$ is the section of the local fundamental exact sequence
	\begin{equation} \label{eq:local_fes}
		1 \to \pi^{\et}_1(X_{\Kbar_v}) \to \pi^{\et}_1(X_{K_v}) \to G_v \to 1
	\end{equation}
	coming from a $K_v$-rational point~$x_v \in X(K_v)$. The set of $\pi^{\et}_1(X_{\Kbar})$-conjugacy classes of Selmer sections is denoted by $\Sec(X/K)^{\Sel}$.
\end{defn}

Sections which are induced globally by a $K$-rational point are clearly Selmer, so we have embeddings
\begin{equation} \label{eq:section_conjecture_inclusions}
	X(K) \subseteq \Sec(X/K)^{\Sel} \subseteq \Sec(X/K).
\end{equation}
It is natural to divide Grothendieck's Section Conjecture into two subproblems, stating that each of the two inclusions in~\eqref{eq:section_conjecture_inclusions} is a bijection:
\begin{enumerate}
	\item \label{item:sc_part_1}
	Show that every Selmer section is induced by a $K$-rational point.
	\item \label{item:sc_part_2}
	Show that every section is Selmer.
\end{enumerate}

The second problem amounts to the $p$-adic Section Conjecture, i.e.\ the analogue of the Section Conjecture over finite extensions of~$\bQ_p$. Just like for number fields, the $p$-adic Section Conjecture is still open, although there are some partial results, most notably the proof of the Birational $p$-adic Section Conjecture \cite{koenigsmann:SC}. The latter implies that every birationally liftable section of $\pi_1^{\et}(X) \to G_K$ (i.e., lifting to a section of $G_{K(X)} \to G_K$, where $K(X)$ is the function field of~$X$) is Selmer \cite[Proposition~1]{stix:birationalSC}.

This paper is concerned with the first problem, which we call the Selmer Section Conjecture.\footnote{As remarked above, this variant of the Section Conjecture is natural, both in being a possible route towards addressing the Section Conjecture itself, and also in its connection to the finite descent obstruction, as we explain in \S\ref{ss:finite_descent}. Although this conjecture does not appear to be written down explicitly in the literature, a number of mathematicians have considered some version of it, including Mohamed Sa\"idi, who introduced the conjecture to the first author. Minhyong Kim informs us that Florian Pop and Akio Tamagawa were also interested in this question.}

\begin{conj}[Selmer Section Conjecture, cf.\ {\cite{betts-stix}}]\label{conj:selmer_sc}
	Let~$X/K$ be a smooth projective curve of genus~$\geq 2$ over a number field~$K$. Then
	\[
	X(K) = \Sec(X/K)^{\Sel} \,.
	\]
\end{conj}

In this paper, we propose a new strategy for proving instances of the Selmer Section Conjecture. The viewpoint we wish to promote, building on \cite{betts-stix}, is that $p$-adic methods for studying rational and $S$-integral points on curves can also be used to study Selmer sections. In \cite{betts-stix}, the first author and Jakob Stix applied the Lawrence--Venkatesh method to prove finiteness results for the Selmer section set of a general~$X$. In this paper, we instead demonstrate that one can use the Chabauty--Kim method to \emph{explicitly compute} the Selmer section set for particular curves.

To fix notation, suppose now that~$K=\bQ$ and~$X/\bQ$ is a smooth projective curve of genus~$\geq 2$. Suppose moreover that~$X$ has at least one rational point (for use as a basepoint for fundamental groups). The Chabauty--Kim method as described in \cite{BDCKW} defines for any prime~$p$ a \emph{Chabauty--Kim locus}
\[
X(\bQ_p)_{\infty} \subseteq X(\bQ_p)
\]
containing the rational points~$X(\bQ)$.\footnote{The Chabauty--Kim locus $X(\bQ_p)_{\infty}$ is defined in \cite{BDCKW} under the assumption that $p$ is a prime of good reduction. When we recall the definition in Definition~\ref{def: refined Selmer scheme}, we will take care to formulate the definition in a way which does not require good reduction, though the general theory is not yet strong enough to study the Chabauty--Kim locus in the bad reduction case.} The Chabauty--Kim locus is by definition the common vanishing locus of some Coleman analytic functions on~$X(\bQ_p)$; in many cases, one can compute several of these Coleman functions explicitly, and this sometimes enables one to compute the set of rational points exactly. In fact, it is conjectured that in general these Coleman functions should always cut out exactly the set of rational points.

\begin{conj}[Kim's Conjecture, {\cite[Conjecture~3.1]{BDCKW}}]\label{conj:kim_projective}
	Let~$X/\bQ$ be a smooth projective curve of genus~$\geq 2$ with $X(\bQ) \neq \emptyset$. Let~$p$ be a prime. Then
	\[
	X(\bQ) = X(\bQ_p)_{\infty} \,.
	\]
\end{conj}

The main theoretical result we will prove in this paper makes precise the relationship between Kim's Conjecture and the Selmer Section Conjecture.

\begin{thm}[= \Cref{thm:main_conjecture_implication}, projective case]\label{thm:main_conjecture_implication_projective}
	Let~$X/\bQ$ be a smooth projective curve of genus~$\geq 2$ with $X(\bQ) \neq \emptyset$. Suppose that Kim's Conjecture~\ref{conj:kim_projective} holds for $(X,p)$ for all primes~$p$ in a set of primes of Dirichlet density~$1$. Then the Selmer Section Conjecture~\ref{conj:s-selmer_sc} holds for~$X$.
\end{thm}

This theorem thus gives a new strategy for proving instances of the Selmer Section Conjecture: if one can compute the Chabauty--Kim locus $X(\bQ_p)_{\infty}$ for a density~$1$ set of primes~$p$ and show that it is equal to the set of rational points, then one obtains for free that the Selmer Section Conjecture holds for~$X$. The difficulty here, of course, lies in carrying out Chabauty--Kim computations for infinitely many different primes~$p$ at once.

We will demonstrate the viability of this approach by working out one example of a hyperbolic curve where we can indeed verify Kim's Conjecture for infinitely many primes at once. The specific example we will study is that of the thrice-punctured line over~$\bZ[1/2]$, which is an \emph{affine} hyperbolic curve. Therefore, we formulate a more general version of our main theorem which holds for $S$-integral points on arbitrary hyperbolic curves, projective or not.

\begin{defn}\label{def:s-selmer}
Let~$Y$ be a hyperbolic curve over a number field~$K$, let~$S$ be a finite set of places of~$K$ containing all archimedean places, and let~$\cY/\cO_{K,S}$ be a regular model of~$Y$, meaning that~$\cY$ is the complement of a horizontal divisor in a flat proper regular $\cO_{K,S}$-scheme and comes with an isomorphism $\cY_K\cong Y$ of its generic fibre with~$Y$. We say that a section~$s$ of the global fundamental exact sequence~\eqref{eq:fes} (with~$X$ replaced with~$Y$) is \emph{$S$-Selmer} just when the restriction $s|_{G_v}$ to a decomposition group at a place~$v$ is the section of the local fundamental exact sequence~\eqref{eq:local_fes}
coming from a
\[
\begin{cases}
	\text{$K_v$-rational point $y_v \in Y(K_v)$} & \text{for all $v \in S$,} \\
	\text{$\cO_v$-integral point $y_v \in \cY(\cO_v)$} & \text{for all $v \not\in S$.}
\end{cases}
\]
Here~$\cO_v$ is the ring of integers of~$K_v$. The set of $\pi^{\et}_1(Y_{\Kbar})$-conjugacy classes of $S$-Selmer sections is denoted by $\Sec(\cY/\cO_{K,S})^{\lgeom}$.
\end{defn}

The map sending an $S$-integral point to its corresponding section of~\eqref{eq:fes} embeds $\cY(\cO_{K,S})$ as a subset of $\Sec(\cY/\cO_{K,S})^{\lgeom}$, and the Section Conjecture would imply that this inclusion is an equality. So the Selmer Section Conjecture has the following natural analogue for~$S$-integral points.

\begin{conj}[$S$-Selmer Section Conjecture]\label{conj:s-selmer_sc}
Let~$Y/K$ be a hyperbolic curve over a number field~$K$ and let~$\cY/\cO_{K,S}$ be a regular $S$-integral model of~$Y$. Then
\[
\cY(\cO_{K,S}) = \Sec(\cY/\cO_{K,S})^{\lgeom} \,.
\]
\end{conj}

The Chabauty--Kim method also can be applied in the setting of affine hyperbolic curves. Suppose that~$K=\bQ$ and~$\cY/\bZ_S$ is a regular $S$-integral model of a hyperbolic curve~$Y$, where~$S$ is a finite set of primes. Suppose moreover that the smooth compactification of~$Y$ has a rational point (so~$Y$ has either a rational point or a rational tangential point for use as a basepoint for fundamental groups). The Chabauty--Kim method, specifically the refined Chabauty--Kim method described in \cite{BD:refined}, defines for any prime~$p\notin S$ a \emph{refined Chabauty--Kim locus}
\[
\cY(\bZ_p)_{S,\infty}^{\min} \subseteq \cY(\bZ_p)
\]
containing the $S$-integral points~$\cY(\bZ_S)$ and being defined as the common vanishing locus of some Coleman analytic functions on~$\cY(\bZ_p)$. The refined version of Kim's Conjecture predicts that this vanishing locus is precisely the set of $S$-integral points.

\begin{conj}[Refined Kim's Conjecture]\label{conj:kim}
	Let~$Y/\bQ$ be a hyperbolic curve and let~$\cY/\bZ_S$ be a regular $S$-integral model of~$Y$ for a finite set of primes~$S$. Suppose that the smooth compactification of~$Y$ has a rational point. Let~$p\notin S$ be prime. Then
	\[
	\cY(\bZ_S) = \cY(\bZ_p)_{S,\infty}^{\min} \,.
	\]
\end{conj}

We can now state the general version of our main theorem in the setting of possibly affine hyperbolic curves.

\begin{thmABC}\label{thm:main_conjecture_implication}
	Let~$\cY/\bZ_S$ be a regular $S$-integral model of a hyperbolic curve~$Y/\bQ$ whose smooth compactification has a rational point. Suppose that the refined Kim's Conjecture~\ref{conj:kim} holds for $(\cY,S,p)$ for all primes~$p$ in a set of primes of Dirichlet density~$1$. Then the $S$-Selmer Section Conjecture~\ref{conj:s-selmer_sc} holds for~$(\cY,S)$.
\end{thmABC}

Note that if the hyperbolic curve in question happens to be projective, everything in sight is independent of the set~$S$ and model~$\cY$ and this theorem specialises to the one stated for projective curves above. 

In order to show that \Cref{thm:main_conjecture_implication} provides a viable strategy for proving instances of the $S$-Selmer Section Conjecture, we verify the refined Kim's Conjecture in the example of the thrice-punctured line over~$\bZ[1/2]$ for all odd primes~$p$, which is the main computational result in this paper.

\begin{thmABC}\label{thm:main_kim}
	The refined Kim's Conjecture~\ref{conj:kim} holds for~$S=\{2\}$, $\cY=\bP^1_{\bZ[1/2]}\smallsetminus\{0,1,\infty\}$ and all odd primes~$p$.
\end{thmABC}

Consequently, we obtain

\begin{thmABC}\label{thm:main_s-selmer_sc}
	Conjecture~\ref{conj:s-selmer_sc} holds for~$K=\bQ$, $S=\{2\}$ and~$\cY=\bP^1_{\bZ[1/2]}\smallsetminus\{0,1,\infty\}$.
\end{thmABC}

\begin{rem}
	Theorem~\ref{thm:main_kim} is, to the best of our knowledge, the first example of a hyperbolic curve~$\cY/\bZ_S$ with~$\cY(\bA_{\bQ,S}^f)\neq\emptyset$ for which the refined Kim's Conjecture can be verified for infinitely many values of~$p$. (For curves with $\cY(\bA_{\bQ,S}^f)=\emptyset$, the refined Kim's Conjecture holds for more or less trivial reasons, see \cite[Remark~2.8]{BBKLMQSX}.) For this it is essential that we work with the refined Chabauty--Kim method \cite{BD:refined}; we were unable to prove the analogue of \Cref{thm:main_kim} for the original Kim's Conjecture. This illustrates some of the advantages of working with the refined version compared to the original method. 
	
	On the other hand, the  particular case of Conjecture~\ref{conj:s-selmer_sc} proved in Theorem~\ref{thm:main_s-selmer_sc} is already known due to work of Jakob Stix \cite[Corollary~6]{stix:birationalSC} (for $K=\bQ$, $\cY=\bP^1_{\bZ_S}\smallsetminus\{0,1,\infty\}$, and any~$S$). However, our strategy is completely different to Stix's, and demonstrates the applicability of the Chabauty--Kim method to questions related to the Section Conjecture.
\end{rem}

There are a couple of other cases where one can show a Kim-like conjecture for infinitely many primes~$p$. The first involves the unrefined Chabauty--Kim loci $\cY(\bZ_p)_{S,\infty}$ for an affine hyperbolic curve, as originally defined in \cite{kim:motivic,kim:albanese}. These are a priori larger than the refined loci $\cY(\bZ_p)_{S,\infty}^{\min}$ but still conjectured to be equal to the set of $S$-integral points $\cY(\bZ_S)$.

\begin{thmABC}\label{thm:main_unrefined}
	The original (i.e.\ unrefined) Kim's Conjecture holds for~$S=\emptyset$, $\cY=\bP^1_{\bZ}\smallsetminus\{0,1,\infty\}$ and all primes~$p$.
\end{thmABC}

We can also show a Kim-like conjecture for certain projective hyperbolic curves.

\begin{thmABC}\label{thm:main_chabauty}
	Let~$X/\bQ$ be a smooth projective curve of genus~$\geq2$ with a rational point, and suppose that the Jacobian of~$X$ has a factor of dimension~$\geq2$ with Mordell--Weil rank~$0$ and finite Tate--Shafarevich group. Then there is a finite closed subscheme~$Z\subset X$ such that
	\begin{equation}\label{eq:ck_locus_in_finite_subscheme}
	X(\bQ_p)_\infty \subseteq Z(\bQ_p)
	\end{equation}
	for all primes~$p$ of good reduction for~$X$.
\end{thmABC}

\begin{rem}
	The finite closed subscheme $Z$ of $X$ in \Cref{thm:main_chabauty} is without loss of generality reduced, thus a finite union of closed points. If the conclusion of \Cref{thm:main_chabauty} holds with $Z$ consisting only of rational points, then the refined Kim's Conjecture~\ref{conj:kim} holds for~$X$ and for all primes~$p$ of good reduction. Verifying this in a particular case is likely to be quite subtle.
\end{rem}

In particular, Theorems~\ref{thm:main_unrefined} and \ref{thm:main_chabauty} give additional examples where we prove the $S$-Selmer Section Conjecture using our strategy. (In fact, the containment~\eqref{eq:ck_locus_in_finite_subscheme} is already enough to conclude the Selmer Section Conjecture.) Again, neither of these cases of the $S$-Selmer Section Conjecture are new: for $\cY=\bP^1_{\bZ}\smallsetminus\{0,1,\infty\}$, the $S$-Selmer section set is empty since~$\cY(\bZ_2)=\emptyset$, while in the setting of Theorem~\ref{thm:main_chabauty}, the Selmer Section Conjecture is essentially a theorem of Stoll \cite[Theorem~8.6]{stoll:finite_descent}.
\smallskip

Let us now say a little about the proof of Theorem~\ref{thm:main_kim}. The original unrefined Chabauty--Kim method for the thrice-punctured line has been studied extensively in work of Ishai Dan-Cohen, Stefan Wewers and David Corwin \cite{DCW:explicitCK,DCW:mixedtate1,DC:mixedtate2,CDC:polylog1,CDC:polylog2}. They give some explicit Coleman functions vanishing on $\dY(\bZ[1/2])$, but it is far from clear whether these should cut out exactly~$\dY(\bZ[1/2])$. What we do in this paper is to carry out the refined versions of the computations of \cite{CDC:polylog1}, obtaining the following description of the refined Chabauty--Kim locus.

\begin{prop}
	For~$S=\{2\}$, $\cY=\bP^1_{\bZ[1/2]}\smallsetminus\{0,1,\infty\}$ and~$p$ any odd prime, the refined Chabauty--Kim locus $\dY(\bZ_p)_{\{2\},\infty}^{\min}$ is contained in the union of the $S_3$-translates of the locus in $\cY(\bZ_p)$ defined by the equations
	\[
	\log(z)=0 \quad\text{and}\quad \Li_k(z)=0 \quad\text{for $k\geq2$ even,}
	\]
	where $\log$ and~$\Li_k$ are the $p$-adic (poly)logarithms and the action of~$S_3$ on $\cY$ is the usual one permuting the three cusps $\{0,1,\infty\}$.
\end{prop}

So Theorem~\ref{thm:main_kim} follows once we verify that $\log$ and $\Li_{p-3}$ have no common zeroes in~$\cY(\bZ_p)$ other than~$z=-1$ for~$p\geq5$, and that the only zero of~$\log$ in~$\cY(\bZ_3)$ is $z=-1$. The same argument will also imply \Cref{thm:main_unrefined}.

\begin{rem}
	When the refined Kim's Conjecture holds, a natural question is what depth in the fundamental group one needs to go to in order to cut out exactly the $S$-integral points. Our proof of Theorem~\ref{thm:main_kim} establishes that the refined Kim's Conjecture already holds in depth~$n=p-3$ for~$p\geq5$, which notably depends on the prime~$p$. We strongly believe this to be an artefact of the proof: in fact, in \cite{BBKLMQSX} we conjectured that~$n=2$ is enough in this case, and verified this computationally for~$3\leq p\leq10^5$.
\end{rem}

\begin{rem}
	The computations of Corwin and Dan-Cohen \cite{CDC:polylog1} which we use in our proof of Theorem~\ref{thm:main_kim} use a rather different foundation of Chabauty--Kim theory than the original papers of Kim, using the $\bQ$-pro-unipotent motivic fundamental groupoid of $\bP^1\smallsetminus\{0,1,\infty\}$ as constructed by Deligne--Goncharov \cite{deligne-goncharov} in place of the $\bQ_p$-pro-unipotent \'etale fundamental groupoid. So in order to apply their results, we need to show that the two approaches are equivalent. This ultimately amounts to what seems to be a folklore result, that the \'etale realisation functor from mixed Tate motives to mixed Tate Galois representations is an equivalence ``after tensoring with~$\bQ_p$''. Here is a precise statement, which we will prove in the Appendix~\ref{sec:appendix}.
	
	\begin{thm}\label{thm:tannaka_group_comparison}
		Let~$K$ be a number field and~$S$ a finite set of primes of~$\cO_K$. Let~$\MT(\cO_{K,S},\bQ)$ denote the Tannakian category of mixed Tate motives over $\cO_{K,S}$ with coefficients in~$\bQ$ \cite[(1.7)]{deligne-goncharov}, and for a rational prime $p$ not divisible by any element of~$S$, let $\Rep_{\bQ_p}^{\MT,S}(G_K)$ denote the Tannakian category of mixed Tate Galois representations which are unramified outside $S\cup\{\mathfrak p\mid p\}$ and crystalline at all $\mathfrak p\mid p$. Let~$G_{K,S}^{\MT}$ and~$G_{K,S}^{\MTR}$ denote the Tannaka groups of $\MT(\cO_{K,S},\bQ)$ and $\Rep_{\bQ_p}^{\MT,S}(G_K)$ at their canonical fibre functors, respectively. Then the \'etale realisation functor
		\[
		\rho_\et\colon \MT(\cO_{K,S},\bQ) \to \Rep_{\bQ_p}^{\MT,S}(G_K)
		\]
		induces an isomorphism
		\[
		G_{K,S}^{\MTR} \xrightarrow\sim G_{K,S,\bQ_p}^{\MT}
		\]
		of affine group schemes over~$\bQ_p$.
	\end{thm}
\end{rem}

\subsection{Comparing the conjectures}

Let us say a little about our proof of Theorem~\ref{thm:main_conjecture_implication}, giving the relationship between the $S$-Selmer Section Conjecture and the refined Kim's Conjecture. We will relate the two conjectures by relating them both to a third conjecture in obstruction theory. If~$\cY/\cO_{K,S}$ is a regular $S$-integral model of a hyperbolic curve~$Y/K$, then the finite descent obstruction cuts out a \emph{finite descent locus}
\[
\cY(\bA_{K,S})_\bullet^{\fcov} \subseteq \cY(\bA_{K,S})_\bullet
\]
inside the modified $S$-adelic points $\cY(\bA_{K,S})_\bullet$, containing the $S$-integral points $\cY(\cO_{K,S})$ (see \S\ref{ss:finite_descent} for precise definitions). According to a conjecture of Stoll, this inclusion should be an equality.

\begin{conj}[Sufficiency of finite descent obstruction to strong approximation]\label{conj:sufficiency_finite_descent}
Let~$Y/K$ be a hyperbolic curve over a number field~$K$ and let~$\cY/\cO_{K,S}$ be a regular $S$-integral model of~$Y$. Then
\[
\cY(\cO_{K,S}) = \cY(\bA_{K,S})_\bullet^{\fcov} \,.
\]
\end{conj}

Conjecture~\ref{conj:sufficiency_finite_descent} is equivalent to the $S$-Selmer Section Conjecture~\ref{conj:s-selmer_sc}. More generally we have:

\begin{thmABC}[see \Cref{thm:image_of_localisation}]
	\label{thm:selmer-eq-fin-desc}
	There is a canonical bijection
	\[ \Sec(Y/K)^{\lgeom} \cong Y(\bA_K)_\bullet^{\fcov}. \]
\end{thmABC}

\Cref{thm:selmer-eq-fin-desc} is not new. As we will explain in \S\ref{ss:finite_descent}, work of Harari--Stix shows that the image of a certain localisation map 
\[
\loc\colon \Sec(\cY/\cO_{K,S})^{\lgeom} \to \cY(\bA_{K,S})_\bullet
\]
equals the finite descent locus, while a recent theorem of Porowski \cite{porowski:local_conjugacy} shows the injectivity. \Cref{thm:selmer-eq-fin-desc} makes very concrete the connection between Galois sections and finite descent obstructions, cf.~Remark~8.9 of \cite{stoll:finite_descent}.

Using the equivalence of Conjectures~\ref{conj:s-selmer_sc} and \ref{conj:sufficiency_finite_descent} afforded by \Cref{thm:selmer-eq-fin-desc}, our proof of Theorem~\ref{thm:main_conjecture_implication} then consists of two steps. First, we show that the projection of $\cY(\bA_{\bQ,S})_\bullet^\fcov$ on~$\cY(\bZ_p)$ is contained in the refined Chabauty--Kim locus $\cY(\bZ_p)_{S,\infty}^{\min}$ for all~$p\notin S$. So if Conjecture~\ref{conj:kim} held for all primes~$p$ in a set~$\fP$ of Dirichlet density~$1$, then we would deduce that the projection of the finite descent locus on~$\prod_{p\in\fP}\cY(\bZ_p)$ is contained in
\[
\prod_{p\in\fP}\cY(\bZ_S) \,.
\]
By Porowski again, the projection $\cY(\bA_{\bQ,S})_\bullet^\fcov \to \prod_{p\in\fP}\cY(\bZ_S)$ is injective, so the remaining issue is to show that the image is contained in $\cY(\bZ_S)$, embedded diagonally inside~$\prod_{p\in\fP}\cY(\bZ_S)$. This is accomplished by a theorem due to Stoll (in the projective case; we extend Stoll's argument to the general case in Theorem~\ref{thm:diagonal_theorem}) which we dub the \emph{theorem of the diagonal}.

\subsection{Overview of sections}

We begin in Section~\ref{sec:finite_descent} by proving Theorem~\ref{thm:main_conjecture_implication}: that the refined Kim's Conjecture implies the $S$-Selmer Section Conjecture. We then take a detour in Section~\ref{sec: comparison theorem} to explain how the Chabauty--Kim method as formulated in \cite{kim:motivic,kim:albanese,BDCKW,BD:refined} is related to the motivic Chabauty--Kim method used in \cite{DCW:explicitCK,DCW:mixedtate1,DC:mixedtate2,CDC:polylog1,CDC:polylog2}. This is essentially a very careful checking that the two methods agree. The most technical motivic arguments in this section are hived off into Appendix~\ref{sec:appendix}.

Section~\ref{sec: refined} recalls the setup of the refined Chabauty--Kim method, and lays the groundwork for the computation of the refined Chabauty--Kim locus for the thrice-punctured line over~$\bZ[1/2]$, which we carry out in Section~\ref{sec:calculations}. \Cref{sec:calculations} contains the proofs of Theorems~\ref{thm:main_kim}, \ref{thm:main_unrefined} and~\ref{thm:main_chabauty}.

\subsection*{Acknowledgements}

We would like to express our gratitude to Dustin Clausen, David Corwin, Ishai Dan-Cohen, Brendan Creutz, Elden Elmanto, Annette Huber-Klawitter, Bruno Kahn, Marc Levine, Chris\-tophe Soul\'e and Jakob Stix for taking the time to answer our questions during the writing of this paper.

L.A.B.\ is supported by the Simons Collaboration on Arithmetic Geometry, Number Theory, and Computation under grant number 550031. M.L.\ and T.K.\ were supported in part by the LOEWE research unit USAG and acknowledge funding by the Deutsche Forschungsgemeinschaft  (DFG, German Research Foundation) through the Collaborative Research Centre TRR 326 GAUS, project number 444845124. M.L.\ was also supported by an NWO Vidi grant and an MPIM postdoc scholarship.

%% file: descent.tex
\section{Selmer sections, finite descent and Chabauty--Kim}\label{sec:finite_descent}

We begin by discussing the relation between the $S$-Selmer Section Conjecture~\ref{conj:s-selmer_sc}, Stoll's Conjecture on sufficiency of the finite descent obstruction to strong approximation (Conjecture~\ref{conj:sufficiency_finite_descent}) and the refined Kim's Conjecture~\ref{conj:kim}.

We will adopt the following notation: for a geometrically connected variety~$V$ over a field~$F$, we write~$\Sec(V/F)$ for the set of splittings of the fundamental exact sequence
\[
1 \to \pi^{\et}_1(V_{\Fbar}) \to \pi^{\et}_1(V) \to G_F \to 1
\]
up to $\pi^{\et}_1(V_{\Fbar})$-conjugacy. We also write $\Sec(V/F)_{\ab}$ for the set of sections of the geometrically abelian fundamental exact sequence
\[
1 \to \pi_1^{\ab}(V_{\Fbar}) \to \pi_1^{\gab}(V) \to G_F \to 1
\]
given by pushing out the fundamental exact sequence along the abelianisation map~$\pi^{\et}_1(V_{\Fbar})\to\pi_1^{\ab}(V_{\Fbar})\coloneqq\pi_1^{\et}(V_{\Fbar})^{\ab}$. If~$x\in V(F)$ is an $F$-rational point, we write~$\kappa(x)\in\Sec(V/F)$ for the induced section, and~$\kappa_{\ab}(x)\in\Sec(V/F)_{\ab}$ for the abelian section obtained via pushout.
\smallskip

Throughout this section, $K$ will be a number field and~$Y/K$ a hyperbolic curve (or sometimes a curve of non-positive Euler characteristic).\footnote{Our convention is that curves are required to be smooth and geometrically connected, but not necessarily projective.} Our principal object of study will be the set of \emph{Selmer sections} of~$Y/K$, defined as follows.

\begin{defn}
	Let~$Y$ be a hyperbolic curve over a number field~$K$. A section~$s\in\Sec(Y/K)$ (resp.\ abelian section~$s_{\ab}\in\Sec(Y/K)_{\ab}$) is called \emph{Selmer} just when there exists an adelic point~$(y_v)_v\in Y(\bA_K)$ such that
	\[
	s|_{G_v} = \kappa(y_v) \quad \text{resp.} \quad s_{\ab}|_{G_v} = \kappa_{\ab}(y_v)
	\]
	is the (abelian) section coming from~$y_v$ for all places~$v$. We write $\Sec(Y/K)^{\lgeom}\subseteq\Sec(Y/K)$ and $\Sec(Y/K)_{\ab}^{\lgeom}\subseteq\Sec(Y/K)_{\ab}$ for the set of Selmer sections and abelian Selmer sections, respectively.
\end{defn}

Note that the $S$-Selmer section set $\Sec(\cY/\cO_K)^{\Sel}$ of Definition~\ref{def:s-selmer} is exactly the subset of~$\Sec(Y/K)^{\Sel}$ consisting of those Selmer sections for which the adelic point~$(y_v)_v$ can be chosen to be an $S$-adelic point on the model~$\cY$.

It turns out that the adelic point~$(y_v)_v$ associated to a Selmer section is unique, up to a certain ambiguity at infinite places. To explain this, we need the following injectivity result for the local section maps.

\begin{lem}\label{lem:local_injectivity}
	Let~$K_v$ be a local field of characteristic~$0$ and let~$Y/K_v$ be a hyperbolic curve.
	\begin{enumerate}
		\item\label{lempart:local_injectivity_nonarch} If~$K_v$ is non-archimedean, then the section map (resp.\ abelian section map)
		\[
		Y(K_v) \hookrightarrow \Sec(Y/K_v) \quad \text{resp.} \quad Y(K_v) \hookrightarrow \Sec(Y/K_v)_{\ab}
		\]
		is injective.
		\item\label{lempart:local_injectivity_real} If~$K_v=\bR$, then the (abelian) section map factors through an injective map
		\[\pi_0(Y(\bR)) \hookrightarrow \Sec(Y/\bR) \quad \text{resp.} \quad \pi_0(Y(\bR)) \hookrightarrow \Sec(Y/\bR)_{\ab} \,.\]
		\item\label{lempart:local_injectivity_complex} If~$K_v=\bC$, then the (abelian) section map 
		\[ Y(\bC) \to \Sec(Y/\bC)  \quad \text{resp.} \quad Y(\bC) \to \Sec(Y/\bC)_{\ab} \]
		is constant.
	\end{enumerate}
\end{lem}

\begin{proof}
	Part~\eqref{lempart:local_injectivity_complex} is trivial; for part~\eqref{lempart:local_injectivity_nonarch} it suffices to prove that the abelian section map is injective, which is \cite[Proposition~73]{stix:evidence_SC}. For part~\eqref{lempart:local_injectivity_real}, the section map factors through~$\pi_0(Y(\bR))$ e.g.\ by the discussion in \cite[\S16.1]{stix:evidence_SC}, so it suffices to prove injectivity of the abelian section map
	\[
	\pi_0(Y(\bR)) \to \Sec(Y/\bR)_{\ab} \,.
	\]
	This follows from a theorem of Wickelgren \cite[Theorem~1.1]{wickelgren:real-SC}. (The injectivity of the whole section map $\pi_0(Y(\bR)) \to \Sec(Y/\bR)$ was previously proved by Mochizuki \cite{mochizuki:topics}.)
\end{proof}

Now let us write
\[
Y(\bA_K)_{\bullet} = \prod_{v\nmid\infty}^{\prime}Y(K_v)\times \prod_{\text{$v$ real}} \pi_0(Y(K_v))
\]
for the set of \emph{modified adelic points} of~$Y$~\cite{stoll:finite_descent,stoll:finite_descent_errata}. It follows from Lemma~\ref{lem:local_injectivity} that the adelic point~$(y_v)_v\in Y(\bA_K)$ associated to a Selmer section or abelian Selmer section is unique when viewed as an element of~$Y(\bA_K)_\bullet$. Thus the assignment $s\mapsto (y_v)_v$ defines functions
\[
\loc\colon \Sec(Y/K)^{\lgeom} \to Y(\bA_K)_\bullet \quad\text{and}\quad \loc\colon \Sec(Y/K)_{\ab}^{\lgeom} \to Y(\bA_K)_\bullet \,,
\]
which we call the \emph{localisation map} for (abelian) Selmer sections. We will frequently use the following result regarding these localisation maps, a consequence of a remarkable new theorem due to Porowski.
\begin{thm}\label{thm:injectivity}
	Let~$Y/K$ be a curve of non-positive Euler characteristic, and let~$\fP$ be a set of finite places of~$K$ of Dirichlet density~$1$. Then the localisation maps
	\[
	\loc\colon \Sec(Y/K)^{\lgeom} \to \prod_{v\in\fP}Y(K_v) \quad\text{and}\quad \loc\colon \Sec(Y/K)_{\ab}^{\lgeom} \to \prod_{v\in\fP}Y(K_v)
	\]
	are injective.
\end{thm}
\begin{proof}
	We have, by definition, a commuting square
	\begin{center}
	\begin{tikzcd}
		\Sec(Y/K)^{\lgeom} \arrow[r,"\loc"]\arrow[d,hook] & \prod_{v\in\fP}Y(K_v) \arrow[d] \\
		\Sec(Y/K) \arrow[r] & \prod_{v\in\fP}\Sec(Y_{K_v}/K_v)
	\end{tikzcd}
	\end{center}
	in which the bottom horizontal map is given by restriction to the decomposition groups in~$\fP$. This map is injective by \cite[Theorem~1.1]{porowski:local_conjugacy}, and hence so is the localisation map on the top.
	
	For abelian sections, the same argument works, using \cite[Proposition~3.2]{saidi:injectivity_localisation} for the injectivity of the bottom horizontal map. (The statement in \cite{saidi:injectivity_localisation} requires that $\Sec(Y/K)\neq\emptyset$, but injectivity is vacuously true otherwise.) One can alternatively pass to a finite extension of~$K$, embed~$Y$ inside its semiabelian Jacobian, and use \cite[Theorem~1.1]{porowski:local_conjugacy} again.
\end{proof}

\subsection{Finite descent}
\label{ss:finite_descent}

The $S$-Selmer section set is closely related to the finite descent obstruction. If~$(P,\dG)$ is a pair consisting of a finite \'etale $K$-group scheme~$\dG$ and a (right) $\dG$-torsor~$P$ over~$Y$ in the \'etale topology, then we say that a modified adelic point $y\in Y(\bA_K)_{\bullet}$ \emph{survives $(P,\dG)$} just when there is a continuous Galois cocycle $\xi\colon G_K\to \dG(\Kbar)$ such that~$y$ lies in the image of $P_\xi(\bA_K)_{\bullet} \to Y(\bA_K)_{\bullet}$ where~$P_\xi$ denotes the $\xi$-twisted torsor \cite[Definition~5.2]{stoll:finite_descent}.

\begin{defn}[{\cite[Definition~5.4]{stoll:finite_descent}\footnote{Stoll only defines this set in the case that~$Y$ is projective, but the same definition carries over in general.}}]\label{def: descent obstruction}
	We say that $y \in Y(\bA_K)_{\bullet}$ survives the \emph{finite descent obstruction} if it survives every pair $(P,\dG)$. We say it survives the \emph{finite abelian descent obstruction} if it survives every pair $(P,\dG)$ where $\dG$ is abelian.
\end{defn}

The sets of modified adelic points surviving the various descent obstructions are denoted by
\[
Y(\bA_K)_{\bullet}^{\fcov} \subseteq Y(\bA_K)_{\bullet}^{\fab} \subseteq Y(\bA_K)_{\bullet} \,.
\]
The set of rational points $Y(K)$ is contained in each of these sets.

We will also need to consider a third obstruction locus\footnote{Our thanks to Brendan Creutz for pointing out this subtlety.}. We say that a pair~$(P,\dG)$ has \emph{abelian geometric monodromy} just when the image of the induced map~$\pi^{\et}_1(Y_{\Kbar})\to\dG(\Kbar)$ is abelian. We write~$Y(\bA_K)_\bullet^{\fgab}$ for the set of points~$y\in Y(\bA_K)_\bullet$ which survive every pair~$(P,\dG)$ with abelian geometric monodromy. ($Y(\bA_K)_\bullet^{\fgab}$ is called the \emph{finite geometrically abelian descent locus}). This is situated among the other descent loci as
\[
Y(K) \subseteq Y(\bA_K)_{\bullet}^{\fcov} \subseteq Y(\bA_K)_{\bullet}^{\fgab} \subseteq Y(\bA_K)_{\bullet}^{\fab} \subseteq Y(\bA_K)_{\bullet} \,.
\]
\begin{lem}
	\label{lem:fgab-vs-ab}
	If~$Y(K)\neq\emptyset$, then the inclusion $Y(\bA_K)_\bullet^{\fgab}\subseteq Y(\bA_K)_\bullet^{\fab}$ is an equality.
	\begin{proof}
		Suppose that~$y\in Y(\bA_K)_\bullet^{\fab}$ and that~$P=(P,\dG)$ has abelian geometric monodromy. Since~$Y$ has a rational point, \cite[Lemma~$5.5'$]{stoll:finite_descent_errata} implies that there is some twist~$P_\chi=(P_\chi,\dG_\chi)$ of~$P$ which has a connected component~$P_0$ which is geometrically connected. This means that~$P_0$ is a torsor under a subgroup-scheme $\dG_0\subseteq\dG_\chi$. Since~$(P,\dG)$ had abelian geometric monodromy, so too does~$(P_0,\dG_0)$; since~$P_0$ is geometrically connected, this implies that~$\dG_0$ is actually abelian. By assumption, $y$ survives~$(P_0,\dG_0)$, so it survives $(P_\chi,\dG_\chi)$ \cite[Lemma~5.3(1)]{stoll:finite_descent} and hence $(P,\dG)$ \cite[Lemma~5.3(5)]{stoll:finite_descent}. So we have shown that any~$y\in Y(\bA_K)_\bullet^{\fab}$ survives all pairs with abelian geometric monodromy, i.e.\ lies in~$Y(\bA_K)_\bullet^{\fgab}$.
	\end{proof}
\end{lem}

\begin{rem}
	We do not know of any examples where the inclusion $Y(\bA_K)_\bullet^{\fgab}\subseteq Y(\bA_K)_\bullet^{\fab}$ is strict.
\end{rem}

Harari and Stix showed that the finite descent obstruction is closely related to the Selmer section set. Combining their work with Porowski's, we obtain

\begin{thm}\label{thm:image_of_localisation}
	Let~$Y$ be a hyperbolic curve over a number field~$K$. Then:
	\begin{enumerate}
		\item \label{item:image_of_localisation_fcov}
		the localisation map~$\loc\colon\Sec(Y/K)^{\lgeom}\to Y(\bA_K)_\bullet$ is a bijection onto the finite descent locus~$Y(\bA_K)_\bullet^{\fcov}$; and
		\item \label{item:image_of_localisation_fab}
		the abelian localisation map~$\loc\colon\Sec(Y/K)_{\ab}^{\lgeom}\to Y(\bA_K)_\bullet$ is a bijection onto the finite geometrically abelian descent locus~$Y(\bA_K)_\bullet^{\fgab}$.
	\end{enumerate}
	\begin{proof}
		Given Theorem~\ref{thm:injectivity}, we only need to verify that the images of the localisation maps are the claimed descent loci. This is a particular case of the equivalence (i)$\Leftrightarrow$(iii) in \cite[Theorem~11]{harari-stix:descent} (cf.\ the cohomological formulation of the finite descent locus in \cite[\S5]{stoll:finite_descent}). For $Y(\bA_K)^{\fcov}_{\bullet}$ we apply the theorem with $\overline{U} = 1$, while for $Y(\bA_K)^{\fgab}_{\bullet}$ we take $\overline{U}$ to be the first derived subgroup of $\pi_1^{\et}(Y_{\overline{K}})$ and observe that $(P,\dG)$ has abelian geometric monodromy iff the restriction of the class $[P] \in \rH^1(Y, \dG)$ to an element of $\rH^1(\overline{U}, \dG(\overline{K})) = \Hom(\overline{U},\dG(\overline{K}))/\dG(\overline{K})$ is trivial.
	\end{proof}
\end{thm}

Theorem~\ref{thm:image_of_localisation} implies in particular that the $S$-Selmer Section Conjecture~\ref{conj:s-selmer_sc} is equivalent to the sufficiency of the finite descent obstruction to strong approximation~\ref{conj:sufficiency_finite_descent}.

\begin{rem}
	Let us explicitly describe the procedure for lifting an adelic point in the finite descent locus to a Selmer section used in \cite{harari-stix:descent}. Let~$Y/K$ be a hyperbolic curve. There is a bijection between sections of the fundamental exact sequence for~$Y$ up to $\pi_1(Y_\Kbar)$-conjugacy and $K$-forms of the universal covering\footnote{A $K$-form of the universal covering is a pro-finite \'etale covering $\tilde Y\to Y$ whose base change to~$\Kbar$ is a universal covering $\tilde Y_\Kbar\to Y_\Kbar$.} of~$Y_\Kbar$ up to isomorphism, see \cite[\S2.5]{stix:evidence_SC}. This bijection takes a section~$s$ to the pro-covering $\tilde Y\to Y$ corresponding to the closed subgroup $\im(s)\subseteq\pi_1(Y)$. The section comes from a $K$-rational point~$y$ if and only if~$y$ lifts to a $K$-rational point on~$\tilde Y$.
	
	Reformulated in this manner, lifting an adelic point~$(y_v)_v\in\cY(\bA_K)^{\fcov}_\bullet$ to a section amounts to finding a $K$-form of the universal covering~$\tilde Y\to Y$ for which each~$y_v$ lifts to a $K_v$-rational point on~$\tilde Y$. We explain how to do this under the simplifying assumption that $Y(K)\neq\emptyset$. This assumption ensures that there exists some $K$-form of the universal covering~$\tilde Y_0\to Y$. This can be written as an inverse limit of coverings $P_i\to Y$ which become Galois when base-changed to~$\Kbar$, and which are thus torsors under finite \'etale group schemes $\dG_i/K$. If~$(y_v)_v\in\cY(\bA_K)^{\fcov}_\bullet$ lies in the finite descent locus, then for each~$i$, the torsor $P_i\to Y$ has only finitely many twists such that~$(y_v)_v$ lifts to an adelic point of~$P_i$. Indeed, the set of all twists is parametrised by the set $\rH^1(G_K,\dG_i)$, while the fibre of $P_i\to Y$ over~$y_v$ is a $\dG_i$-torsor over~$\Spec(K_v)$, so is classified by an element $[P_{i,y_v}]\in\rH^1(G_v,\dG_i)$. The twists of~$P_i$ to which~$y_v$ lifts are exactly those elements of $\rH^1(G_K,\dG_i)$ whose image in~$\rH^1(G_v,\dG_i)$ is equal to~$[P_{i,y_v}]$. Because the map $\rH^1(G_K,\dG_i)\to\prod_v\rH^1(G_v,\dG_i)$ has finite fibres by Hermite--Minkowski, it follows as claimed that $P_i$ has only finitely many twists to which $(y_v)_v$ lifts.
	
	Let~$T_i$ denote the set of twists of~$P_i$ to which the point $(y_v)_v$ lifts. We have just seen that each~$T_i$ is finite, and it is non-empty by assumption that~$(y_v)_v$ lies in the finite descent locus. If~$P_i$ dominates~$P_j$, then each twist of~$P_i$ determines a twist of~$P_j$, and hence there is an induced function $T_i\to T_j$. Taken together, the~$T_i$ form a cofiltered system of non-empty finite sets, which implies that the inverse limit~$\varprojlim_iT_i$ is non-empty. In other words, there is a compatible system of twists~$P'_i$ of the~$P_i$ to which $(y_v)_v$ lifts.
	
	Now on the one hand, the $P'_i$ fit together into a cofiltered inverse system, and the pro-covering $\tilde Y=\varprojlim_iP'_i\to Y$ is a $K$-form of the universal covering of~$Y_\Kbar$. On the other hand, each point $y_v\in Y(K_v)$ lifts to a compatible sequence of $K_v$-points in the $P'_i$ (again because a cofiltered inverse limit of finite non-empty sets is non-empty). This means that the localisation of the section corresponding to $\tilde Y \to Y$ is the local section coming from~$y_v$.
\end{rem}

\subsubsection{Theorem of the diagonal}

We will need to use one key property of the finite abelian descent set in what follows, which in the projective case is due to Stoll.

\begin{thm}[cf.\ {\cite[Theorem~8.2]{stoll:finite_descent}}]\label{thm:diagonal_theorem}
	Let~$Y$ be a curve of non-positive Euler characteristic over a number field~$K$, and let~$Z\subset Y$ be a finite closed subscheme. Then the image of~$Z(\bA_K)$ in~$Y(\bA_K)_\bullet$ meets the finite abelian descent locus $Y(\bA_K)_\bullet^{\fab}$ in~$Z(K)$. More generally, if $\fP$ is a set of places of $K$ of Dirichlet density~$1$ and $y\in Y(\bA_K)_\bullet^{\fab}$ is such that $y_v\in Z(K_v)$ for all $v \in \fP$, then~$y\in Z(K)$.
\end{thm}

\begin{rem}
	Since the finite descent set~$Y(\bA_K)_\bullet^{\fcov}$ is contained in~$Y(\bA_K)_\bullet^{\fab}$, the analogous statement to Theorem~\ref{thm:diagonal_theorem} holds \emph{a posteriori} for~$Y(\bA_K)_\bullet^{\fcov}$. This latter statement is all we will actually need here, but we prove the stronger statement since it may be of use in obstruction theory.
\end{rem}

We call Theorem~\ref{thm:diagonal_theorem} the \emph{theorem of the diagonal}, since it says that a point in $Z(\bA_K)=\prod_vZ(K_v)$ which is unobstructed by the finite descent obstruction in~$Y$ lies diagonally in the product (with the usual caveats at infinite places).

In the proof, we will frequently use the fact that we can safely ignore the components of elements of $Y(\bA_K)_\bullet^\fab$ at a sparse set of places~$v$.

\begin{lem}\label{lem:multiplicity_one}
	Let~$Y/K$ be a curve of non-positive Euler characteristic, and let~$\fP$ be a set of finite places of~$K$ of Dirichlet density~$1$. Then the projection map
	\[
	Y(\bA_K)_\bullet^{\fab} \to \prod_{v\in\fP}Y(K_v)
	\]
	is injective.
\end{lem}

\begin{proof}
	If~$Y(K)\neq\emptyset$, then $Y(\bA_K)_\bullet^{\fab}=Y(\bA_K)_\bullet^{\fgab}$, so the result is just the second part of Theorem~\ref{thm:injectivity}, using the second part of Theorem~\ref{thm:image_of_localisation}.
	
	In the general case, consider any two elements~$y,y'\in Y(\bA_K)_\bullet^{\fab}$ such that~$y_v=y'_v$ for all~$v\in\fP$. If~$L/K$ is a finite extension for which $Y(L)\neq\emptyset$, let us write~$y_L$ and~$y'_L$ for the images of~$y$ and~$y'$ in~$Y(\bA_L)_\bullet$. According to \cite[Proposition~5.16]{stoll:finite_descent}\footnote{Stoll only states~\cite[Proposition~5.16]{stoll:finite_descent} for projective~$Y$, but the same proof works in the affine case.}, the points~$y_L$ and~$y'_L$ lie in the finite abelian descent locus $Y(\bA_L)_\bullet^{\fab}$, and $y_{L,w}=y'_{L,w}$ for all places~$w$ of~$L$ in the set~$\fP_L$ of places of~$L$ lying over places in~$\fP$. The set~$\fP_L$ also has density~$1$.
	
	Hence, if~$Y(L)\neq\emptyset$, then we deduce that~$y_L=y'_L$. In other words, $y_{L,w}=y'_{L,w}$ for all finite places~$w$ of~$L$, and~$y_{L,w}$ and~$y'_{L,w}$ lie in the same connected component of~$Y(L_w)$ for all real places~$w$ of~$L$. This implies immediately that~$y_v=y'_v$ for all finite places~$v$ of~$K$, but the argument for real places~$v$ of~$K$ is a little more delicate, since it could happen that every place above~$v$ in~$L$ is complex.
	
	We deal with this by a judicious choice of field~$L$. Fix a real place~$v$ of~$K$. If~$Y(K_v)=\emptyset$ then $Y(\bA_K)_\bullet=\emptyset$ and there is nothing to prove. Otherwise, we may choose a morphism $f\colon Y\to\bP^1$ which is \'etale in a neighbourhood of some $K_v$-point of~$Y$. The image $f(Y(K_v))$ then contains an open interval~$I$ inside $\bP^1(K_v)=\bP^1(\bR)$. Choosing a point~$z$ inside~$I\cap\bP^1(K)$, the scheme-theoretic fibre of~$f$ over~$z$ is a finite $K$-scheme with a $K_v$-point, so has a point defined over a field extension~$L/K$ possessing a real place~$w$ over~$v$. So we deduce that~$y_v$ and $y'_v$ lie in the same component of~$Y(K_v)=Y(L_w)$. Since we know this for all real~$v$, we have shown that~$y=y'$ as elements of~$Y(\bA_K)_\bullet$.
\end{proof}

\begin{cor}\label{cor:injectivity_descent_loci}
	Let~$Y\hookrightarrow X$ be a locally closed immersion of a curve~$Y/K$ of non-positive Euler characteristic in a $K$-variety~$X$. Then the induced map
	\[
	Y(\bA_K)_\bullet^{\fab} \hookrightarrow X(\bA_K)_\bullet^{\fab}
	\]
	is injective (even though the map $Y(\bA_K)_\bullet\to X(\bA_K)_\bullet$ may not be).
\end{cor}

\begin{proof}
	Follows from the commuting square
	\[
	\begin{gathered}[b] 
		\begin{tikzcd}
			Y(\bA_K)_\bullet^{\fab} \arrow[r]\arrow[d,hook] & X(\bA_K)_\bullet^{\fab} \arrow[d] \\
			\prod_{v\nmid\infty}Y(K_v) \arrow[r,hook] & \prod_{v\nmid\infty}X(K_v) \,. & 
		\end{tikzcd} \\[-\dp\strutbox]
	\end{gathered}\qedhere
	\]
\end{proof}

We now begin the proof of \Cref{thm:diagonal_theorem} with some preliminary reductions. Firstly, we may enlarge the curve $Y$ by filling in punctures:

\begin{lem}\label{lem: theorem of the diagonal enlarge curve}
	Let $X$ be a curve over $K$ containing $Y$ as an open subset. If \Cref{thm:diagonal_theorem} holds for $X$, then it also holds for $Y$.
\end{lem}

\begin{proof}
	Let~$y\in Y(\bA_K)_\bullet^{\fab}$ be such that $y_v\in Z(K_v)$ for all~$v\in\fP$. By Corollary~\ref{cor:injectivity_descent_loci} we have
	\[
	Y(\bA_K)_\bullet^{\fab} \subseteq X(\bA_K)_\bullet^{\fab} \,,
	\]
	so we can regard~$y$ as an element of~$X(\bA_K)_\bullet^{\fab}$ without ambiguity. If Theorem~\ref{thm:diagonal_theorem} holds for~$X$ then we deduce that $y$ lies in $Z(K)$ and are done.
\end{proof}

Secondly, we are free to enlarge the base field:

\begin{lem}
	\label{lem: theorem of the diagonal scalar extension}
	Let $L/K$ be a finite extension. If \Cref{thm:diagonal_theorem} holds for $Y_L$, then it also holds for $Y$.
\end{lem}

\begin{proof}
	Let~$y\in Y(\bA_K)_\bullet^{\fab}$ be such that $y_v\in Z(K_v)$ for all~$v\in\fP$. Applying Corollary~\ref{cor:injectivity_descent_loci} to the closed embedding $Y\hookrightarrow\Res^L_KY_L$ and using \cite[Proposition~5.15]{stoll:finite_descent} (cf.\ the correction in \cite[\S2(3)]{stoll:finite_descent_errata}) we have
	\[
	Y(\bA_K)_\bullet^{\fab} \subseteq Y(\bA_L)_\bullet^{\fab} \,.
	\]
	The set~$\fP_L$ of places of~$L$ lying over places in~$\fP$ is again of Dirichlet density~$1$. If Theorem~\ref{thm:diagonal_theorem} holds for~$Y_L$, we deduce that~$y$ lies in
	\[
	Z(L)\cap Y(\bA_K)_\bullet^{\fab} = Z(K)
	\]
	and are done.
\end{proof}

\begin{proof}[Proof of Theorem~\ref{thm:diagonal_theorem}]
	Using the two preceding reduction steps, we see that it suffices to deal with the cases that $Y=\bG_m$ or~$Y$ is smooth projective of genus at least~$1$. The latter case was dealt with in \cite[Theorem~8.2]{stoll:finite_descent}; we deal with the former by following the strategy of \cite[Proposition~3.6]{stoll:finite_descent}. For this, we need a preliminary cohomological calculation.
	
	\begin{lem}\label{lem:bounded_torsion}
		Let~$m$ denote the number of roots of unity in~$K$. Then for any positive integer~$N$, the cohomology group $\rH^1(G_{K(\mu_N)|K},\mu_N)$ is $2m$-torsion.
		\begin{proof}
			By the Chinese Remainder Theorem, it suffices to show that the cohomology group $\rH^1(G_{K(\mu_N)|K},\mu_{p^r})$ is $2m$-torsion for all prime powers~$p^r$ dividing~$N$. We have the inflation-restriction exact sequence
			\[
			0 \to \rH^1(G_{K(\mu_{p^r})|K},\mu_{p^r}) \to \rH^1(G_{K(\mu_N)|K},\mu_{p^r}) \to \rH^1(G_{K(\mu_N)|K(\mu_{p^r})},\mu_{p^r})^{G_{K(\mu_N)|K}},
			\]
			in which the left-hand group is $2$-torsion by \cite[Proposition~9.1.6]{nsw}. On the other hand, the right-hand group is $\Hom(G_{K(\mu_N)|K(\mu_{p^r})},\mu_{p^r}(K))$, which is $m$-torsion since~$\mu_{p^r}(K)$ is. Thus the middle group is $2m$-torsion.
		\end{proof}
	\end{lem}
	
	Now we complete the proof of Theorem~\ref{thm:diagonal_theorem} in the case~$Y=\bG_m$. We may suppose, enlarging~$K$ if necessary using \Cref{lem: theorem of the diagonal scalar extension}, that~$Z^{\redu}$ consists of a finite set of~$K$-rational points. Let~$y\in\bG_m(\bA_K)_\bullet^{\fab}$ be such that $y_v\in Z(K_v)=Z(K)$ for all~$v\in\fP$. Since~$y$ lies in the finite abelian descent set, by \Cref{lem:fgab-vs-ab} and \Cref{thm:image_of_localisation}\,(2) we have that the image of~$y$ in $\prod_v \rH^1(G_v, \hat\bZ(1))$ is the image of an element $\hat y=(y_N)_N\in\rH^1(G_K,\hat\bZ(1))=\varprojlim_N(K^\times/K^{\times N})$ under the localisation map
	\[
	\varprojlim_N(K^\times/K^{\times N}) \to \prod_v\varprojlim_N(K_v^\times/K_v^{\times N}) \,.
	\]
	Here we are using that $\Sec(\bG_m/K)_{\ab} = \rH^1(G_K, \hat\bZ(1))$.
	
	Now fix a positive integer~$N$, and consider the fields $K_z:=K(\mu_N,\sqrt[N]{y_N/z})$ for~$z\in Z(K)$, where~$y_N\in K^\times/K^{\times N}$ is the image of~$\hat y$. If~$v\in\fP$ is a place which splits completely in~$K(\mu_N)$, then it splits completely in some~$K_z$, since $y_N$ is congruent to~$y_v\in Z(K_v)=Z(K)$ modulo~$K_v^{\times N}$. By the Chebotarev density theorem, the density of places which split completely in a fixed~$K_z$ is $\frac1{[K_z:K]}$, while the density of primes which split completely in~$K(\mu_N)$ is~$\frac1{[K(\mu_N):K]}$, from which we deduce the inequality
	\[
	\sum_{z\in Z(K)}\frac1{[K_z:K]} \geq \frac1{[K(\mu_N):K]} \,.
	\] 
	Multiplying the inequality by $[K(\mu_N) : K]$, we obtain
	\[ \sum_{z\in Z(K)}\frac1{[K_z:K(\mu_N)]} \geq 1. \]
	So there is some~$z\in Z(K)$ such that~$[K_z:K(\mu_N)]\leq\# Z(K)$. In other words, there is some $z\in Z(K)$ such that~$y_N/z$ has order at most~$\# Z(K)$ in~$K(\mu_N)^\times/K(\mu_N)^{\times N}$. So $(y_N/z)^{\# Z(K)!}$ is an $N$th power in~$K(\mu_N)$. By \Cref{lem:bounded_torsion} and inflation-restriction, the kernel of \[K^\times/K^{\times N}\to K(\mu_N)^\times/K(\mu_N)^{\times N}\] is $2m$-torsion, where $m$ is the number of roots of unity in~$K$. We deduce that $(y_N/z)^M$ is an $N$th power in~$K$, where~$M \coloneqq 2m\cdot\# Z(K)!$.
	
	So we have proved that for every positive integer~$N$ there exists some~$z\in Z(K)$ such that~$(y_N/z)^M$ is an $N$th power in~$K$, where~$M$ does not depend on~$N$. If $(y_N/z)^M \in K^{\times N}$ then the same is true for any factor of~$N$, thus for each~$z \in Z(K)$ either there is a largest integer~$A$ such that $(y_{A!}/z)^M \in K^{\times A!}$ or we have $(y_{A!}/z)^M \in K^{\times A!}$ for all~$A$. 
	Since the set~$Z(K)$ is finite, the latter must hold for at least one~$z \in Z(K)$, so there is a single~$z$ such that $(y_N/z)^M \in K^{\times N}$ for all~$N$. We deduce that $(\hat y/z)^M=1$. But the torsion subgroup of $\varprojlim_N(K^\times/K^{\times N})$ is the group~$\mu_\infty(K)$ of roots of unity in~$K$, so we deduce that~$\hat y\in\mu_\infty(K)\cdot Z(K)$ and in particular $\hat y \in K^\times$. Since we have $\hat y \equiv y_v \bmod K_v^{\times N}$ for all $N$ and all $v$, the points~$y_v$ must all be equal to~$\hat y$ and so in fact $\hat y\in Z(K)$ and we are done.
\end{proof}

\subsection{Refined Chabauty--Kim}

For the remainder of this section, we restrict to the case~$K=\bQ$, and assume moreover that the smooth compactification of our hyperbolic curve~$Y$ has a rational point. We recall in outline the refined Chabauty--Kim method of \cite{BD:refined}.

Let~$\pi_1^{\et,\bQ_p}(Y_{\Qbar},b)$ be the $\bQ_p$-pro-unipotent \'etale fundamental group of~$Y$ based at either a rational basepoint or rational tangential basepoint~$b$, and let~$\pi_1^{\et,\bQ_p}(Y_{\Qbar},b)\twoheadrightarrow\Pi$ be a $G_\bQ$-equivariant quotient. For every rational point~$y\in Y(\bQ)$, we write~${}_y\Pi_b$ for the pushout of the $\bQ_p$-pro-unipotent \'etale path torsor $\pi_1^{\et,\bQ_p}(Y_{\Qbar};b,y)$ along the map $\pi_1^{\et,\bQ_p}(Y_{\Qbar},b)\to\Pi$. The map
\[
j\colon Y(\bQ) \to \rH^1(G_\bQ,\Pi(\bQ_p)) \quad , \quad y\mapsto [{}_y\Pi_b]
\]
is called the \emph{pro-unipotent Kummer map}. Exactly the same construction also provides a local pro-unipotent Kummer map
\[
j_\ell\colon Y(\bQ_\ell) \to \rH^1(G_\ell,\Pi(\bQ_p)) \quad , \quad y\mapsto [{}_y\Pi_b]
\]
for each prime~$\ell$.

The local non-abelian continuous cohomology sets~$\rH^1(G_\ell,\Pi(\bQ_p))$ are the~$\bQ_p$-points of certain local non-abelian continuous cohomology schemes~$\rH^1(G_\ell,\Pi)$, which are affine $\bQ_p$-schemes that are of finite type if~$\Pi$ is finite-dimensional \cite[Proposition~2]{kim:motivic}. The corresponding global object is called a \emph{Selmer scheme}. Several different versions of this object appear in the literature; we will primarily be interested in the \emph{refined Selmer scheme} of \cite{BD:refined}.

\begin{defn}
	\label{def: refined Selmer scheme}
	Suppose that~$S$ is a finite set of primes not containing~$p$, and that~$\cY/\bZ_S$ is a regular $S$-integral model of~$Y$. We define the \emph{refined Selmer scheme}
	\[
	\Sel_{S,\Pi}^{\min}(\cY)
	\]
	to be the $\bQ_p$-scheme parametrising continuous $G_\bQ$-cohomology classes~$\xi \in \rH^1(G_{\bQ}, \Pi)$ such that
	\[
	\xi|_{G_\ell} \in
	\begin{cases}
		j_\ell(\cY(\bZ_\ell))^{\Zar} & \text{if~$\ell\notin S$} \\
		j_\ell(Y(\bQ_\ell))^{\Zar} & \text{if~$\ell\in S$}
	\end{cases}
	\]
	for all primes~$\ell$ (including~$p$), where~$(-)^{\Zar}$ denotes Zariski-closure inside $\rH^1(G_\ell,\Pi)$.
	
	The refined Selmer scheme sits in a commuting square
	\begin{equation}\label{diag:kim_square_refined}
		\begin{tikzcd}
			\cY(\bZ_S) \arrow[r,hook]\arrow[d,"j_S"] & \cY(\bZ_p) \arrow[d,"j_p"] \\
			\Sel_{S,\Pi}^{\min}(\cY)(\bQ_p) \arrow[r,"\loc_p"] & \rH^1(G_p,\Pi(\bQ_p))
		\end{tikzcd}
	\end{equation}
	where the localisation map~$\loc_p$ is just restriction to the subgroup $G_p\subseteq G_\bQ$. The \emph{refined Chabauty--Kim locus}
	\[
	\cY(\bZ_p)_{S,\Pi}^{\min}\subseteq\cY(\bZ_p)
	\]
	is defined to be the set of all points~$y\in\cY(\bZ_p)$ such that~$j_p(y)$ lies in the scheme-theoretic image of~$\loc_p\colon\Sel_{S,\Pi}^{\min}(\cY)\to\rH^1(G_p,\Pi)$. In the special case that $\Pi$ is the whole fundamental group~$\pi_1^{\et,\bQ_p}(Y_{\Qbar},b)$, we write~$\cY(\bZ_p)_{S,\infty}^{\min}$ for the refined Chabauty--Kim locus. Recall that the Refined Kim's Conjecture~\ref{conj:kim} postulates that $\cY(\bZ_S) = \cY(\bZ_p)_{S,\infty}^{\min}$.
\end{defn}

\begin{rem}
	We recall for the benefit of the reader a few facts that are known about the images of the local Kummer maps above.
	\begin{itemize}
		\item If~$\ell\notin S\cup\{p\}$, then $j_\ell(\cY(\bZ_\ell))$ is finite \cite[Corollary~0.2]{kim-tamagawa}. If moreover~$\ell$ is a prime of good reduction for~$(\cY,b)$ -- meaning that~$\cY_{\bZ_\ell}$ is the complement of an \'etale divisor in a smooth proper $\bZ_\ell$-scheme and~$b$ is~$\bZ_\ell$-integral -- then
		\[
		j_\ell(\cY(\bZ_\ell)) \subseteq \{*\}
		\]
		where~$*$ is the basepoint in~$\rH^1(G_\ell,\Pi)$. (It can happen that~$\cY(\bZ_\ell)$ is empty when~$b$ is tangential, in which case the containment is strict.)
		\item If~$\ell\in S$, then the Zariski-closure of~$j_\ell(Y(\bQ_\ell))$ has dimension~$\leq1$ \cite[Proposition~1.2.1(2)]{BD:refined}.
		\item If~$\ell=p$ and~$p$ is a prime of good reduction for~$(\cY,b)$, then
		\[
		j_p(\cY(\bZ_p))^{\Zar} \subseteq \rH^1_f(G_p,\Pi)
		\]
		is contained in the subscheme parametrising crystalline~$\Pi$-torsors. We have equality provided~$\cY(\bZ_p)\neq\emptyset$ \cite[Theorem~1]{kim:albanese}\cite[Proposition~1.4]{kim:tangential}.
	\end{itemize}
\end{rem}

\begin{rem}
	\label{rem: unrefined CK}
	We will later see one other kind of Selmer scheme. Namely, suppose that the model~$\cY/\bZ_S$ of the hyperbolic curve~$Y$ is the complement of an \'etale divisor in a smooth proper $\bZ_S$-scheme and that the basepoint~$b$ is $\bZ_S$-integral. Then we also have the \emph{Bloch--Kato Selmer scheme}
	\[
	\rH^1_{f,S}(G_\bQ,\Pi)
	\]
	parametrising $\Pi$-torsors with a compatible continuous action of~$G_\bQ$ which is unramified outside~$S\cup\{p\}$ and crystalline at~$p$ (assumed $\notin S$). One can use the Bloch--Kato Selmer scheme to cut out a (non-refined) \emph{Chabauty--Kim locus} $\cY(\bZ_p)_{S,\Pi}\subseteq\cY(\bZ_p)$ via the commuting square
	\begin{equation}\label{diag:kim_square_unrefined}
		\begin{tikzcd}
			\cY(\bZ_S) \arrow[r,hook]\arrow[d,"j_S"] & \cY(\bZ_p) \arrow[d,"j_p"] \\
			\rH^1_{f,S}(G_\bQ,\Pi(\bQ_p)) \arrow[r,"\loc_p"] & \rH^1_f(G_p,\Pi(\bQ_p)) \,,
		\end{tikzcd}
	\end{equation}
	namely the set of points~$y\in\cY(\bZ_p)$ such that $j_p(y)$ lies in the scheme-theoretic image of~$\loc_p\colon\rH^1_{f,S}(G_\bQ,\Pi)\to\rH^1_f(G_p,\Pi)$. This is the Chabauty--Kim locus as defined in \cite{kim:motivic,kim:albanese}.
	
	Under the above assumptions, we have the containment
	\[
	\Sel_{S,\Pi}^{\min}(\cY)\subseteq\rH^1_{f,S}(G_\bQ,\Pi) \,,
	\]
	and hence the Chabauty--Kim locus contains the refined Chabauty--Kim locus: $\cY(\bZ_p)_{S,\Pi}\supseteq\cY(\bZ_p)_{S,\Pi}^{\min}$.
\end{rem}

\subsubsection{Finite descent and refined Chabauty--Kim}

The $\bQ_p$-pro-unipotent \'etale fundamental group~$\pi_1^{\et,\bQ_p}(Y_{\Qbar},b)$ is the Mal\u cev completion of the profinite \'etale fundamental group~$\pi_1^{\et}(Y_{\Qbar},b)$, meaning that there is a continuous group homomorphism
\[
\pi_1^\et(Y_{\Qbar},b) \to \pi_1^{\et,\bQ_p}(Y_{\Qbar},b)(\bQ_p)
\]
which is initial among continuous group homomorphisms from $\pi_1^\et(Y_{\Qbar},b)$ into the $\bQ_p$-points of pro-unipotent groups over~$\bQ_p$. Moreover, for any rational point~$y\in Y(\bQ)$ the $\bQ_p$-pro-unipotent \'etale path torsor $\pi_1^{\et,\bQ_p}(Y_{\Qbar};b,y)$ is the pushout of the profinite \'etale path torsor $\pi_1^{\et}(Y_{\Qbar};b,y)$ from~$\pi_1^{\et}(Y_{\Qbar},b)$ to~$\pi_1^{\et,\bQ_p}(Y_{\Qbar},b)$, meaning that we have
\[
\pi_1^{\et,\bQ_p}(Y_{\Qbar};b,y)(R) = \pi_1^{\et}(Y_{\Qbar};b,y)\times^{\pi_1^{\et}(Y_{\Qbar},b)}\pi_1^{\et,\bQ_p}(Y_{\Qbar},b)(R)
\]
for all~$\bQ_p$-algebras~$R$. The discussion in \cite[\S2.4]{stix:evidence_SC} shows that the class of the $\pi_1^{\et}(Y_{\Qbar},b)$-torsor $\pi_1^{\et}(Y_{\Qbar};b,y)$, endowed with its natural $G_\bQ$-action induced from the action on~$Y_{\Qbar}$, is equal to the class of the section~$\kappa(y)$ under the canonical identifications
\[
\Sec(Y/\bQ) = \rH^1(G_\bQ,\pi_1^{\et}(Y_{\Qbar},b)) = \{\text{$G_\bQ$-equivariant $\pi_1^{\et}(Y_{\Qbar},b)$-torsors}\}/\sim \,.
\]
Hence the composite
\[
Y(\bQ) \xrightarrow\kappa \rH^1(G_\bQ,\pi_1^{\et}(Y_{\Qbar},b)) \to \rH^1(G_\bQ,\Pi(\bQ_p))
\]
is equal to the pro-unipotent Kummer map~$j$. Similarly, the local pro-unipotent Kummer map~$j_\ell$ is the composite
\[
Y(\bQ_\ell) \xrightarrow{\kappa_\ell} \rH^1(G_\ell,\pi_1^{\et}(Y_{\Qbar},b)) \to \rH^1(G_\ell,\Pi(\bQ_p)) \,.
\]

As an easy consequence, we obtain the following compatibility between the finite descent locus and the refined Chabauty--Kim locus.

\begin{lem}\label{lem:project_into_kim}
	Let~$Y/\bQ$ be a hyperbolic curve with a rational basepoint $b$ (possibly tangential), let~$\cY/\bZ_S$ be a regular $S$-integral model of~$Y$ and let~$p\notin S$ be prime. Then the image of the projection
	\[
	\cY(\bA_{\bQ,S})_\bullet^{\fcov} \to \cY(\bZ_p)
	\]
	is contained in the refined Chabauty--Kim locus~$\cY(\bZ_p)_{S,\Pi}^{\min}$ for every quotient~$\Pi$ of $\pi_1^{\et}(Y_{\Qbar}, b)$.
\end{lem}

\begin{proof}
	Let $y = (y_v)_v \in \cY(\bA_{\bQ,S})_{\bullet}^{\fcov}$ be an $S$-integral adelic point which survives the finite descent obstruction, and let~$s\in\Sec(\cY/\bZ_S)^{\lgeom}$ be a Selmer section such that~$\loc(s)=y$ by Theorem~\ref{thm:image_of_localisation}. It follows from the above discussion that we have a commuting diagram
	\begin{center}
		\begin{tikzcd}
			{} & \rH^1(G_\bQ,\pi_1^{\et}(Y_{\Qbar},b)) \arrow[d]\arrow[r] & \rH^1(G_\bQ,\Pi(\bQ_p)) \arrow[d] \\
			Y(\bQ_\ell) \arrow[r,"\kappa_\ell"] & \rH^1(G_\ell,\pi_1^{\et}(Y_{\Qbar},b)) \arrow[r] & \rH^1(G_\ell,\Pi(\bQ_p))
		\end{tikzcd}
	\end{center}
	in which the composite along the bottom row is the local pro-unipotent Kummer map~$j_\ell$. So if~$\xi\in\rH^1(G_\bQ,\Pi(\bQ_p))$ is the image of~$s\in\Sec(Y/\bQ)=\rH^1(G_\bQ,\pi_1^{\et}(Y_{\Qbar},b))$, then $\xi|_{G_\ell}=j_\ell(y_\ell)$ for all primes~$\ell$. It follows that~$\xi\in\Sel_{S,\Pi}(\cY)^{\min}(\bQ_p)$, and hence that~$y_p\in\cY(\bZ_p)_{S,\Pi}^{\min}$ as desired.
\end{proof}

\begin{cor}\label{cor:conjecture_implication_2}
	Let~$Y/\bQ$ be a hyperbolic curve and let~$\cY/\bZ_S$ be a regular $S$-integral model of~$Y$ which has an $S$-integral point or $S$-integral tangential point. If Conjecture~\ref{conj:kim} holds for~$(\cY,S,p)$ for all primes~$p$ in a set of Dirichlet density~$1$, then Conjecture~\ref{conj:sufficiency_finite_descent} holds for~$(\cY,S)$.
\end{cor}

\begin{proof}
	Suppose that Conjecture~\ref{conj:kim} holds for all primes~$p$ in a set~$\fP$ of Dirichlet density~$1$. Let~$Z\subset Y$ denote the reduced closed subscheme of~$Y$ consisting of the finite set~$\cY(\bZ_S)$ of $S$-integral points. If $y=(y_v)_v\in\cY(\bA_{\bQ,S})_\bullet^{\fcov}$ lies in the finite descent locus, then by Lemma~\ref{lem:project_into_kim} we must have
	\[
	y_p \in \cY(\bZ_p)_{S,\infty}^{\min} = Z(\bQ)
	\]
	for all~$p\in\fP$. So we deduce by the theorem of the diagonal that
	\[
	y\in Y(\bQ)\cap\cY(\bA_{\bQ,S})_\bullet = \cY(\bZ_S)
	\]
	and we are done.
\end{proof}

This completes the proof of Theorem~\ref{thm:main_conjecture_implication}. %

\subsection{Generalisation to number fields}
\label{sec:remarks-number-fields}

Since Kim's Conjecture currently has only officially been postulated over~$\bQ$, we have restricted the discussion of the Chabauty--Kim method and the statement of \Cref{thm:main_conjecture_implication} to this case. Nevertheless, our results admit a straightforward generalisation to general number fields, which we briefly sketch. For a detailed study of the Chabauty--Kim method in this setting we refer to \cite{Dogra:unlikely-intersections}; see also \cite{DC:mixedtate2} for the specific case of the thrice-punctured line.

Given a hyperbolic curve $Y/K$ with regular model $\cY/\cO_{K,S}$, if $p$ is a rational prime not divisible by any prime in~$S$, we have a commuting square which generalises diagram~\eqref{diag:kim_square_refined}:
\begin{equation}\label{diag:kim_square_refined_number_field}
	\begin{tikzcd}
		\cY(\cO_{K,S}) \arrow[r,hook]\arrow[d,"j_S"] & \cY(\cO_K \otimes \bZ_p) = \prod_{v|p} \cY(\cO_v) \arrow[d,"\prod_{v|p} j_v"] \\
		\Sel_{S,\Pi}^{\min}(\cY)(\bQ_p) \arrow[r,"\loc_p"] & \prod_{v|p} \rH^1(G_v,\Pi(\bQ_p)).
	\end{tikzcd}
\end{equation}
The Chabauty--Kim locus $\cY(\cO_K \otimes \bZ_p)_{S,\Pi}^{\min}$ is defined as the set of points $y = (y_v)_{v|p} \in \cY(\cO_K \otimes \bZ_p)$ such that $(j_v(y_v))_{v|p}$ lies in the scheme-theoretic image of $\loc_p$ in the above diagram. For a single place $v$ dividing~$p$, we define $\cY(\cO_v)_{S,\Pi}^{\min}$ as the image of the projection of $\cY(\cO_K \otimes \bZ_p)_{S,\Pi}^{\min}$ to the $v$-component. 
\Cref{lem:project_into_kim} now generalises to the statement that the image of the projection
\[ \cY(\bA_{K,S})_{\bullet}^{\fcov} \to \cY(\cO_K \otimes \bZ_p) = \prod_{v|p} \cY(\cO_v) \]
is contained in $\cY(\cO_K \otimes \bZ_p)_{S,\Pi}^{\min}$, so projecting further to $\cY(\cO_v)$ one lands in the Chabauty--Kim locus $\cY(\cO_v)_{S,\Pi}^{\min}$. From this one obtains the following general version of \Cref{thm:main_conjecture_implication} in the same way as in the case $K = \bQ$ above.

\begin{thm}
	\label{thm:main-thm-number-fields}
	Assume that $\cY(\cO_{K,S}) = \cY(\cO_v)_{S,\infty}^{\min}$ for a Dirichlet-dense set of primes $v$ of $K$. Then the $S$-Selmer Section Conjecture~\ref{conj:s-selmer_sc} holds for $(\cY,S)$.
\end{thm}

One can use \Cref{thm:main-thm-number-fields} to prove instances of the $S$-Selmer Section Conjecture over imaginary quadratic fields.

\begin{rem}
	\label{rem:qzeta3-results}
	When $K$ is an imaginary quadratic field, $S = \emptyset$, and $\cY = \bP^1 \smallsetminus \{0,1,\infty\}$, using essentially the same reasoning as in \cite[§6.1]{BDCKW}, one sees that $\cY(\cO_v)_{\emptyset,1}^{\min} \subseteq Z(K_v)$ for all split primes~$v$ of~$K$, where $Z \subset Y$ is the reduced finite subscheme consisting of the primitive sixth roots of unity. If $K = \bQ(\zeta_3)$ then $\zeta_6$ and $\zeta_6^{-1}$ are indeed $\cO_K$-points of~$\cY$, so we have the equality
	\[ \cY(\bZ[\zeta_3]) = \{\zeta_6, \zeta_6^{-1}\} = \cY(\cO_v)_{\emptyset,1}^{\min}. \]
	Since the split primes of~$K$ have Dirichlet density~1, \Cref{thm:main-thm-number-fields} implies that the Selmer Section Conjecture~\ref{conj:s-selmer_sc} holds for $K = \bQ(\zeta_3)$, $S = \emptyset$, $\cY = \bP^1 \smallsetminus \{0,1,\infty\}$. In fact the inclusion $\cY(\cO_v)_{\emptyset,1}^{\min} \subseteq Z(K_v)$ is enough to deduce the Selmer Section Conjecture, so Conjecture~\ref{conj:s-selmer_sc} holds more generally for any imaginary quadratic field (and $S = \emptyset$). Further results in this direction are contained in forthcoming work by M.L. and Xiang Li. Like \Cref{thm:main_s-selmer_sc}, these instances of the Selmer Section Conjecture were already proved with different methods by Stix \cite[Corollary~6]{stix:birationalSC}.
\end{rem}

%% file: motives.tex
\section{Comparison of \'etale and motivic Selmer schemes}
\label{sec: comparison theorem}

The Chabauty--Kim method for~$\cY=\bP^1\smallsetminus\{0,1,\infty\}$ is understood quite explicitly thanks to work of Corwin, Dan-Cohen and Wewers \cite{DCW:mixedtate1,DC:mixedtate2,DCW:explicitCK,CDC:polylog1,CDC:polylog2}. Their work uses a different foundation for the theory, using the motivic fundamental groupoid of the thrice-punctured line in place of the \'etale fundamental groupoid. In this section, we will recall their method and explain how it relates to the Chabauty--Kim method as originally formulated by Kim. We advise the reader to skip this section on a first reading: we will not use the motivic approach in the following sections, and the point of this section is simply to justify why we can import results from the papers of Corwin, Dan-Cohen and Wewers despite the different foundations of the method.\smallskip

Let $S$ be a finite set of primes and denote by $\bZ_S$ the ring of $S$-integers. Let
\[ \MT(\bZ_S, \bQ) \]
be the category of \emph{mixed Tate motives} over $\bZ_S$ with $\bQ$-coefficients~\cite[§1]{deligne-goncharov}. This is a Tannakian category over $\bQ$. Its simple objects are the \emph{Tate motives} $\bQ(n)$ for $n \in \bZ$, which satisfy
\begin{align}
	\label{eq: mixed tate motives ext groups}
	\Ext^1_{\MT(\bZ_S,\bQ)}(\bQ(0), \bQ(n)) = K_{2n-1}(\bZ_S)_{\bQ} = \begin{cases}
		0 & \text{if $n \leq 0$},\\
		\bZ_S^\times \otimes \bQ & \text{if $n=1$},\\
		K_{2n-1}(\bQ) \otimes \bQ &  \text{if $n\geq 2$},
	\end{cases}
\end{align}
and
\[ \Ext^2_{\MT(\bZ_S,\bQ)}(\bQ(0), \bQ(n)) = 0 \quad\text{for all~$n$}\,. \]
A neutral fibre functor for~$\MT(\bZ_S,\bQ)$ is given by the \emph{de Rham realisation functor}
\[
\rho_{\dR} \colon \MT(\bZ_S,\bQ) \to \bQ\mhyphen\mathrm{Vect_f},
\]
where $\bQ\mhyphen\mathrm{Vect_f}$ is the category of finite-dimensional $\bQ$-vector spaces. We write~$G_{\bQ,S}^{\MT}$ for the Tannakian fundamental group of~$\MT(\bZ_S,\bQ)$ based at~$\rho_{\dR}$, which is a semidirect product of~$\bG_m$ by a pro-unipotent group~$U_{\bQ,S}^{\MT}$.

Let $\cY/\bZ_S$ be the thrice-punctured line with generic fibre $Y/\bQ$, and let $b$ be an $S$-integral base point. By this we mean either an $S$-integral point $b \in \cY(\bZ_S)$ or an $S$-integral tangent vector, i.e.\ a nowhere vanishing section of the tangent bundle $y^*\dT_{\bP^1/\bZ_S}$ over $\bZ_S$ at a cusp $y \in \{0,1,\infty\}$.
Deligne and Goncharov \cite[Th\'eor\`eme~4.4]{deligne-goncharov} construct the ``motivic fundamental group'' $\pi_1(Y,b)$ of $\bP^1 \smallsetminus\{0,1,\infty\}$, which is a pro-algebraic group in $\MT(\bZ_S,\bQ)$ (in the sense of \cite[\S5]{deligne:droite-projective}), whose de Rham realisation is the de Rham fundamental group~$\pi_1^{\dR}(Y,b)$. We let~$\Pi$ be a quotient of~$\pi_1(Y,b)$ in the category~$\MT(\bZ_S,\bQ)$. The \emph{motivic Selmer scheme} is defined to be the affine $\bQ$-scheme parametrising $\Pi$-torsors in~$\MT(\bZ_S,\bQ)$ as follows.

\begin{defn}
	For a $\bQ$-algebra~$R$, we write
	\[
	\rH^1(G_{\bQ,S}^{\MT},\Pi^{\dR})(R)
	\]
	for the set of algebraic cocycles $\xi\colon G_{\bQ,S,R}^{\MT}\to\Pi^{\dR}_R$ defined over~$R$, modulo the natural twisting action of~$\Pi^{\dR}(R)$. The functor
	\[
	R \mapsto \rH^1(G_{\bQ,S}^{\MT},\Pi^{\dR})(R)
	\]
	is representable by an affine $\bQ$-scheme, which is called the \emph{motivic Selmer scheme} $\rH^1(G_{\bQ,S}^{\MT},\Pi^{\dR})$. Equivalently (and without reference to the de Rham fibre functor), the motivic Selmer scheme can be described as the affine $\bQ$-scheme representing the functor
	\[
	R \mapsto \{\text{$\Pi$-torsors over~$R$ in~$\MT(\bZ_S,\bQ)$}\}/\text{iso} \,,
	\]
	where~$R$ is viewed as a ring in~$\ind{\MT(\bZ_S,\bQ)}$ in the obvious way (as a direct sum of copies of the unit object).
\end{defn}

This cohomology scheme can be understood quite explicitly \cite[Theorem~A.4]{corwin:mixed_elliptic}\footnote{\cite[Theorem~A.4]{corwin:mixed_elliptic} is only stated in the case that~$\Pi$ is unipotent, i.e.~finite-dimensional. But the same proof works verbatim for a general pro-unipotent quotient~$\Pi$. In particular, this proves that the functor~$\rH^1(G^\MT_{\bQ,S},\Pi^{\dR})$ is representable, a fact which is not immediately obvious when~$\Pi$ is infinite-dimensional. One can also prove representability by using a Mittag--Leffler argument to show that $\rH^1(G^\MT_{\bQ,S},\Pi^{\dR})=\varprojlim_n\rH^1(G^\MT_{\bQ,S},\Pi^{\dR}_n)$ where~$\Pi^{\dR}_n$ is the maximal $n$-step unipotent quotient of~$\Pi^{\dR}$. See e.g.~\cite[Proposition~2.2.6]{betts:effective} for a similar Mittag--Leffler argument.}. Every cohomology class $[\xi]\in\rH^1(G_{\bQ,S}^{\MT},\Pi^{\dR})(R)$ can be represented by a unique cocycle $\xi\colon G_{\bQ,S,R}^{\MT}\to\Pi^{\dR}_R$ whose restriction to~$\bG_{m,R}$ is trivial. The restriction of~$\xi$ to the pro-unipotent part~$U_{S,R}^{\MT}$ is then $\bG_{m,R}$-equivariant, and the assignment $[\xi] \mapsto \xi|_{U_{\bQ,S,R}^{\MT}}$ gives an identification
\[
\rH^1(G_{\bQ,S}^{\MT},\Pi^{\dR}) = \rZ^1(U_{\bQ,S}^{\MT},\Pi^{\dR})^{\bG_m}
\]
between the motivic Selmer scheme and the affine $\bQ$-scheme parametrising $\bG_m$-equivariant cocycles for the action of~$U_{\bQ,S}^{\MT}$ on~$\Pi^{\dR}$. The affine ring~$\cO(U_{\bQ,S}^{\MT})$ of~$U_{\bQ,S}^{\MT}$ is known as the \emph{ring of unipotent motivic periods} \cite[§3.4.1]{brown:integral_points}, and the $p$-adic period map $\cO(U_{\bQ,S}^{\MT})\to\bQ_p$ defines a point in~$U_{\bQ,S}^{\MT}(\bQ_p)$, which is the inverse of the point $\eta_p^\ur$ from \cite[Lemma~2.2.5]{chatzistamatiou-unver:p-adic_periods}. Evaluation at~$(\eta_p^\ur)^{-1}$ defines a morphism
\begin{equation}
	\label{eq: motivic localisation map}
	\ev_p\colon \rH^1(G_{\bQ,S}^{\MT},\Pi^{\dR})_{\bQ_p} = \rZ^1(U_{\bQ,S}^{\MT},\Pi^{\dR})^{\bG_m}_{\bQ_p} \to \Pi^{\dR}_{\bQ_p}
\end{equation}
of affine $\bQ_p$-schemes, also called the motivic localisation map~$\loc_\Pi$ in~\cite[§2.4.2]{CDC:polylog1}.\footnote{\cite{CDC:polylog1} actually defines $\loc_\Pi$ to be evaluation at $\eta_p^{\ur}$ rather than its inverse. However, it is evaluation at~$(\eta_p^{\ur})^{-1}$ that is needed for the theory \cite{CDC:polylog1}, as it is this choice which makes the square~\eqref{diag:motivic_kim_square} below commute.}
\smallskip

For any $S$-integral point~$y\in\cY(\bZ_S)$, Deligne and Goncharov also define a \emph{motivic path torsor} $\pi_1(Y;b,y)$, which is a $\pi_1(Y,b)$-torsor over $\bQ$ in $\MT(\bZ_S,\bQ)$ \cite{deligne-goncharov}. Its pushout along the quotient map $\pi_1(Y,b) \twoheadrightarrow \Pi$ yields a $\Pi$-torsor
\[
{}_y\Pi_b \coloneqq \pi_1(Y;b,y) \overset{\pi_1(Y,b)}{\times} \Pi
\]
over $\bQ$ in $\MT(\bZ_S,\bQ)$, i.e.\ a $\bQ$-point of $\rH^1(G_{\bQ,S}^{\MT},\Pi^{\dR})$. We thus have a \emph{motivic Kummer map}
\[ j_S^{\dR}\colon \cY(\bZ_S) \to \rH^1(G_{\bQ,S}^{\MT},\Pi^{\dR})(\bQ) \quad , \quad y \mapsto [{}_y\Pi_b] \,. \]

Additionally, when~$\Pi=\pi_1(Y,b)$ is the full fundamental group, the affine ring of~$\Pi^{\dR}$ can be identified with the shuffle algebra on the space of differential forms on~$Y$ with at worst logarithmic poles along the boundary \cite[\S2.2.2]{CDC:polylog1}, so one can define a \emph{de Rham Kummer map}
\[
j_p^{\dR} \colon \cY(\bZ_p) \to \pi_1^{\dR}(Y,b)(\bQ_p)
\]
sending a point~$y\in\cY(\bZ_p)$ to the point corresponding to the map
\[
\cO(\pi_1^{\dR}(Y,b)) \to \bQ_p,\quad \omega \mapsto \int_b^y\omega \,,
\]
where $\int_b^y(-)$ denotes the iterated Coleman integral. For a general quotient~$\Pi$, composing with the projection $\pi_1(Y,b)\twoheadrightarrow\Pi$ yields a de Rham Kummer map
\[
j_p^{\dR} \colon \cY(\bZ_p) \to \Pi^{\dR}(\bQ_p) \,.
\]
These two Kummer maps fit into a commuting square \cite[\S2.4.2]{CDC:polylog1}
\begin{equation}\label{diag:motivic_kim_square}
	\begin{tikzcd}
		\cY(\bZ_S) \arrow[r,hook]\arrow[d,"j_S^{\dR}"] & \cY(\bZ_p) \arrow[d,"j^{\dR}_p"] \\
		\rH^1(G_{\bQ,S}^{\MT},\Pi^{\dR})(\bQ_p) \arrow[r,"\ev_p"] & \Pi^{\dR}(\bQ_p) \,,
	\end{tikzcd}
\end{equation}
strongly reminiscent of the commuting square~\eqref{diag:kim_square_unrefined} in the classical Chabauty--Kim method. What we wish to explain carefully in this section is how to identify these two squares with one another, in order that the results of \cite{CDC:polylog1} can be applied to the classical Chabauty--Kim square~\eqref{diag:kim_square_unrefined} and its refined variant~\eqref{diag:kim_square_refined}.

\subsection{\'Etale realisation}

Now let~$p\notin S$ be a prime. We fix an algebraic closure $\Qbar/\bQ$, let $G_{\bQ} \coloneqq \Gal(\Qbar/\bQ)$ and denote by $\Rep_{\bQ_p}^{f,S}(G_\bQ)$ the category of continuous $\bQ_p$-linear $G_{\bQ}$-representations which are unramified outside $S \cup \{p\}$ and crystalline at $p$. This is a Tannakian category over $\bQ_p$. There is a $p$-adic étale realisation functor
\[
\rho_\et\colon \MT(\bZ_S,\bQ) \to \Rep_{\bQ_p}^{f,S}(G_\bQ), \quad M \mapsto M^{\et} \,,
\]
which is an exact $\bQ$-linear tensor-functor. The \'etale realisation functor takes the motivic fundamental group~$\pi_1(Y,b)$ to the $\bQ_p$-pro-unipotent \'etale fundamental group $\pi_1^{\et,\bQ_p}(Y_{\Qbar},b)$, and takes a motivic path torsor~$\pi_1(Y;b,y)$ to the corresponding $\bQ_p$-pro-unipotent \'etale path torsor $\pi_1^{\et,\bQ_p}(Y_{\Qbar};b,y)$ \cite[Th\'eor\`eme~4.4]{deligne-goncharov}. The \'etale realisation~$\Pi^{\et}$, therefore, is a $G_\bQ$-equivariant quotient of~$\pi_1^{\et,\bQ_p}(Y_{\Qbar},b)$ and the \'etale realisation of~${}_y\Pi_b$ is isomorphic to~${}_y\Pi^{\et}_b$ for all~$y\in\cY(\bZ_S)$, i.e.\ the following diagram commutes
\begin{center}
	\begin{tikzcd}
		\cY(\bZ_S) \arrow[r]\arrow[d,equals] & \rH^1(G_{\bQ,S}^{\MT},\Pi^{\dR})(\bQ) \arrow[d,"\rho_\et"] \\
		\cY(\bZ_S) \arrow[r] & \rH^1_{f,S}(G_\bQ,\Pi^{\et})(\bQ_p) \,.
	\end{tikzcd}
\end{center}

Now if~$R$ is a $\bQ_p$-algebra and~$P$ is a $\Pi$-torsor over~$R$ in~$\MT(\bZ_S,\bQ)$, then the \'etale realisation~$P^{\et}$ is a $\Pi^{\et}$-torsor over $R\otimes_{\bQ}\bQ_p$. Base-changing along the multiplication map~$R\otimes_{\bQ}\bQ_p\to R$ yields a $\Pi^{\et}$-torsor over~$R$. This construction provides a morphism
\begin{equation}\label{eq: Selmer scheme comparison}
	\rH^1(G_{\bQ,S}^{\MT},\Pi^{\dR})_{\bQ_p} \to \rH^1_{f,S}(G_\bQ,\Pi^{\et})
\end{equation}
of affine~$\bQ_p$-schemes. The compatibility between the classical Chabauty--Kim method and the motivic Chabauty--Kim method is made precise in the following theorem.

\begin{thm}\label{thm:comparison_selmer_schemes}
	Let $\pi_1(Y,b) \twoheadrightarrow \Pi$ be a quotient in~$\MT(\bZ_S,\bQ)$ and $\Pi^{\et}$ its $p$-adic étale realisation. Then the map \eqref{eq: Selmer scheme comparison} is an isomorphism, fitting into a commuting diagram
	\begin{equation}\label{diag:etale-motivic_comparison_commuting_diagram}
		\begin{tikzcd}[column sep = small]
			{} & \cY(\bZ_S) \arrow[rr,hook]\arrow[dl,"j_S^{\dR}"] \arrow[dd,"j_S"', pos=0.3] && \cY(\bZ_p) \arrow[dl,"j^{\dR}_p"]\arrow[dd,"j_p"] \\
			\rH^1(G_{\bQ,S}^{\MT},\Pi^{\dR})(\bQ_p) \arrow[rr,"\ev_p",pos=0.7, crossing over]\arrow[dr,equal,"\eqref{eq: Selmer scheme comparison}"',"\sim"] && \Pi^{\dR}(\bQ_p) \arrow[dr,equals,"\sim"] & {} \\
			{} & \rH^1_{f,S}(G_\bQ,\Pi^{\et}(\bQ_p)) \arrow[rr,"\loc_p"]  && \rH^1_f(G_p,\Pi^{\et}(\bQ_p))
		\end{tikzcd}
	\end{equation}
	where the identification $\Pi^{\dR}_{\bQ_p}\cong\rH^1_f(G_p,\Pi^{\et})$ is the one coming from the Bloch--Kato exponential \cite[Proposition~1.4]{kim:tangential}.
\end{thm}

The vertical face of the commuting prism above is the usual (unrefined) Chabauty--Kim square \eqref{diag:kim_square_unrefined}, and the top face is the motivic Chabauty--Kim square \eqref{diag:motivic_kim_square}.

\begin{rem}
	The commutativity of the above diagram implies that the unrefined Chabauty--Kim locus~$\cY(\bZ_p)_{S,\Pi^\et}$ is exactly the set of points~$y\in\cY(\bZ_p)$ such that~$j_p^{\dR}(y)$ lies in the scheme-theoretic image of the evaluation map~$\ev_p$. We remark that this could \emph{a priori} be a different set to the motivic Chabauty--Kim locus of \cite[Definition~2.27]{CDC:polylog1}, which we here denote~$\cY(\bZ_p)_{S,\Pi}^{\mathrm{mot}}$. Indeed, one can consider the universal cocycle evaluation map \cite[Definition~2.20]{CDC:polylog1}
	\[
	\mathfrak{ev}\colon \rH^1(G_{\bQ,S}^{\MT},\Pi^{\dR})\times U_{\bQ,S}^{\MT} \to \Pi^{\dR}\times U_{\bQ,S}^{\MT} \,,
	\]
	and the definition of the motivic Chabauty--Kim locus is equivalent to saying that $\cY(\bZ_p)_{S,\Pi}^{\mathrm{mot}}$ is the set of points~$y\in\cY(\bZ_p)$ such that~$j_p^{\dR}(y)$ lies in the fibre of the scheme-theoretic image of~$\mathfrak{ev}$ over the point~$(\eta_p^\ur)^{-1}\in U_{\bQ,S}^{\MT}(\bQ_p)$. The fibre of the Zariski-closure of the image of~$\mathfrak{ev}$ could \emph{a priori} be larger than the Zariski-closure of the image of~$\ev_p$ (which is the fibre of~$\mathfrak{ev}$ at~$(\eta_p^\ur)^{-1}$), so we only \emph{a priori} have the containment
	\[
	\cY(\bZ_p)_{S,\Pi^\et} \subseteq \cY(\bZ_p)_{S,\Pi}^{\mathrm{mot}} \,,
	\]
	see \cite[Remark~2.28]{CDC:polylog1}. Note that the justification of~\cite[Remark~2.28]{CDC:polylog1} implicitly uses the compatibility between \'etale and motivic Selmer schemes made precise in Theorem~\ref{thm:comparison_selmer_schemes}, but this is not proved in \cite{CDC:polylog1}.
\end{rem}

\subsection{Mixed Tate categories}

The remainder of this section is devoted to a proof of Theorem~\ref{thm:comparison_selmer_schemes}. We start by recalling some general definitions and properties of mixed Tate categories. For a reference, see for instance~\cite[Appendix~A]{goncharov-zhu}.

\begin{defn}
	Let $K$ be a field. Let $\mathcal{T}$ be a Tannakian $K$-category, and let $K(1)$ be a fixed rank 1 object of $\dT$. 
	A pair $(\mathcal{T}, K(1))$ is called a \textit{mixed Tate category} if the objects $K(n) \coloneqq K(1)^{\otimes n}$ for  $n \in \bZ$, are mutually non-isomorphic, any irreducible object in $\mathcal{T}$ is isomorphic to some $K(n)$, and one has
	\[
	\Ext_{\mathcal{T}}^{1}(K(0), K(n))=0
	\]
	for all $n \leq 0$.
\end{defn}

If $(\mathcal{T}, K(1))$ is a mixed Tate category, then every object $M \in \dT$ admits a unique increasing filtration $W_\bullet M$ indexed by the even integers, called the \emph{weight filtration}, such that $\gr_{2n}^W M$ is a direct sum of copies of $K(-n)$. There exists a canonical fibre functor to the category of finite-dimensional $K$-vector spaces
\begin{equation}
	\label{eq: canonical fibre functor}
	\omega = \omega_{\mathcal{T}} \colon \mathcal{T} \to K\mhyphen\mathrm{Vect_f}, \quad M \mapsto \bigoplus_n \Hom(K(-n), \gr_{2n}^W M) \,.\quad\footnote{In the case of the category $\MT(\bZ_S,\bQ)$ of mixed Tate motives, the canonical fibre functor and the de Rham realisation functor are canonically isomorphic \cite[Proposition~2.10]{deligne-goncharov}.}
\end{equation}
By the Tannakian formalism, $\dT$ is equivalent to the category of representations of the Tannakian fundamental group
\[ G_{\dT} \coloneqq \pi_1(\dT, \omega) \coloneqq \underline{\Aut}^{\otimes}(\omega) \,. \]
The subcategory of semisimple objects (direct sums of the Tate objects $K(n)$) is isomorphic to the category of graded finite-dimensional $K$-vector spaces (with $K(n)$ sitting in degree~$-n$), which defines a surjective homomorphism $G_{\dT} \twoheadrightarrow \bG_m$. In other words, the homomorphism $G_{\dT}\to\bG_m$ is given by the action of~$G_{\dT}$ on the fibre of~$K(-1)$.\footnote{This is the opposite of the convention in \cite{deligne-goncharov}, where the homomorphism is given by the action on~$K(1)$.} Its kernel is a pro-unipotent group $U_{\dT}$. The grading on $\omega(M)$ for all $M \in \dT$ defines a splitting, so that $G_{\dT}$ is canonically a semi-direct product:
\[ G_{\dT} = U_{\dT} \rtimes \bG_m \,. \]
The Lie algebra of $U_{\dT}$ has a grading via the $\bG_m$-action. It is (non-canonically) isomorphic to the graded pro-nilpotent Lie algebra generated by the graded vector space
\[ \Ext^1_{\dT}(K(0), K(n))^\vee \text{ in degree $-n$, where $n=1,2,\ldots$} \]
with relations coming from $\Ext^2_{\dT}(K(0), K(n))$ for $n \geq 1$. In particular, if all of these $\Ext^2$ spaces vanish, then $\Lie(U_\dT)$ is a free pro-nilpotent Lie algebra. The abelianisation of~$\Lie(U_{\dT})$ is \emph{canonically} isomorphic to the product
\[
\prod_{n>0}\Ext^1_{\dT}(K(0),K(n))^\vee
\]

\subsubsection{Tate functors}

\begin{defn}
	\label{def: Tate functor}
	Let $K_2/K_1$ be an extension of fields and let $(\dT_1, K_1(1))$ and $(\dT_2, K_2(1))$ be mixed Tate categories over~$K_1$ and~$K_2$, respectively. A \emph{Tate functor}
	\[
	\rho \colon \dT_1 \to \dT_2
	\]
	is a $K_1$-linear exact tensor functor equipped with an isomorphism $\rho(K_1(1)) = K_2(1)$.
\end{defn}

Any Tate functor $\rho\colon \dT_1 \to \dT_2$ is automatically compatible with the canonical fibre functors, hence induces a homomorphism of Tannakian fundamental groups
\[
\rho^*\colon G_{\dT_2} \to G_{\dT_1,K_2} \,.
\]
This homomorphism is compatible with the semidirect product structure $G_{\dT_i} = U_{\dT_i} \rtimes \bG_m$, so it induces a graded Lie algebra homomorphism
\[
\rho^*\colon \Lie(U_{\dT_2}) \to \Lie(U_{\dT_1})_{K_2} \,.
\]

\begin{lem}\label{lem:iso_on_tannaka_groups}
	Let $K_2/K_1$ be an extension of fields and $\rho\colon\dT_1\to\dT_2$ a Tate functor. Suppose that the induced map
	\[
	\rho_*\colon \Ext^1_{\dT_1}(K_1(0),K_1(n))\otimes_{K_1}K_2 \to \Ext^1_{\dT_2}(K_2(0),K_2(n))
	\]
	is an isomorphism for all~$n>0$, and that $\Ext^2_{\dT_1}(K_1,K_1(n)) = 0$ for all~$n$. Then the induced map
	\[
	\rho^* \colon G_{\dT_2}\to G_{\dT_1,K_2}
	\]
	is an isomorphism.
\end{lem}

\begin{proof}
	It suffices to show that the induced map
	\[
	\rho^* \colon \Lie(U_{\dT_2}) \to \Lie(U_{\dT_1})_{K_2}
	\]
	is an isomorphism. The first assumption ensures that~$(\rho^*)^{\ab}$ is an isomorphism, and this implies inductively that~$\rho^*$ is surjective modulo each step in the descending central series. So~$\rho^*$ is surjective. Since $\Lie(U_{\dT_1})$ is free, $\rho^*$ must have a right inverse~$s\colon \Lie(U_{\dT_1})_{K_2} \to \Lie(U_{\dT_2})$. Since~$s^{\ab}$ is also an isomorphism, exactly the same argument establishes that~$s$ is surjective. So~$\rho^*$, having a surjective right inverse, is an isomorphism.
\end{proof}

\subsection{Comparison of Selmer schemes}
\label{sec: CK comparison}

We say that a representation $V$ of $G_\bQ$ is \emph{mixed Tate} just when it admits a finite ascending filtration $W_\bullet V$ supported in even degrees, called the \emph{weight filtration}, such that $\gr^{W}_{2n}V$ is isomorphic to a direct sum of copies of~$\bQ_p(-n)$ for all~$n$. The weight filtration on any mixed Tate representation is unique, and any $G_\bQ$-equivariant map between mixed Tate representations is automatically strict for the weight filtrations. We write
\[
\Rep_{\bQ_p}^{\MT,S}(G_{\bQ}) \subseteq \Rep_{\bQ_p}^{f,S}(G_{\bQ})
\]
for the category of \emph{mixed Tate representations} which are unramified outside $S\cup\{p\}$ and crystalline at~$p$. The category $\Rep_{\bQ_p}^{\MT,S}(G_{\bQ})$ is a mixed Tate category over $\bQ_p$: the Ext-groups
\[
\Ext^1_{\Rep_{\bQ_p}^{\MT,S}(G_{\bQ})}(\bQ_p(0),\bQ_p(n))
\]
vanish for $n<0$ since any extension of $\bQ_p(0)$ by $\bQ_p(n)$ is split by the weight filtration, and vanish for $n=0$ since any continuous homomorphism $G_{\bQ}\to\bQ_p$ unramified outside~$S\cup\{p\}$ must factor through the finite group $G_{S\cup\{p\}}^{\ab}$. We write~$G_{\bQ,S}^{\MTR}$ for the Tannakian fundamental group of~$\Rep_{\bQ_p}^{\MT,S}(G_\bQ)$ based at its canonical fibre functor. As in the category of mixed Tate motives, the canonical fibre functor is canonically $\otimes$-isomorphic to the de Rham fibre functor~$\sD_{\dR}(-|_{G_p})$, and we permit ourselves to identify the two.

The \'etale realisation functor
\[
\rho_\et \colon \MT(\bZ_S,\bQ) \to \Rep_{\bQ_p}^{f,S}(G_{\bQ})
\]
has image contained inside the mixed Tate representations. So there is an induced homomorphism
\begin{equation}\label{eq:tannaka_group_comparison}
G_{\bQ,S}^{\MTR} \to G_{\bQ,S,\bQ_p}^{\MT}
\end{equation}
of affine group schemes over~$\bQ_p$. In the appendix, we will show that the map
\[
\Ext^1_{\MT(\bZ_S,\bQ)}(\bQ(0),\bQ(n)) \otimes_{\bQ} \bQ_p \to \Ext^1_{\Rep_{\bQ_p}^{\MT,S}(G_\bQ)}(\bQ_p(0),\bQ_p(n))
\]
induced by \'etale realisation is an isomorphism. This is proved by Kummer theory for $n=1$, and by identifying it with Soulé's regulator map
\[ K_{2n-1}(K) \otimes_{\bZ} \bQ_p \to \rH^1_{\et}(\Spec(K), \bQ_p(n)) \]
for $n >1$, which is known to be an isomorphism by \cite[Theorem~1]{soule:higher} (for $p$ odd) and \cite[Lemma~2.19]{hyperbolic-tesselations} (for $p=2$).
Since $\Ext^2$'s in $\MT(\bZ_S,\bQ)$ vanish \cite[Proposition~1.9]{deligne-goncharov}, this implies by Lemma~\ref{lem:iso_on_tannaka_groups} that~\eqref{eq:tannaka_group_comparison} is an isomorphism (as per Theorem~\ref{thm:tannaka_group_comparison} in the introduction).

We use this to finish the proof of Theorem~\ref{thm:comparison_selmer_schemes}. The map
\[
\rH^1(G_{\bQ,S}^{\MT},\Pi^{\dR})_{\bQ_p} \to \rH^1_{f,S}(G_\bQ,\Pi^{\et})
\]
described earlier can be factored as
\[
\rH^1(G_{\bQ,S,\bQ_p}^{\MT},\Pi^{\dR}_{\bQ_p}) \xrightarrow{\sim} \rH^1(G_{\bQ,S}^{\MTR},\Pi^{\dR}_{\bQ_p}) \to \rH^1_{f,S}(G_\bQ,\Pi^\et)
\]
where the first map is pullback along the isomorphism~\eqref{eq:tannaka_group_comparison} and the second corresponds to the inclusion
\[
\{\text{$\Pi^{\et}$-torsors over~$R$ in~$\Rep_{\bQ_p}^{\MT,S}(G_\bQ)$}\}/\text{iso} \subseteq \{\text{$\Pi^{\et}$-torsors over~$R$ in~$\Rep_{\bQ_p}^{f,S}(G_\bQ)$}\}/\text{iso}
\]
for a $\bQ_p$-algebra~$R$.

To show that the second map is also an isomorphism we need to show that any $\Pi^{\et}$-torsor~$P^{\et}$ over a $\bQ_p$-algebra~$R$ in~$\Rep_{\bQ_p}^{f,S}(G_\bQ)$ is automatically mixed Tate, i.e.\ is in $\Rep_{\bQ_p}^{\MT,S}(G_{\bQ})$. For this, let
\[
\bQ_p=W_0\cO(\Pi^{\et}) \subseteq W_2\cO(\Pi^{\et}) \subseteq W_4\cO(\Pi^{\et}) \subseteq \dots
\]
denote the weight filtration on $\cO(\Pi^{\et})\in\ind\Rep_{\bQ_p}^{\MT,S}(G_{\bQ})$. Picking an element~$\gamma\in P^{\et}(R)$ (which exists since $\Pi^{\et}$ is pro-unipotent) induces an isomorphism $\Pi^{\et}_R \simeq P^{\et}$ and hence an isomorphism $\cO(P^{\et}) \simeq \cO(\Pi^{\et})\otimes_{\bQ_p}R$. We define
\[
R = W_0\cO(P^{\et}) \subseteq W_2\cO(P^{\et}) \subseteq W_4\cO(P^{\et}) \subseteq \dots
\]
to be the $R$-linear filtration on $\cO(P^{\et})$ corresponding to the weight filtration on $\cO(\Pi^{\et})$ under this identification. As in \cite[Proposition~2.6]{betts:motivic-anabelian-geometry}, this filtration is independent of the choice of~$\gamma$, as is the induced identification
\[
\gr^W_\bullet\cO(P^{\et}) \cong \gr^W_\bullet\cO(\Pi^{\et})\otimes_{\bQ_p}R
\]
on graded pieces. In particular, the filtration on $\cO(P^{\et})$ is $G_\bQ$-invariant, and each $\gr^W_{2n}\cO(P^{\et})$ is $G_\bQ$-equivariantly isomorphic to a direct sum of copies of $\bQ_p(-n)\otimes_{\bQ_p}R$. So $\cO(P^{\et})\in\ind\Rep_{\bQ_p}^{\MT,S}(G_\bQ)$ and we are done. This completes the proof that~\eqref{eq: Selmer scheme comparison} is an isomorphism.

It remains to prove that~\eqref{diag:etale-motivic_comparison_commuting_diagram} commutes. The commutativity of the triangle
\begin{center}
	\begin{tikzcd}[column sep = small]
		{} & \cY(\bZ_S) \arrow[dl,"j_S^{\dR}"]\arrow[dd,"j_S"'] \\
		\rH^1(G_{\bQ,S}^{\MT},\Pi^{\dR})(\bQ_p) \arrow[dr,equal,"\eqref{eq: Selmer scheme comparison}"',"\sim"] & {} \\
		{} & \rH^1_{f,S}(G_\bQ,\Pi^{\et}(\bQ_p))
	\end{tikzcd}
\end{center}
follows from the fact that the $\bQ_p$-pro-unipotent \'etale path torsor~$\pi_1^{\et,\bQ_p}(Y_{\Qbar};b,y)$ is the \'etale realisation of the motivic path torsor~$\pi_1(Y;b,y)$. The commutativity of the other triangle in~\eqref{diag:etale-motivic_comparison_commuting_diagram} is part of the usual Chabauty--Kim diagram. So it remains to justify why the square
\begin{equation}\label{diag:etale-motivic_comparison_commuting_square}
	\begin{tikzcd}
		\rH^1(G_{\bQ,S}^{\MT},\Pi^{\dR})(\bQ_p) \arrow[r,"\ev_p"]\arrow[d,equal,"\eqref{eq: Selmer scheme comparison}"',"\wr"] & \Pi^{\dR}(\bQ_p) \arrow[d,equals,"\wr"] \\
		\rH^1_{f,S}(G_\bQ,\Pi^{\et}(\bQ_p)) \arrow[r,"\loc_p"] & \rH^1_f(G_p,\Pi^{\et}(\bQ_p))
	\end{tikzcd}
\end{equation}
commutes. For this, we need to recall the definition of both the Bloch--Kato logarithm
\[
\rH^1_f(G_p,\Pi^\et(\bQ_p)) \cong \Pi^{\dR}(\bQ_p)
\]
and the $p$-adic period point
\[
\eta_p^\ur \in U_{\bQ,S}^{\MT}(\bQ_p) \,.
\]
For the former, $\rH^1_f(G_p,\Pi^\et(\bQ_p))$ is the set of isomorphism classes of crystalline $G_p$-equivariant $\Pi^\et$-torsors over~$\bQ_p$. Given such a torsor~$P$, we know that~$\sD_{\dR}(P)\coloneqq\Spec(\sD_{\dR}(\cO(P)))$ is a $\sD_{\dR}(\Pi^\et)\cong\Pi^{\dR}$-torsor in the category of weakly admissible filtered $\varphi$-modules (where the crystalline Frobenius comes from the comparison $\sD_{\dR}\cong\sD_{\cris}$); in particular, $\sD_{\dR}(P)$ carries a Frobenius automorphism and a Hodge filtration on its affine ring. There is a unique Frobenius-invariant path~$\gamma_{\cris}\in\sD_{\dR}(P)(\bQ_p)$ and a unique Hodge-filtered path~$\gamma_{\dR}\in\sD_{\dR}(P)(\bQ_p)$, the latter meaning that the map $\cO(\sD_{\dR}(P))\to\bQ_p$ given by evaluation at~$\gamma_{\dR}$ is compatible with Hodge filtrations. The Bloch--Kato logarithm is then given by
\[
\logBK([P]) = \gamma_{\dR}^{-1}\gamma_{\cris} \,.
\]

For the latter, given a mixed Tate filtered $\varphi$-module in the sense of \cite[Definition~1.1.3]{chatzistamatiou-unver:p-adic_periods}, there are two $\otimes$-functorial splittings of the weight filtration on~$M$. The first splitting~$s_{\dR}$ comes from the fact that the Hodge and weight filtrations on~$M$ are opposed, so
\[
M = \bigoplus_iW_{2i}M\cap F^iM \,.
\]
The second splitting $s_{\cris}$ is the one coming from the decomposition of~$M$ into $\varphi$-eigenspaces. So the automorphism of~$M$ given by
\[
M \xrightarrow[\sim]{s_{\cris}^{-1}} \bigoplus_i \gr^{W}_{2i}M \xrightarrow[\sim]{s_{\dR}} M
\]
is $\otimes$-natural in~$M$, and so defines an element $\eta_p^\ur$ in the Tannaka group of the category of mixed Tate filtered $\varphi$-modules. Since the $\bQ_p$-linear de Rham realisation $M^{\dR}_{\bQ_p}\coloneqq \bQ_p\otimes_{\bQ}M^{\dR}$ of any mixed Tate motive~$M\in\MT(\bZ_S,\bQ)$ is a mixed Tate filtered $\varphi$-module, $\eta_p^\ur$ also defines an element of~$G^{\MT}_{\bQ,S}(\bQ_p)$. This element lies in~$U^{\MT}_{\bQ,S}(\bQ_p)$ since it acts trivially on~$M=\bQ(1)$.

The definition of~$\eta_p^{\ur}$ makes the following clear.

\begin{lem}\label{lem:eta_and_phi}
	Let~$M$ be a mixed Tate filtered $\varphi$-module. Then~$\eta_p^{\ur}\varphi(\eta_p^{\ur})^{-1}$ is the $\bQ_p$-linear automorphism of~$M$ which acts on~$W_{2i}M\cap F^iM$ by multiplication by~$p^{-i}$.
\end{lem}

\begin{rem}
	In fact, $\eta_p^{\ur}$ is uniquely characterised by this property and the fact that it acts trivially on~$\bQ_p(1)$.
\end{rem}

Now we show that~\eqref{diag:etale-motivic_comparison_commuting_square} commutes. Let~$P$ be a $\Pi$-torsor over~$\bQ_p$ in the category~$\MT(\bZ_S,\bQ)$, representing an element of~$\rH^1(G^{\MT}_{\bQ,S},\Pi^{\dR})(\bQ_p)$. We let~$P^{\dR}$ and~$P^{\et}$ denote the torsors obtained by applying the $\bQ_p$-linear de Rham and \'etale realisation functors, respectively, and then base-changing along the multiplication map~$\bQ_p\otimes_{\bQ}\bQ_p\to\bQ_p$. So $[P^{\et}]\in\rH^1_{f,S}(G_\bQ,\Pi^{\et}(\bQ_p))$ is the element corresponding to~$[P]\in\rH^1(G^{\MT}_{\bQ,S},\Pi^{\dR})(\bQ_p)$ under~\eqref{eq: Selmer scheme comparison}, and~$P^{\dR}$ can equivalently be described as the $\bQ$-linear de Rham realisation of~$P$ (which is a $G^{\MT}_{\bQ,S}$-equivariant $\Pi^{\dR}$-torsor over~$\bQ_p$). The usual comparison isomorphisms between realisations imply that we have a canonical isomorphism
\[
\sD_{\dR}(P^\et)\cong P^{\dR}
\]
of~$\sD_{\dR}(\Pi^\et)\cong\Pi^{\dR}$-torsors in the category of mixed Tate $\varphi$-modules. By mild abuse of notation, we also write~$\gamma_{\cris}\in P^{\dR}(\bQ_p)$ and~$\gamma_{\dR}\in P^{\dR}(\bQ_p)$ for the unique Frobenius-invariant and Hodge-filtered paths, respectively.

\begin{lem}
	$\gamma_{\dR}\in P^{\dR}(\bQ_p)$ is the unique element such that the associated cocycle
	\[
	\xi_{\gamma_{\dR}}\colon G^{\MT}_{\bQ,S,\bQ_p} \to \Pi^{\dR}_{\bQ_p} \quad\text{given by}\quad \xi_{\gamma_{\dR}}(\sigma) \coloneqq \gamma_{\dR}^{-1}\sigma(\gamma_{\dR})
	\]
	restricts trivially to~$\bG_{m,\bQ_p}$. Additionally, we have
	\[
	(\eta_p^{\ur})^{-1}(\gamma_{\dR}) = \gamma_{\cris} \,.
	\]
	\begin{proof}
		The map $\ev_{\gamma_{\dR}}\colon\cO(P^{\dR})\to\bQ_p$ given by evaluation at~$\gamma_{\dR}$ is compatible with Hodge filtrations, and also (trivially) compatible with weight filtrations. Hence~$\ev_{\gamma_{\dR}}$ is projection onto the $0$th factor in the grading
		\[
		\cO(P^{\dR}) = \bigoplus_iW_{2i}\cO(P^{\dR})\cap F^i\cO(P^{\dR})
		\]
		associated to the torus~$\bG_{m,\bQ_p}$. So~$\gamma_{\dR}$ is $\bG_{m,\bQ_p}$-invariant, and hence its associated cocycle vanishes on~$\bG_{m,\bQ_p}$.
		
		\Cref{lem:eta_and_phi} implies that $\eta_p^{\ur}\varphi(\eta_p^{\ur})^{-1}$ is a $\bQ_p$-point of this torus, so fixes~$\gamma_{\dR}$. This implies that
		\[
		(\eta_p^{\ur})^{-1}(\gamma_{\dR}) = \varphi(\eta_p^{\ur})^{-1}(\gamma_{\dR})
		\]
		is Frobenius-invariant, and hence
		\[
		(\eta_p^{\ur})^{-1}(\gamma_{\dR}) = \gamma_{\cris} \,. \qedhere
		\]
	\end{proof}
\end{lem}

Now using the above lemma, we see that the image of~$[P^{\et}|_{G_p}]\in\rH^1_f(G_p,\Pi^{\et}(\bQ_p))$ under the Bloch--Kato logarithm is
\[
\gamma_{\dR}^{-1}\gamma_{\cris} = \gamma_{\dR}^{-1}(\eta_p^{\ur})^{-1}(\gamma_{\dR}) = \xi_{\gamma_{\dR}}((\eta_p^{\ur})^{-1}) \,,
\]
which is equal to the image of~$[P]\in\rH^1(G^{\MT}_{\bQ,S},\Pi^{\dR}(\bQ_p))=\rZ^1(U^{\MT}_{\bQ,S},\Pi^{\dR}(\bQ_p))^{\bG_m}$ under the map~$\ev_p$ given by evaluation at~$(\eta_p^{\ur})^{-1}$. This concludes the proof of Theorem~\ref{thm:comparison_selmer_schemes}.\qed

%% file: refined.tex
\section{Refined Selmer scheme of the thrice-punctured line}
\label{sec: refined}

Let $S$ be a finite set of primes and let $\cY/\bZ_S$ be the thrice-punctured line with generic fibre $Y/\bQ$. Let $p \not\in S$ be prime and let $U = \pi_1^{\et,\bQ_p}(Y_{\overline{\bQ}},b)$ be the $\bQ_p$-pro-unipotent étale fundamental group of $Y$ at an $S$-integral point or tangent vector~$b$. In this section we want to describe the refined Selmer scheme $\Sel_{S,\Pi}^{\min}(\cY)$ of the thrice-punctured line, for any $G_{\bQ}$-equivariant quotient $\Pi$ of~$U$ dominating the abelianisation:
\[U \twoheadrightarrow \Pi \twoheadrightarrow U^{\ab}.\]

Assuming $2 \in S$ (otherwise, $\Sel_{S,\Pi}^{\min}(\cY)$ is empty), we find that the refined Selmer scheme can be written as a union of closed subschemes $\Sel_{S,\Pi}^{\Sigma}(\cY)$ over all \emph{refinement conditions} $\Sigma \in \{0,1,\infty\}^S$:
\begin{equation}
	\label{eq: refined Selmer scheme union}
	\Sel_{S,\Pi}^{\min}(\cY) = \bigcup_{\Sigma} \Sel_{S,\Pi}^{\Sigma}(\cY)
\end{equation}
(\Cref{thm: refined Selmer scheme} below). This induces a similar description of the refined Chabauty--Kim locus $\cY(\bZ_p)_{S,\Pi}^{\min}$:
\begin{equation}
	\label{eq: refined CK locus union}
	\cY(\bZ_p)_{S,\Pi}^{\min} = \bigcup_{\Sigma} \cY(\bZ_p)_{S,\Pi}^{\Sigma}.
\end{equation}

Each refinement condition $\Sigma =(\Sigma_{\ell})_{\ell \in S} \in \{0,1,\infty\}^S$ can be seen as a collection of mod-$\ell$ congruence conditions on $S$-integral points for $\ell \in S$. Namely, if we denote by $\cY(\bZ_S)_{\Sigma}$ the set of those $S$-integral points whose image under the mod-$\ell$ reduction map
\[ \mathrm{red}_{\ell}\colon \cY(\bQ_{\ell}) \subseteq \bP^1(\bQ_{\ell}) = \bP^1(\bZ_{\ell}) \twoheadrightarrow \bP^1(\bF_{\ell}) \]
lies in $\cY(\bF_{\ell}) \cup \{\Sigma_{\ell}\}$ for all $\ell \in S$, then we can write
\begin{equation}
	\label{eq: S-integral points union}
	\cY(\bZ_S) = \bigcup_{\Sigma} \cY(\bZ_S)_{\Sigma},
\end{equation}
and the inclusion $\cY(\bZ_S) \subseteq \cY(\bZ_p)_{S,\Pi}^{\min}$ of $S$-integral points~\eqref{eq: S-integral points union} in the refined Chabauty--Kim locus~\eqref{eq: refined CK locus union} can be refined to inclusions
\[ \cY(\bZ_S)_{\Sigma} \subseteq \cY(\bZ_p)_{S,\Pi}^{\Sigma} \]
for each refinement condition $\Sigma \in \{0,1,\infty\}^S$.

In order to prove these facts, we show a general statement to the effect that the equations defining the refined subscheme $\Sel_{S,\Pi}^{\min}(\cY)$ inside the full (unrefined) Selmer scheme $\rH^1_{f,S}(G_{\bQ}, \Pi)$ all come from the abelian (or depth~1) Selmer scheme, i.e. they are pulled back along the canonical map
\[ \rH^1_{f,S}(G_{\bQ}, \Pi) \to  \rH^1_{f,S}(G_{\bQ}, U^{\ab}) \]
induced by the abelianisation map $\Pi \twoheadrightarrow U^{\ab}$
(\Cref{thm:refinement conditions intermediate quotient} below). We are then reduced to the case $\Pi = U^{\ab}$ where we already know how to describe the refined Selmer scheme thanks to \cite{BBKLMQSX}. We end this section by noting in §\ref{sec: change of fundamental group quotient} and §\ref{sec: S3-action} how refined Chabauty--Kim loci interact with a change of the fundamental group quotient and with the $S_3$-action on the thrice-punctured line, which will be useful for determining these loci in concrete examples in \Cref{sec:calculations}.

Recall (\Cref{def: refined Selmer scheme}) that the \emph{refined Selmer scheme} $\Sel_{S,\Pi}^{\min}(\cY)$ of~$Y$ with respect to $\Pi$ is the scheme parametrising those continuous cohomology classes in $\rH^1(G_{\bQ}, \Pi)$ whose image under the localisation map
\[ \loc_\ell\colon \rH^1(G_{\bQ}, \Pi) \to \rH^1(G_\ell, \Pi) \]
is contained in the Zariski closure of $j_\ell(Y(\bQ_\ell))$ for $\ell \in S$, respectively of $j_{\ell}(\cY(\bZ_\ell))$ for $\ell \not\in S$, where $j_\ell$ is the local Kummer map
\[ j_\ell \colon Y(\bQ_\ell) \to \rH^1(G_\ell, \Pi(\bQ_p)). \]
We will now determine these local Selmer images. 

\subsection{Local conditions at primes outside $S$}

Let us first consider the primes $\ell \not\in S$. The prime $\ell = 2$ has a special role: since $\cY(\bF_2) = \bP^1(\bF_2) \smallsetminus \{0,1,\infty\} = \emptyset$, we have $\cY(\bZ_2) = \emptyset$. So in particular 
\[ j_2(\cY(\bZ_2)) = \emptyset,\]
which leads to the local condition at~2 being unsatisfiable if $2 \not \in S$:

\begin{prop}
	\label{thm: empty refined Selmer scheme}
	If $2 \not \in S$, then the refined Selmer scheme is empty:
	\[ \Sel_{S,\Pi}^{\min}(\cY) = \emptyset. \]
\end{prop}

If $\ell \neq 2$, then we always have $-1 \in \cY(\bZ_\ell)$, so that in particular $j_\ell(\cY(\bZ_\ell))$ is non-empty. For any $y \in \cY(\bZ_\ell)$, the étale path torsor $\pi_1^{\et,\bQ_p}(Y_{\Qbar}; b,y)$ is unramified ($\ell \neq p$) respectively crystalline ($\ell = p$). 
For $\ell \neq p$, the only unramified class in $\rH^1(G_\ell, \Pi(\bQ_p))$ is the trivial one \cite[Proof of Cor.~0.3 in §2]{kim-tamagawa}, i.e.\ we have 
\[ j_\ell(\cY(\bZ_\ell)) = \{\ast\}. \]
At $\ell = p$, the Zariski closure of $j(\cY(\bZ_p))$ in $\rH^1(G_p, \Pi)$ is precisely the subspace of crystalline classes:
\[ j_p(\cY(\bZ_p))^{\Zar} = \rH^1_f(G_p, \Pi). \]
This follows from the Bloch--Kato isomorphism $\logBK\colon \rH^1_f(G_p, \Pi) \cong F^0\backslash\Pi^{\dR}$ and the fact that the de Rham Kummer map
\[ j_p^{\dR} \colon \cY(\bZ_p) \to F^0\backslash\Pi^{\dR}(\bQ_p) \]
has Zariski-dense image \cite[Theorem~1]{kim:albanese}.

In conclusion, if $2 \in S$, then the local conditions at $\ell \not\in S$ defining the refined Selmer scheme $\Sel_{S,\Pi}^{\min}(\cY)$ are the same as those defining the unrefined Selmer scheme $\rH^1_{f,S}(G_\bQ,\Pi)$ (cf. \Cref{rem: unrefined CK}). We have therefore an inclusion
\[ \Sel_{S,\Pi}^{\min}(\cY) \subseteq \rH^1_{f,S}(G_\bQ,\Pi), \]
with $\Sel_{S,\Pi}^{\min}(\cY)$ defined as a closed subscheme by the remaining local conditions at $\ell \in S$. We now determine those local conditions at primes contained in $S$.

\subsection{Reduction to abelianised fundamental group}
\label{sec: reduction abelianisation}

\begin{lem}
	\label{thm: reduction to abelian cohomology}
	For $\ell \neq p$, the morphism $\rH^1(G_\ell, \Pi) \to \rH^1(G_\ell, U^{\ab})$ is an isomorphism.
\end{lem}

\begin{proof}
	We have to show that the map $\rH^1(G_\ell, \Pi(R)) \to \rH^1(G_\ell, U^{\ab}(R))$ is an isomorphism for any $\bQ_p$-algebra~$R$.
	Let $\Pi_n$ denote the $n$th quotient of $\Pi$ along the lower central series. Let $V_n$ be the kernel of the natural projection $\Pi_n \to \Pi_{n-1}$, so that we have a short exact sequence of $\bQ_p$-group schemes
	\begin{equation}
		\label{eq: central extension of pi}
		1 \to V_n \to \Pi_n \to \Pi_{n-1} \to 1.
	\end{equation}
	The abelianisation $\Pi^{\ab} = U^{\ab} = V_1$ is a vector group whose $\bQ_p$-points as a Galois representation are given by $\bQ_p(1) \oplus \bQ_p(1)$. In particular, $V_1$ is pure of weight~$-2$. For larger~$n$, we have a surjection 
	\[ V_1(\bQ_p)^{\otimes n} \twoheadrightarrow V_n(\bQ_p), \quad x_1 \otimes \ldots \otimes x_n \mapsto [x_1,[x_2,[\ldots,x_n ]]],\]
	so $V_n(\bQ_p)$ is pure of weight $-2n$ and therefore a direct sum of copies of $\bQ_p(n)$:
	\[ V_n(\bQ_p) = \bigoplus \bQ_p(n). \]
	Since $\ell \neq p$, we have
	\[ \rH^1(G_\ell, \bQ_p(n)) = 0 \quad \text{and} \quad \rH^2(G_\ell, \bQ_p(n)) = 0 \]
	for all $n \geq 2$ by \cite[Prop.~(7.3.10) and remark]{nsw}, and hence $\rH^1(G_\ell, V_n(\bQ_p)) = \rH^2(G_\ell, V_n(\bQ_p)) = 0$. Since $\rH^i(G_{\ell}, V_n(R)) = \rH^i(G_{\ell}, V_n(\bQ_p)) \otimes_{\bQ_p} R$, the same holds on the level of $R$-valued points:
	\[ \rH^1(G_\ell, V_n(R)) = 0 \quad \text{and} \quad \rH^2(G_\ell, V_n(R)) = 0. \]
	The long exact cohomology sequence in nonabelian Galois cohomology for the central extension~\eqref{eq: central extension of pi} now implies that the map
	\begin{equation}
		\label{eq: iso on H1 for quotients}
		(p_n)_*\colon \rH^1(G_\ell, \Pi_n(R)) \overset{\cong}{\to} \rH^1(G_\ell, \Pi_{n-1}(R))
	\end{equation}
	is an isomorphism for $n \geq 2$. Composing these isomorphisms we find that the map $\rH^1(G_\ell,\Pi_n(R)) \to \rH^1(G_\ell, U^{\ab}(R))$ induced by the abelianisation map $\Pi_n \twoheadrightarrow \Pi_1 = U^{\ab}$ is an isomorphism for all $n \geq 1$. Finally, to go from all $\Pi_n$ to their limit~$\Pi$ we use the isomorphism
	\[ \rH^1(G_{\ell}, \Pi(R)) \cong \varprojlim_n \rH^1(G_{\ell}, \Pi_n(R)), \]
	which is shown in \cite[Lemma~4.5]{betts:motivic-anabelian-geometry} for $R = \bQ_p$ but the same proof works for any $\bQ_p$-algebra.
\end{proof}

\begin{cor}
	\label{thm:refinement conditions intermediate quotient}
	Let $\Pi$ be a $G_{\bQ}$-equivariant intermediate quotient of $U \twoheadrightarrow U^{\ab}$. Then the refined Selmer scheme $\Sel_{S,\Pi}^{\min}(\cY) \subseteq \rH^1_{f,S}(G_{\bQ}, \Pi)$ is the inverse image of the refined Selmer scheme $\Sel_{S,U^{\ab}}^{\min}(\cY)$ for the abelianisation $U^{\ab}$:
	\[
	\begin{tikzcd}
		\Sel_{S,\Pi}^{\min}(\cY) \dar \rar[symbol=\subseteq] \arrow[dr, phantom, "\scalebox{1.5}{$\lrcorner$}" , very near start, color=black] & \rH^1_{f,S}(G_{\bQ}, \Pi) \dar\\
		\Sel_{S,U^{\ab}}^{\min}(\cY) \rar[symbol=\subseteq] & \rH^1_{f,S}(G_{\bQ}, U^{\ab}).
	\end{tikzcd}
	\]
\end{cor}

\begin{proof}
	By definition, the local conditions at primes $\ell \in S$ defining the refined Selmer scheme are determined by the Zariski closure of the image of the $\ell$-adic Kummer map $j_{\ell}\colon Y(\bQ_{\ell}) \to \rH^1(G_{\ell}, \Pi(\bQ_p))$. By \Cref{thm: reduction to abelian cohomology}, these local conditions can be checked on the level of the abelianised fundamental group.
\end{proof}

\subsection{Refined Selmer scheme for the abelianised fundamental group}
\label{sec: refined abelian Selmer scheme}

The refined Selmer scheme $\Sel_{S,U^{\ab}}^{\min}(\cY)$ for the abelianised fundamental group was previously studied in~\cite{BBKLMQSX}, so we only give a brief summary of the facts we need here. We have $U^{\ab}(\bQ_p) = \bQ_p(1) \oplus \bQ_p(1)$ as a Galois representation, with the two factors corresponding to loops around~$0$ and~$1$, respectively. By Kummer theory, we have
\[ \rH^1(G_{\ell}, \bQ_p(1)) = \bigl(\varprojlim_n \bQ_{\ell}^\times/(\bQ_{\ell}^\times)^{p^n}\bigr) \otimes_{\bZ_p} \bQ_p \cong \bQ_p, \]
for $\ell \neq p$, with the last isomorphism induced by the $\ell$-adic valuation map $v_{\ell}\colon\bQ_{\ell}^\times \to \bZ$. As a consequence, we get 
\[ \rH^1(G_{\ell}, U^{\ab}) \cong \bA_{\bQ_p}^2 \]
for $\ell \neq p$. We choose our coordinates $(x_{\ell}, y_{\ell})$ on $\rH^1(G_{\ell}, U^{\ab})$ to be those from the canonical isomorphism above, except for adding a minus sign in the $y_{\ell}$-coordinate. 
Then we get following explicit description of the local Kummer map $j_{\ell}\colon Y(\bQ_{\ell}) \to \rH^1(G_{\ell}, U^{\ab}(\bQ_p))$ for the quotient $U^{\ab}$.

\begin{lem}
	\label{thm: local Kummer in coordinates}
	For $\ell \neq p$, the local Kummer map $j_{\ell}$ is given as $z \mapsto (v_\ell(z), -v_\ell(1-z))$ in the coordinates $(x_{\ell}, y_{\ell})$. \qed
\end{lem}

\begin{rem}
	\label{rem: sign difference in local Selmer scheme coordinates}
	Adding a minus sign in the second coordinate means that our choice of coordinates differs from that in \cite{BBKLMQSX}. However it agrees with the choice made in \cite{CDC:polylog1}, which we will adopt for the discussion in \Cref{sec:calculations}, where we make explicit calculations for small $S$ working in the polylogarithmic quotient.
\end{rem}

Taking the choice of coordinates into account, we showed in \cite[Lemma\,2.9]{BBKLMQSX} that the Zariski closure of the image $j_{\ell}(Y(\bQ_{\ell}))$ equals the union $R_{0} \cup R_{1} \cup R_{\infty}$ of the three following linear subspaces: 
\begin{align*}
	R_{0} &= \left\{(x, y) \in \bA^2 \; \middle\vert \; y = 0\right\}\!,\\
	R_{1} &= \left\{(x, y) \in \bA^2 \; \middle\vert \; x = 0 \right\}\!,\\
	R_{\infty} &= \left\{(x, y) \in \bA^2 \; \middle\vert \; x + y = 0 \right\}\!,
\end{align*}
where the naming convention comes from the fact that these linear conditions correspond to the three cusps.

Again by Kummer theory, the unrefined Selmer scheme $\rH^1_{f,S}(G_{\bQ}, U^{\ab})$ is an affine space of the form $\bA_{\bQ_p}^S \times \bA_{\bQ_p}^S$, as discussed in \cite[§2.4]{BBKLMQSX}.\footnote{The Selmer scheme $\rH^1_{f,S}(G_{\bQ}, U^{\ab})$ is denoted $\Sel_{S,1}$ in \cite{BBKLMQSX}.} We choose coordinates $((x_{\ell})_{\ell \in S}$, $(y_{\ell})_{\ell \in S})$ such that the global Kummer map 
\[j_S \colon \mathcal{Y}(\bZ_S) \to \rH^1_{f,S}(G_{\bQ}, U^{\ab}) = \bA_{\bQ_p}^S \times \bA_{\bQ_p}^S \] 
is given by
\begin{equation}
	\label{eq:global Kummer in coordinates}
	z\mapsto ((v_{\ell}(z))_{\ell \in S}, (-v_{\ell}(1-z))_{\ell \in S}).
\end{equation}
The coordinates $y_{\ell}$ thus differ again by a sign from the ones used in \cite{BBKLMQSX}.

The localisation map 
\[ \loc_{\ell}\colon \rH^1_{f,S}(G_{\bQ}, U^{\ab}) \to \rH^1(G_{\ell}, U^{\ab}) \]
for $\ell \neq p$ is now simply the projection $\bA^S \times \bA^S \to \bA^2$ onto the $(x_{\ell},y_{\ell})$-components. A point $((x_{\ell})_{\ell \in S}, (y_{\ell})_{\ell \in S})$ of the Selmer scheme is therefore contained in the \emph{refined} Selmer scheme if and only if for all $\ell \in S$ we have $(x_{\ell}, y_{\ell}) \in R_0 \cup R_1 \cup R_{\infty}$. This yields the following explicit description of the refined Selmer scheme.

\begin{prop}
	\label{thm: refined selmer scheme ab}
	Assume $2 \in S$. Then the refined Selmer scheme $\Sel_{S,U^{\ab}}^{\min}(\cY)$ for the abelianised fundamental group is the subscheme of $\bA_{\bQ_p}^S \times \bA_{\bQ_p}^S$ defined by the equations
	\[ x_{\ell} y_{\ell} (x_{\ell} + y_{\ell}) = 0 \quad \text{for all $\ell \in S$}. \]
	It can be expressed as the union of the following $3^{\#S}$ linear subspaces
	\begin{equation}
		\label{eq: Sel ab decomposition}
		\Sel_{S,U^{\ab}}^{\min}(\cY) = \bigcup_{\Sigma} \Sel_{S,U^{\ab}}^{\Sigma}(\cY),
	\end{equation}
	where 
	\[ \Sel_{S,U^{\ab}}^{\Sigma}(\cY) \coloneqq \left\{((x_{\ell})_{\ell \in S},(y_{\ell})_{\ell \in S})  \;\middle\vert\; (x_{\ell},y_{\ell}) \in R_{\Sigma_{\ell}} \text{ for } \ell \in S \right\} \subseteq \bA_{\bQ_p}^S \times \bA_{\bQ_p}^S \]
	and $\Sigma = (\Sigma_{\ell})_{\ell \in S}\in \{0,1,\infty\}^S$. \qed
\end{prop}

\subsection{Refined Selmer scheme for general fundamental group quotients}

Via \Cref{thm:refinement conditions intermediate quotient}, we obtain a similar description of the refined Selmer scheme for larger quotients. Assume $2 \in S$ and let $U \twoheadrightarrow \Pi \twoheadrightarrow U^{\ab}$ be any $G_{\bQ}$-equivariant intermediate quotient.

\begin{defn}
	For each $\Sigma = (\Sigma_{\ell})_{\ell \in S} \in \{0,1,\infty\}^S$ define $\Sel_{S,\Pi}^{\Sigma}(\cY)$ as the pullback of $\Sel_{S,U^{\ab}}^{\Sigma}(\cY)$ from the abelian Selmer scheme:
	\[
	\begin{tikzcd}
		\Sel_{S,\Pi}^{\Sigma}(\cY) \dar \rar[symbol=\subseteq] \arrow[dr, phantom, "\scalebox{1.5}{$\lrcorner$}" , very near start, color=black] & \rH^1_{f,S}(G_{\bQ}, \Pi) \dar\\
		\Sel_{S,U^{\ab}}^{\Sigma}(\cY) \rar[symbol=\subseteq] & \rH^1_{f,S}(G_{\bQ}, U^{\ab}).
	\end{tikzcd}
	\]
	We call $\Sel_{S,\Pi}^{\Sigma}(\cY)$ the refined Selmer scheme for the \emph{refinement condition}~$\Sigma$. 
	
	Moreover, define the refined Chabauty--Kim locus $\cY(\bZ_p)_{S,\Pi}^{\Sigma} \subseteq \cY(\bZ_p)$ for the refinement condition~$\Sigma$ as the set of all $z \in \cY(\bZ_p)$ whose image under the $p$-adic Kummer map $j_p$ is contained in the scheme-theoretic image of $\Sel_{S,\Pi}^{\Sigma}(\cY)$ under the localisation map $\loc_p$.
\end{defn}

Then \Cref{thm:refinement conditions intermediate quotient} yields the following:

\begin{cor}
	\label{thm: refined Selmer scheme}
	The refined Selmer scheme $\Sel_{S,\Pi}^{\min}(\cY)$ can be written as a union
	\[ \Sel_{S,\Pi}^{\min}(\cY) = \bigcup_{\Sigma} \Sel_{S,\Pi}^{\Sigma}(\cY) \]
	over all refinement conditions $\Sigma \in \{0,1,\infty\}^S$. Similarly, the refined Chabauty--Kim locus $\cY(\bZ_p)_{S,\Pi}^{\min}$ can be written as a union
	\[ \cY(\ZZ_p)_{S, \Pi}^{\min} = \bigcup_{\Sigma}\cY(\ZZ_p)_{S, \Pi}^{\Sigma}. \]
\end{cor}

\begin{rem}
	\label{rem: refinement conditions S-integral points}
	As discussed at the beginning of this section, the presentation of the refined Selmer scheme and the refined Chabauty--Kim locus as a union over refinement conditions~$\Sigma = (\Sigma_{\ell})_{\ell \in S} \in \{0,1,\infty\}^S$ corresponds to writing the $S$-integral points as a union
	\[ \cY(\bZ_S) = \bigcup_{\Sigma} \cY(\bZ_S)_{\Sigma}, \]
	with $\cY(\bZ_S)_{\Sigma}$ denoting the set of those $S$-integral points whose mod-$\ell$ reduction lies in $\cY(\bF_{\ell}) \cup \{\Sigma_{\ell}\}$ for all $\ell \in S$. We have the inclusion
	\[ \cY(\bZ_S)_{\Sigma} \subseteq \cY(\bZ_p)_{S,\Pi}^{\Sigma} \]
	for all $\Sigma \in \{0,1,\infty\}^S$.
\end{rem}

\begin{rem}
	We have worked in this section with the $p$-adic étale Selmer scheme rather than the motivic Selmer scheme because the refinement conditions are phrased in terms of the localisation maps
	\[ \loc_{\ell} \colon \rH^1_{f,S}(G_{\bQ}, \Pi) \to \rH^1(G_\ell, \Pi), \]
	and these are only defined over $\bQ_p$. However, the equations which we computed for the refined Selmer scheme have coefficients in $\bQ$ and are independent of~$p$, so they could be used a posteriori to define a refined Selmer subscheme of the motivic Selmer scheme $\rH^1(G_{\bQ,S}^{\MT},\Pi^{\dR})$ over $\bQ$. It would be interesting to find a motivic definition of the refined Selmer scheme that is independent of~$p$ by construction.
\end{rem}

For later use we note two functoriality properties of refined Chabauty--Kim loci: their behaviour with respect to a change of the fundamental group quotient, and their interaction with the natural $S_3$-action on the thrice-punctured line.

\subsection{Change of fundamental group quotient}
\label{sec: change of fundamental group quotient}

Smaller quotients of the fundamental group result in larger corresponding Chabauty--Kim loci:

\begin{lem}
	\label{thm: change of fundamental group quotient}
	Let $U \twoheadrightarrow \Pi_1 \twoheadrightarrow \Pi_2 \twoheadrightarrow U^{\ab}$ be two $G_{\bQ}$-equivariant quotients of $U$ dominating the abelianisation. Then we have the inclusion $\cY(\ZZ_p)_{S, \Pi_1}^{\min} \subseteq \cY(\ZZ_p)_{S, \Pi_2}^{\min}$. More precisely, for each refinement condition $\Sigma \in \{0,1,\infty\}^S$, we have the inclusion $\cY(\ZZ_p)_{S, \Pi_1}^{\Sigma} \subseteq \cY(\ZZ_p)_{S, \Pi_2}^{\Sigma}$.
\end{lem}

\begin{proof}
	The quotient map $\Pi_1 \twoheadrightarrow \Pi_2$ induces the vertical maps in the commutative diagram
	\[
	\begin{tikzcd}
		\Sel_{S,\Pi_1}^{\Sigma}(\cY) \dar \rar["\loc_p"] & \rH^1_f(G_p, \Pi_1) \dar \\
		\Sel_{S,\Pi_2}^{\Sigma}(\cY) \rar["\loc_p"] & \rH^1_f(G_p, \Pi_2).
	\end{tikzcd}
	\]
	The inclusion $\cY(\ZZ_p)_{S, \Pi_1}^{\Sigma} \subseteq \cY(\ZZ_p)_{S, \Pi_2}^{\Sigma}$ now follows from the definitions.
\end{proof}

\subsection{The \texorpdfstring{$S_3$}{S\textthreeinferior}-action}
\label{sec: S3-action}

The thrice-punctured line $\cY = \bP^1 \smallsetminus \{0,1,\infty\}$ carries a natural $S_3$-action given by the Möbius transformations
\[ z, \quad \frac1{z}, \quad 1-z, \quad \frac1{1-z}, \quad \frac{z-1}{z}, \quad \frac{z}{z-1}, \]
which permute the three cusps $\{0,1,\infty\}$. (We identify $S_3$ with the symmetric group $S_{\{0,1,\infty\}}$.) This action can be exploited for calculating refined Chabauty--Kim loci. 

For $1 \leq n \leq \infty$, denote by $U_n$ the depth-$n$ quotient of the $\bQ_p$-pro-unipotent étale fundamental group~$U$ of the thrice-punctured line~$\cY$, and let 
\[ \cY(\bZ_p)_{S,n}^{\min} = \bigcup_{\Sigma} \cY(\bZ_p)_{S,n}^{\Sigma} \]
be the associated refined Chabauty--Kim locus, written as a union over refinement conditions $\Sigma \in \{0,1,\infty\}^S$ as in \Cref{thm: refined Selmer scheme}.

\begin{lem}
	\label{thm: S3-action}
	The refined Chabauty--Kim locus $\cY(\bZ_p)_{S,n}^{\min}$ is stable under the $S_3$-action. More precisely, the $S_3$-action permutes the subsets $\cY(\bZ_p)_{S,n}^{\Sigma}$ according to the action of $S_3 = S_{\{0,1,\infty\}}$ on $\{0,1,\infty\}^S$, in the sense that
	\[ \sigma(\cY(\bZ_p)_{S,n}^{\Sigma}) = \cY(\bZ_p)_{S,n}^{\sigma(\Sigma)} \]
	for all $\Sigma \in \{0,1,\infty\}^S$ and $\sigma \in S_3$.
\end{lem}

\begin{proof}
	The first statement is proved in \cite[Corollary~2.15(i)]{BBKLMQSX} as a consequence of more general functoriality properties of Chabauty--Kim loci. The second statement is proved in \cite[Corollary~2.15(ii)]{BBKLMQSX}. It is stated there only for $n \leq 2$ because the sets $\cY(\bZ_p)_{S,n}^{\Sigma}$, had only been defined in depth at most~$2$ in loc.\ cit., but the same proof works for arbitrary~$n \leq \infty$.
\end{proof}

\begin{rem}
	\label{rem: S3-action}
	\Cref{thm: S3-action} does not hold for arbitrary fundamental group quotients, i.e., if $U \twoheadrightarrow \Pi$ is any $G_{\bQ}$-equivariant quotient and $\sigma$ an automorphism of $\cY/\bZ_S$, then the refined Chabauty--Kim locus $\cY(\bZ_p)_{S,\Pi}^{\min}$ is in general not stable under~$\sigma$. The fact that the loci for the descending central series quotients $U_n$ are stable under~$\sigma$ relies on those quotients being $\sigma$-equivariant, in a suitable sense that takes different base points into account. However, if for example $\Pi = U_{\PL}$ is the polylogarithmic quotient studied in \cite{CDC:polylog1}, the corresponding locus $\cY(\bZ_p)_{S,\PL}^{\min}$ will be stable only under the involution $z \mapsto 1/z$ but not under all automorphisms of the thrice-punctured line. This is the motivation for defining an $S_3$-symmetrised (unrefined) Chabauty--Kim locus in \cite[§5.2]{CDC:polylog1}.
\end{rem}

%% file: calculations.tex
\section{Determination of Chabauty--Kim loci}
\label{sec:calculations}

The goal of this section is to prove that the thrice-punctured line satisfies the refined Kim's Conjecture for $S = \{2\}$ and all odd primes~$p$ (\Cref{thm:main_kim}). The main steps are the following:

\begin{enumerate}
	\item We obtain from \cite{CDC:polylog1} a formula for the localisation map 
	\[ \loc_p\colon \rH^1_{f,S}(G_{\bQ}, U_{\PL}) \to \rH^1_f(G_p, U_{\PL}), \]
	where $U_{\PL}$ is the \emph{polylogarithmic quotient} of the fundamental group (\Cref{def: polylog quotient}). This formula holds for any set of primes~$S$ and any prime $p \not\in S$. Corwin--Dan-Cohen work in the motivic setting, so we use our motivic-étale comparison theorem from \Cref{sec: CK comparison} to transfer their results to the usual étale setting. 
	
	\item From \Cref{sec: refined} we know the equations cutting out the \emph{refined} Selmer scheme inside the full Selmer scheme of $U_{\PL}$.
	
	\item Specialising to $S= \{2\}$, we find functions that vanish on the scheme-theoretic image of the refined Selmer scheme $\Sel_{\{2\}, U_{\PL}}^{(1)}(\cY)$ under the $p$-adic localisation map, for the particular refinement condition $\Sigma = (1)$. As a result, we find infinitely many functions vanishing on the associated refined Chabauty--Kim locus $\cY(\bZ_p)_{\{2\}, U_{\PL}}^{(1)}$:
	\[ \log(z) = 0, \quad \Li_k(z) = 0 \quad \text{for $k \geq 2$ even}. \]
	
	\item We then show that the only common solution in $\cY(\bZ_p)$ of these equations is the $\{2\}$-integral point $z = -1$.
	
	\item By exploiting the $S_3$-action, we show that those calculations for the particular refinement condition $\Sigma = (1)$ are enough to determine the complete refined Chabauty--Kim locus, thus confirming that it consists exactly of the $\{2\}$-integral points $\{2,-1,1/2\}$.
\end{enumerate}

In §\ref{sec: proof of unrefined kim conjecture} we show how the same methods can be used to prove the classical (non-refined) Kim's Conjecture in the case $S= \emptyset$ (\Cref{thm:main_unrefined}). This had previously been proved for half the primes in \cite[§6]{BDCKW} in depth~2; we can prove it for all primes by going into higher depth. Finally, we end by showing in §\ref{sec: higher genus} that certain higher genus curves also satisfy a Kim-like conjecture, strong enough to deduce the Selmer Section Conjecture for those curves (\Cref{thm:main_chabauty}).

\subsection{The polylogarithmic quotient}
\label{sec:polylog quotient}

Let $S$ be any finite set of primes, fix a prime $p \not\in S$ and let $U = \pi_1^{\et,\bQ_p}(Y_{\overline{\bQ}}, b)$ the $\bQ_p$-pro-unipotent étale fundamental group of $Y$ with respect to the tangential base point $b = \vec{1}_0$. Here, $\vec{1}_0$ denotes the tangent vector at~$0$ corresponding to~$1$ under the canonical identification $T_0 \bP^1 \cong \bA^1$. The inclusion $Y \hookrightarrow \bG_m$ induces a $G_{\bQ}$-equivariant homomorphism of fundamental groups $U \to \bQ_p(1)$. 

\begin{defn}
	\label{def: polylog quotient}
	The \emph{polylogarithmic quotient} of~$U$ is defined as
	\[ U_{\PL} = U/[N,N], \]
	where $N$ is the kernel of the homomorphism $U \to \bQ_p(1)$ induced by the inclusion $Y \hookrightarrow \bG_m$.
\end{defn}

The de Rham variant of the polylogarithmic quotient was studied extensively in \cite{CDC:polylog1, CDC:polylog2}. In order to define it, one can start with the full $\bQ$-pro-unipotent de Rham fundamental group~$U^{\dR}$, defined as the Tannaka group of the category of unipotent $\bQ$-vector bundles with connection on~$Y$, and define its polylogarithmic quotient $U^{\dR} \twoheadrightarrow U_{\PL}^{\dR}$ in the same way as $U_{\PL}$ is defined as a quotient of the étale fundamental group~$U$ above. 

\begin{rem}
	\label{polylog quotient motivic}
	The polylogarithmic quotient is of motivic origin: let $\pi_1(Y,b)$ the $\bQ$-prounipotent motivic fundamental group in $\MT(\bZ_S, \bQ)$. Then a definition similar to \Cref{def: polylog quotient} defines the polylogarithmic quotient $\pi_1(Y,b) \twoheadrightarrow \pi_1(Y,b)_{\PL}$ in the category $\MT(\bZ_S, \bQ)$. It specialises to $U_{\PL}$ under the $p$-adic étale realisation functor, and to $U_{\PL}^{\dR}$ under the de Rham realisation functor.
\end{rem}

Let $U_{\PL,n}$ (resp.~$U_{\PL,n}^{\dR}$) denote the depth-$n$ quotient of the polylogarithmic étale (resp.~de Rham) fundamental group, for $n \leq \infty$, i.e.\ the quotient of $U_{\PL}$ (resp.\ $U_{\PL}^{\dR}$) by the $(n+1)$th step of the descending central series. The group $U_{\PL,n}^{\dR}$ has coordinate ring given by
\begin{equation}
	\label{eq: polylog fundamental group coordinates}
	\cO(U_{\PL,n}^{\dR}) = \bQ[\log, \Li_1, \Li_2,\ldots,\Li_n].
\end{equation}
The coordinates $\log,\Li_1,\Li_2,\dots$ on~$U^{\dR}_{\PL,n}$ are the formal logarithm and formal polylogarithms, whose definitions we briefly recall. Like any functionals on a Tannaka group, they can be described formally as matrix coefficients, i.e.~triples~$(V,v,\tau)$ where~$V$ is a unipotent vector bundle with integrable connection, $v\in V_b$ and $\phi\in V_b^*$. A matrix coefficient $(V,v,\phi)$ determines a functional in $U^{\dR}\to\bA^1$ taking an element $\gamma\in U^{\dR}(R)$ to $\phi(\gamma_V(v))\in R$. In our case, the logarithm $\log$ means the functional on~$U^{\dR}$ determined by the matrix coefficient where~$V$ is the trivial vector bundle $\cO_Y^{\oplus 2}$ with connection
\[
\nabla - \begin{pmatrix}0&\frac{\rd z}z\\0&0\end{pmatrix} \,,
\]
with $v=e_2$ the second basis vector in $V_b=\bQ^{\oplus2}$ and $\phi=\proj_1$ the first projection. The polylogarithm $\Li_n$ means the functional where~$V=\cO_Y^{\oplus n+1}$ with connection
\[
\nabla - \begin{pmatrix}0&\frac{\rd z}z&&&&\\&0&\frac{\rd z}z&&&\\&&\ddots&\ddots&&\\&&&0&\frac{\rd z}z&\\&&&&0&\frac{\rd z}{1-z}\\&&&&&0\end{pmatrix} \,,
\]
where~$n-1$ of the superdiagonal entries are~$\frac{\rd z}z$, the last is $\frac{\rd z}{1-z}$, and all other entries are~$0$. We take $v=e_{n+1}$ and~$\phi=\proj_1$.

In terms of these coordinates, the $p$-adic de Rham Kummer map
\[j_p^{\dR}\colon \cY(\bZ_p) \to U_{\PL}^{\dR}(\bQ_p) \]
is given by
\[
\log(j_p^{\dR}(z)) = \log^p(z) \quad\text{and}\quad \Li_k(j_p^{\dR}(z)) = \Li_k^p(z) \quad \text{for $k \geq 1$} \,.
\]
In other words, the algebraic functions $\log$ and $\Li_k$ on the de Rham fundamental group pull back to the $p$-adic logarithm and polylogarithms on $\cY(\bZ_p)$. We sometimes drop the superscript $(-)^p$ and write simply $\log(z)$ and $\Li_k(z)$ when no confusion with the algebraic functions on $U_{\PL,n}^{\dR}$ by the same name is likely to arise.

\subsection{Coordinates on the polylogarithmic Selmer scheme}

In order to carry out the Chabauty--Kim method for the polylogarithmic quotient, we need to understand the Bloch--Kato Selmer scheme $\rH^1_{f,S}(G_\bQ,U_{\PL,n})$. Its motivic analogue was studied in \cite{CDC:polylog1}, so we apply our motivic-étale comparison theorem from \Cref{sec: CK comparison} to transfer their results to the étale setting.

\begin{prop}
	\label{thm: coordinates on polylog Selmer scheme}
	The Selmer scheme~$\rH^1_{f,S}(G_\bQ,U_{\PL,n})$ is an affine space over $\QQ_p$ of dimension 
	\[ \dim \, \rH^1_{f,S}(G_\bQ,U_{\PL,n}) = 2\#S + \lfloor (n-1)/2 \rfloor, \]
	with coordinate ring isomorphic to
	\begin{equation}
		\label{eq: polylog Selmer scheme coordinates}
		\cO(\rH^1_{f,S}(G_\bQ,U_{\PL,n})) \cong \bQ_p[(x_{\ell})_{\ell \in S}, (y_{\ell})_{\ell \in S}, z_3, z_5, \ldots ]
	\end{equation}
	where the coordinates $z_{2i+1}$ are indexed by the odd integers in the interval $[3,n]$.
\end{prop}

\begin{proof}
	By \Cref{thm:comparison_selmer_schemes}, the étale Selmer scheme is isomorphic to the motivic one, basechanged from~$\bQ$ to $\bQ_p$:
	\[ \rH^1_{f,S}(G_\bQ,U_{\PL,n}) \cong \rH^1(G_{\bQ,S}^{\MT},U_{\PL,n}^{\dR})_{\bQ_p}. \]
	By \cite[Corollary~3.11 and §3.3.1]{CDC:polylog1}, the motivic Selmer scheme is an affine space with specified coordinates. The translation between their and our naming convention is as follows:
	\[x_\ell = \Phi_{e_0}^{\tau_\ell}, \quad y_\ell = \Phi_{e_1}^{\tau_\ell}, \quad z_{2i+1} = \Phi^{\sigma_{2i+1}}_{e_1 e_0\cdots e_0}. \qedhere \]
\end{proof}

\begin{rem}
	\label{rem: coordinates not canonical}
	The coordinates $x_{\ell}$, $y_{\ell}$, $z_{2i+1}$ in \Cref{thm: coordinates on polylog Selmer scheme} can be constructed as follows. Recall from §\ref{sec: comparison theorem} that we have a description of the motivic Selmer scheme as a space of $\bG_m$-equivariant cocycles
	\[ \rH^1(G_{\bQ,S}^{\MT}, U_{\PL,n}^{\dR}) = \rZ^1(U_{\bQ,S}^{\MT}, U_{\PL,n}^{\dR})^{\bG_m}. \]
	Since the action of $U^{\MT}_{\bQ,S}$ on the polylogarithmic fundamental group $U_{\PL}^{\dR}$ is trivial \cite[Proposition~16.13]{deligne:droite-projective}, $\bG_m$-equivariant cocycles agree with $\bG_m$-equivariant homomorphisms. Moreover, homomorphisms of pro-unipotent groups are in bijection with homomorphisms between their Lie algebras, so the motivic Selmer scheme is the space of graded homomorphisms
	\[ \rH^1(G_{\bQ,S}^{\MT}, U_{\PL,n}^{\dR}) = \Hom(\Lie(U^{\MT}_{\bQ,S}), \Lie(U_{\PL,n}^{\dR}))^{\bG_m}. \]
	As discussed in \cite[§4.1]{CDC:polylog1}, the Lie algebra of $U_{\bQ,S}^{\MT}$ is free pro-nilpotent on a set of generators $\{\tau_{\ell} : \ell \in S\} \cup \{\sigma_3, \sigma_5, \ldots\}$ with the $\tau_{\ell}$ living in degree (= half-weight)~$-1$ and $\sigma_{2i+1}$ living in degree~$-(2i+1)$. On the other hand, a basis of $\Lie(U_{\PL,n}^{\dR})$ is given by $\{e_0,e_1\}$ in degree~$-1$ and $\mathrm{ad}(e_0)^{k-1} e_1$ in degree $2 \leq k \leq n$. Thus, for a graded cocycle $c \colon \Lie(U^{\MT}_{\bQ,S}) \to \Lie(U_{\PL,n}^{\dR})$ (defined over some $\bQ$-algebra~$R$) we can write
	\begin{align*}
		c(\tau_{\ell}) &= x_{\ell}(c) e_0 + y_{\ell}(c) e_1 \quad (\ell \in S) \\
		c(\sigma_{2i+1}) &= z_{2i+1}(c) \mathrm{ad}(e_0)^{2i}e_1 \quad (1 \leq i \leq \lfloor (n-1)/2\rfloor)
	\end{align*}
	with $x_{\ell}(c), y_{\ell}(c), z_{2i+1}(c) \in R$. This defines the coordinates on the motivic Selmer scheme.
	Note that their construction depends on the choice of free generators of $\Lie(U_{\bQ,S}^{\MT})$. There is a natural choice for the $\tau_{\ell}$, which determines the coordinates~$x_{\ell}$ and~$y_{\ell}$, but there is no (obvious) natural choice for the $\sigma_{2i+1}$, so the coordinates $z_{i+1}$ on $\rH^1_{f,S}(G_\bQ,U_{\PL,n})$ are not canonical. 
\end{rem}

\begin{rem}
	\label{rem: x and y coordinates}
	The coordinates $x_{\ell}$ and $y_{\ell}$ on $\rH^1_{f,S}(G_\bQ,U_{\PL,n})$ from \Cref{thm: coordinates on polylog Selmer scheme} agree with those defined in \Cref{sec: refined abelian Selmer scheme} on the abelian Selmer scheme $\rH^1_{f,S}(G_\bQ,U^{\ab})$ when they are pulled back along the natural morphism
	\[ \rH^1_{f,S}(G_\bQ,U_{\PL,n}) \to \rH^1_{f,S}(G_\bQ,U^{\ab}) \]
	induced by the abelianisation map $U_{\PL,n} \twoheadrightarrow U_{\PL,1} = U^{\ab}$. 
\end{rem}

\subsection{The Chabauty--Kim diagram for the polylog quotient}

Consider the Chabauty--Kim diagram~\eqref{diag:kim_square_unrefined} for the depth-$n$ polylogarithmic quotient~$U_{\PL,n}$:
\begin{equation}
	\label{diag: CK diagram polylog quotient}
	\begin{tikzcd}
		\cY(\bZ_S) \arrow[r,hook]\arrow[d,"j_S"] & \cY(\bZ_p) \arrow[d,"j_p"] \\
		\rH^1_{f,S}(G_\bQ,U_{\PL,n})(\bQ_p) \arrow[r,"\loc_p"] & \rH^1_f(G_p,U_{\PL,n})(\bQ_p) \,.
	\end{tikzcd}
\end{equation}
Thanks to our motivic-étale comparison theorem and the work of \cite{CDC:polylog1}, we understand the localisation map quite explicitly. We have the coordinates~\eqref{eq: polylog Selmer scheme coordinates} on $\rH^1_{f,S}(G_\bQ,U_{\PL,n})$ and the coordinates~\eqref{eq: polylog fundamental group coordinates} on $\rH^1_f(G_p,U_{\PL,n})$ via the Bloch--Kato logarithm $\rH^1_f(G_p,U_{\PL,n}) \cong (U_{\PL,n}^{\dR})_{\bQ_p}$.

\begin{thm}
	\label{thm: polylog localisation map}
	The localisation map $\loc_p\colon \rH^1_{f,S}(G_\bQ,U_{\PL,n}) \to \rH^1_f(G_p,U_{\PL,n})$ in diagram~\eqref{diag: CK diagram polylog quotient} corresponds to the homomorphism on coordinate rings
	\[ \loc_p^{\sharp} \colon \QQ_p[\log,\Li_1,\ldots,\Li_n] \to \QQ_p[(x_\ell)_{\ell \in S}, (y_\ell)_{\ell \in S}, (z_{2i+1})_{1\leq i\leq \lfloor(n-1)/2\rfloor}], \]
	which is explicitly given as follows:
	\begin{align}
		\label{eq: localisation log}
		\loc_p^{\sharp}(\log) &= \sum_{\ell \in S} a_{\tau_{\ell}} x_\ell,\\
		\label{eq: localisation Li_k}
		\loc_p^{\sharp}(\Li_k) &= \sum_{\ell_1,\ldots,\ell_{k-1},q \in S} a_{\tau_{\ell_1}\cdots \tau_{\ell_{k-1}} \tau_q} x_{\ell_1}\cdots x_{\ell_{k-1}} y_q \\
		& \quad + \sum_{i=1}^{\lfloor (k-1)/2 \rfloor} \sum_{\ell_1,\ldots,\ell_{k-2i-1} \in S} a_{\sigma_{2i+1} \tau_{\ell_1}\cdots \tau_{\ell_{k-2i-1}}} x_{\ell_1}\cdots x_{\ell_{k-2i-1}} z_{2i+1}. \nonumber
	\end{align}
	Here, the coefficients $a_w$ whose subscripts are words in the symbols $\{\tau_{\ell} : \ell \in S\} \cup \{\sigma_3,\sigma_5,\ldots\}$ are constants in~$\bQ_p$.
\end{thm}

\begin{proof}
	By \Cref{thm:comparison_selmer_schemes}, the étale localisation map is isomorphic to the motivic localisation map~\eqref{eq: motivic localisation map}
	\[ \ev_p\colon \rH^1(G_{\bQ,S}^{\MT},U_{\PL,n}^{\dR})_{\bQ_p} \to (U_{\PL,n}^{\dR})_{\bQ_p}. \]
	As discussed in \Cref{sec: comparison theorem}, this map is given by identifying $\rH^1(G_{\bQ,S}^{\MT},U_{\PL,n}^{\dR})_{\bQ_p}$ with the space $\rZ^1(U_{\bQ,S}^{\MT},U_{\PL,n}^{\dR})^{\bG_m}_{\bQ_p}$ of $\bG_m$-equivariant cocycles and evaluating at a certain $\bQ_p$-point $(\eta_p^{\ur})^{-1} \in U_{\bQ,S}^{\MT}(\bQ_p)$. This is equivalent to specialising the ``universal cocycle evaluation map'' \cite[Definition~2.20]{CDC:polylog1} to the point $(\eta_p^{\ur})^{-1}$. A formula for the universal cocycle evaluation map is proved in \cite[Corollary~3.11]{CDC:polylog1}. That formula involves various motivic periods $f_w \in \cO(U_{\bQ,S}^{\MT})$; specialising to $(\eta_p^{\ur})^{-1}$ amounts to evaluating them at this point, or equivalently applying the $p$-adic period map $\per_p\colon \cO(U_{\bQ,S}^{\MT}) \to \bQ_p$:
	\[ a_w \coloneqq f_w((\eta_p^{\ur})^{-1}) = \per_p(f_w) \in \bQ_p. \]
	In this way the formula by Corwin--Dan-Cohen specialises to the one stated here, up to using different names for the coordinates.
\end{proof}

\begin{rem}
	\label{rem: p-adic periods not canonical}
	As discussed in \Cref{rem: coordinates not canonical} above, the choice of coordinates on the polylogarithmic Selmer scheme is not canonical but depends on a choice of free generators $\{\tau_{\ell} : \ell \in S\} \cup \{\sigma_3, \sigma_5, \ldots \}$ of the Lie algebra of $U_{\bQ,S}^{\MT}$. As a consequence, the $p$-adic constants $a_w$ appearing as coefficients in the localisation map~\eqref{eq: localisation log}--\eqref{eq: localisation Li_k} also depend on this choice. The subscripts of the $a_w$ are in fact words in those chosen generators.	
\end{rem}

Partly due to the choices involved in their definition, the $p$-adic periods $a_w$ are difficult to determine in practice. As explained in \cite[§4.1]{CDC:polylog1}, there are however choices for which at least those $a_w$ subscripted by a single letter $\tau_{\ell}$ or $\sigma_{2i+1}$ have particular known values. These values are $p$-adic logarithms of the primes in~$S$ in the case of $a_{\tau_{\ell}}$ and $p$-adic zeta values in the case of $a_{\sigma_{2i+1}}$:
\begin{align}
	\label{eq: p-adic periods log}
	a_{\tau_{\ell}} &= \log^p(\ell) \quad \text{for $\ell \in S$},\\
	\label{eq: p-adic periods zeta}
	a_{\sigma_{2i+1}} &= \zeta^p(2i+1) \quad \text{for $i \geq 1$}.
\end{align}
We shall adopt this choice in the following.

For concreteness, let us spell out the localisation map in depth~4:
\begin{align}
	\label{eq: concrete localisation map log}
	\loc_p^{\sharp}(\log) &= \sum_{\ell \in S} \log^p(\ell) x_\ell,\\
	\label{eq: concrete localisation map Li1}
	\loc_p^{\sharp}(\Li_1) &= \sum_{\ell \in S} \log^p(\ell) y_\ell,\\
	\label{eq: concrete localisation map Li2}
	\loc_p^{\sharp}(\Li_2) &= \sum_{\ell,q \in S} a_{\tau_\ell \tau_q} x_{\ell} y_q,\\
	\loc_p^{\sharp}(\Li_3) &= \sum_{\ell_1,\ell_2,q \in S} a_{\tau_{\ell_1} \tau_{\ell_2} \tau_q} x_{\ell_1} x_{\ell_2} y_q + \zeta^p(3) z_3,\\
	\loc_p^{\sharp}(\Li_4) &= \sum_{\ell_1,\ell_2,\ell_3,q \in S} a_{\tau_{\ell_1} \tau_{\ell_2} \tau_{\ell_3} \tau_q} x_{\ell_1} x_{\ell_2} x_{\ell_3} y_q + \sum_{\ell \in S} a_{\sigma_3 \tau_\ell} x_{\ell} z_3.
\end{align}

\begin{rem}
	\label{rem: DCW coefficients}
	The first three formulas \eqref{eq: concrete localisation map log}--\eqref{eq: concrete localisation map Li2} describe the localisation map in depth~2 that was studied in \cite{DCW:explicitCK} and \cite[§2.3]{BBKLMQSX}. The coefficients $a_{\tau_{\ell} \tau_{q}}$ of the bilinear polynomial \eqref{eq: concrete localisation map Li2} agree with those denoted $a_{\ell,q}$ in \cite{BBKLMQSX}.\footnote{Compared to \cite{BBKLMQSX}, our $y_{\ell}$-coordinates on the Selmer scheme as well as the third coordinate on the de Rham fundamental group differ by a sign. These two signs cancel out, so that the resulting coefficients for the bilinear map are the same.} They do not depend on choices. While there is no closed formula for them, there exists an algorithm based on Tate's calculation of the Milnor $K$-group $K_2(\bQ)$ \cite[Theorem~11.6]{milnor} to express the $a_{\tau_{\ell} \tau_{q}}$ as $\bQ$-linear combinations of dilogarithms of rational numbers (see \cite{KLS:dcwcoefficients} for our implementation of this algorithm in SageMath).
\end{rem}

The following proposition summarises what we know about the polylogarithmic Chabauty--Kim diagram:

\begin{prop}
	The Chabauty--Kim diagram for the polylogarithmic quotient~\eqref{diag: CK diagram polylog quotient} is isomorphic to the following diagram
	\begin{equation}
		\label{eq: motivic CK diagram polylog quotient}
		\begin{tikzcd}
				\cY(\bZ_S) \arrow[r,hook]\arrow[d,"j_S"] & \cY(\bZ_p) \arrow[d,"j_p"] \\
				\Spec(\bQ_p[(x_{\ell})_{\ell \in S}, (y_{\ell})_{\ell \in S}, z_3, z_5, \ldots ])(\bQ_p) \arrow[r,"\loc_p"] & \Spec(\bQ_p[\log,\Li_1,\ldots,\Li_n])(\bQ_p) \,.
			\end{tikzcd}
	\end{equation}
	The localisation map $\loc_p$ is given by Eqs.~\eqref{eq: localisation log}--\eqref{eq: localisation Li_k}, the local Kummer map $j_p$ is given by \begin{align*}
		\log(j_p(z)) &= \log^p(z),\\
		\Li_k(j_p(z)) &= \Li_k^p(z) \quad \text{for $k \geq 1$},
	\end{align*} 
	and the $x_{\ell}$- and $y_{\ell}$-components of the global Kummer map~$j_S$ are given by
	\begin{align*}
		x_{\ell}(j_S(z)) &= v_{\ell}(z), \\
		y_{\ell}(j_S(z)) &= -v_{\ell}(1-z).
	\end{align*}
\end{prop}

\begin{rem}
	As mentioned in \Cref{rem: coordinates not canonical}, the $z_{2i+1}$-coordinates on the Selmer scheme depend on choices. As a consequence, the $z_{2i+1}$-components of the global Kummer map~$j_S$ are difficult to understand. It follows however from its motivic origin that the numbers $z_{2i+1}(j_S(z))$ (along with $x_{\ell}(j_S(z))$ and $y_{\ell}(j_S(z))$) are contained in~$\bQ \subset \bQ_p$ and independent of~$p$.
\end{rem}

From \Cref{sec: refined} we also know the equations defining the \emph{refined} Selmer scheme inside the full Selmer scheme:

\begin{prop}
	\label{thm: refined polylog Selmer scheme}
	Assume $2 \in S$. Then, in terms of the coordinates from \Cref{thm: coordinates on polylog Selmer scheme}, the refined Selmer scheme $\Sel_{S,\PL,n}^{\min}(\cY)$ for the polylogarithmic quotient is defined inside $\rH^1_{f,S}(G_\bQ,U_{\PL,n})$ by the equations
	\[ x_{\ell} y_{\ell} (x_{\ell} + y_{\ell}) = 0 \quad \text{for $\ell \in S$}. \]
	It can be written as a union
	\[ \Sel_{S,\PL,n}^{\min}(\cY) = \bigcup_{\Sigma} \Sel_{S,\PL,n}^{\Sigma}(\cY) \]
	over the $3^{\#S}$ refinement conditions $\Sigma \in \{0,1,\infty\}^S$, where $\Sel_{S,\PL,n}^{\Sigma}(\cY)$ is defined by the equations
	\begin{align*} 
		y_{\ell} &= 0 \quad \text{if $\Sigma_{\ell} = 0$},\\
		x_{\ell} &= 0 \quad \text{if $\Sigma_{\ell} = 1$},\\
		x_{\ell} + y_{\ell} &= 0 \quad \text{if $\Sigma_{\ell} = \infty$},
	\end{align*}
	for $\ell \in S$.
\end{prop}

\begin{proof}
	This is a special case of \Cref{thm: refined Selmer scheme}, combined with the fact that the coordinates $x_{\ell}$ and $y_{\ell}$ on the polylogarithmic Selmer scheme are those coming from the abelian Selmer scheme (see \Cref{rem: x and y coordinates}).
\end{proof}

\subsection{Equations for a refined polylogarithmic Chabauty--Kim locus}

Now let $S = \{2\}$ and let $p$ be any odd prime. We carry out the refined Chabauty--Kim method in this case, starting by determining the functions cutting out the refined Chabauty--Kim locus for the particular refinement condition $\Sigma = (1)$. Abbreviate $x\coloneqq x_2$, $y \coloneqq y_2$, and $\tau \coloneqq \tau_2$. Then by \Cref{thm: polylog localisation map}, the localisation map for the polylogarithmic quotient $U_{\PL,n}$ corresponds on coordinate rings to the homomorphism
\[ \loc_p^{\sharp}\colon \bQ_p[\log,\Li_1,\ldots,\Li_n] \to \bQ_p[x, y, (z_{2i+1})_{1\leq i\leq \lfloor(n-1)/2\rfloor}] \]
given by
\begin{align}
	\label{eq: localisation map 2 log}
	\loc_p^{\sharp}(\log) &= a_{\tau} x,\\
	\label{eq: localisation map 2 Li_k}
	\loc_p^{\sharp}(\Li_k) &= a_{\tau^k} x^{k-1}y + \sum_{i=1}^{\lfloor(k-1)/2\rfloor} a_{\sigma_{2i+1} \tau^{k-2i-1}} x^{k-2i-1} z_{2i+1} \quad \text{($1 \leq k \leq n$)}.
\end{align}

By \Cref{thm: refined polylog Selmer scheme}, the \emph{refined} Selmer scheme is the union of the three hyperplanes
\[ \Sel_{\{2\},\PL,n}^{\min}(\cY) = \Sel_{\{2\},\PL,n}^{(0)}(\cY) \cup \Sel_{\{2\},\PL,n}^{(1)}(\cY) \cup \Sel_{\{2\},\PL,n}^{(\infty)}(\cY) \]
defined by $y = 0$, $x = 0$, and $x+y= 0$, respectively. We calculate the restriction of the localisation map to the refined subscheme for the refinement condition~$\Sigma = (1)$. Denote by $i^{(1)}$ the inclusion of the refined subscheme $\Sel_{\{2\},\PL,n}^{(1)}(\cY)$ in the full Selmer scheme. 

\begin{prop}
	\label{thm: refined localisation map 2}
	On the subspace $\Sel_{\{2\},\PL,n}^{(1)}(\cY)$ defined by $x=0$ the localisation map \eqref{eq: localisation map 2 log}--\eqref{eq: localisation map 2 Li_k} restricts as follows:
	\begin{align*}
		(i^{(1)} \circ \loc_p)^{\sharp}(\log) &= 0, \\
		(i^{(1)} \circ \loc_p)^{\sharp}(\Li_1) &= a_{\tau} y,\\
		(i^{(1)} \circ \loc_p)^{\sharp}(\Li_{k}) &= 0 \quad \text{for $k \geq 2$ even},\\
		(i^{(1)} \circ \loc_p)^{\sharp}(\Li_{k}) &= a_{\sigma_{k}} z_{k} \quad \text{for $k \geq 3$ odd}.
	\end{align*}
\end{prop}

\begin{proof}
	For $k \geq 2$, all terms in \eqref{eq: localisation map 2 Li_k} contain factors of $x$, except the very last summand if $k$ is odd.
\end{proof}

\Cref{thm: refined localisation map 2} tells us immediately that the functions $\log$ and $\Li_{k}$ for $k \geq 2$ even vanish on the scheme-theoretic image of $\Sel_{\{2\},\PL,n}^{(1)}(\cY)$ under the localisation map. Pulling these back along the local Kummer map gives us many $p$-adic analytic functions that vanish on the corresponding refined Chabauty--Kim locus:

\begin{prop}
	\label{thm: refined equations}
	On the polylogarithmic refined Chabauty--Kim locus $\cY(\ZZ_p)_{\{2\},\PL,n}^{(1)}$ of depth $n$, the following equations hold:
	\begin{align}
		\log(z) &= 0, \\
		\Li_{k}(z) &= 0 \quad \text{for $2 \leq k \leq n$ even}.
	\end{align}
\end{prop}

\begin{rem}
	\label{rem: generalising AWS Theorem A}
	\Cref{thm: refined equations} generalises \cite[Theorem~A]{BBKLMQSX} from depth~$n=2$ to arbitrary depth.
\end{rem}

\subsection{Proof of refined Kim's Conjecture for $S=\{2\}$}
\label{sec: proof of Kim conjecture}

Recall that the refined Kim's Conjecture for $S= \{2\}$ and any odd prime~$p$ states that the inclusion 
\[ \{2,-1, 1/2\} = \cY(\bZ[1/2]) \subseteq \cY(\ZZ_p)_{\{2\},\infty}^{\min}\]
is an equality. The refined Chabauty--Kim locus $\cY(\ZZ_p)_{\{2\},\infty}^{(1)}$ for the refinement condition $\Sigma = (1)$ should therefore consist only of $-1$, the unique $\{2\}$-integral point reducing to the cusp~$1$ modulo~$2$ (see \Cref{rem: refinement conditions S-integral points}). We first show that this is indeed the case already for the polylogarithmic quotient (\Cref{thm: refined locus 1} below), and then we exploit the $S_3$-action to deduce the refined Kim's Conjecture for the full refined locus. 

Note that the element $z=-1$ satisfies the equations from \Cref{thm: refined equations}:
\[ \log(-1) = 0 \quad \text{and} \quad \Li_k(-1) = 0 \text{ for $k\geq 2$ even} \]
since $-1$ is contained in $\cY(\bZ_p)_{\{2\},\PL,\infty}^{(1)}$. We show that $z = -1$ is the only solution. The key to this is the following result:

\begin{prop}
	\label{thm: Li vanishes only on -1}
	Let~$p \geq 5$ be a prime, and let~$\cY=\bP^1_{\bZ_p}\smallsetminus\{0,1,\infty\}$. Then the only element $z\in \cY(\bZ_p)$ satisfying $\log(z)=0$ and $\Li_{p-3}(z)=0$ is~$z=-1$.
\end{prop}

\begin{proof}
	The modified $p$-adic polylogarithm
	\[
	\Li^{(p)}_n(z) \coloneqq \Li_n(z)-\frac1{p^n}\Li_n(z^p) \,,
	\]
	is $\bZ_p$-valued and is congruent modulo~$p$ to~$\frac1{1-z^p}\li_n(z)$ where
	\[
	\li_n(z) \coloneqq \sum_{k=1}^{p-1}\frac{z^k}{k^n} \in \bF_p[z]
	\]
	is the finite polylogarithm \cite[Proposition~2.1]{besser:finite_polylogs}. In the particular case that~$n=p-3$, we have
	\[
	\li_{p-3}(z) = \sum_{k=1}^{p-1}k^2z^k = \frac{z(z+1)(z^p-1)}{(z-1)^3} = z(z+1)(z-1)^{p-3}
	\]
	and so the only solution to~$\li_{p-3}(z)=0$ in $\bF_p\smallsetminus\{0,1\}$ is~$z=-1 \bmod p$.
	
	If $z\in \cY(\bZ_p)$ satisfies $\log(z)=0$ and $\Li_{p-3}(z) = 0$, then~$z$ is a root of unity in~$\bZ_p$, so satisfies~$z^p=z$. Hence~$z$ also satisfies~$\Li^{(p)}_{p-3}(z)=0$, and so~$\li_{p-3}(\bar z)=0$, where~$\bar z\in\bF_p$ is the reduction of~$z$. But we've seen that this implies that~$\bar z=-1$, which implies that~$z=-1$ as desired.
\end{proof}

\begin{cor}
	\label{thm: refined locus 1}
	For any odd prime~$p$ we have 
	\[ \cY(\bZ_p)_{\{2\},\PL,\infty}^{(1)} = \{-1\}.\]
\end{cor}

\begin{proof}
	Let $z \in \cY(\bZ_p)_{\{2\},\PL,\infty}^{(1)}$, so $z$ satisfies the equations from \Cref{thm: refined equations}: $\log(z) = 0$ and $\Li_k(z) = 0$ for $k \geq 2$ even. For $p \geq 5$, we get $z = -1$ from \Cref{thm: Li vanishes only on -1}. For $p = 3$, already the first equation $\log(z) = 0$ implies $z = -1$ since that is the only root of unity in $\cY(\bZ_3)$.
\end{proof}

We can now determine the full refined Chabauty--Kim locus by exploiting the $S_3$-action on the thrice-punctured line. At this point we have to use the full fundamental group rather than its polylogarithmic quotient in order to have the $S_3$-symmetries present on the refined Chabauty--Kim locus (see \Cref{rem: S3-action}).

\begin{cor}
	\label{thm: refined CK for prime 2}
	The thrice-punctured line satisfies the refined Chabauty--Kim Conjecture for $S = \{2\}$ and all odd primes~$p$:
	\[ \cY(\bZ_p)_{\{2\},\infty}^{\min} = \{2, -1, 1/2\}. \]
\end{cor}

\begin{proof}
	Let $z$ be an element of the refined Chabauty--Kim locus $\cY(\bZ_p)_{\{2\},\infty}^{\min}$. By \Cref{thm: refined Selmer scheme}, the refined Selmer scheme is a union over three refinement conditions corresponding to the three cusps:
	\[ \cY(\bZ_p)_{\{2\},\infty}^{\min} = \cY(\bZ_p)_{\{2\},\infty}^{(0)} \cup \cY(\bZ_p)_{\{2\},\infty}^{(1)} \cup \cY(\bZ_p)_{\{2\},\infty}^{(\infty)}. \]
	By \Cref{thm: S3-action}, there is an $S_3$-conjugate~$\sigma(z)$ of $z$ contained in $\cY(\bZ_p)_{\{2\},\infty}^{(1)}$. By \Cref{thm: change of fundamental group quotient} applied to the quotient map $U \twoheadrightarrow U_{\PL}$, that refined locus is contained in the corresponding one for the polylogarithmic quotient:
	\[ 	\cY(\bZ_p)_{\{2\},\infty}^{(1)} \subseteq \cY(\bZ_p)_{\{2\},\PL,\infty}^{(1)}.\]
	That latter locus consists only of $-1$ by \Cref{thm: refined locus 1}. In particular we have $\sigma(z) = -1$. The original point~$z$ is thus an element of the $S_3$-orbit of~$-1$, which is precisely $\{2,-1,1/2\}$, the set of $\{2\}$-integral points of~$\cY$.
\end{proof}

\subsection{Proof of Kim's Conjecture for \texorpdfstring{$S = \emptyset$}{S=\unichar{"2205}}}
\label{sec: proof of unrefined kim conjecture}

We have focused in this paper on the refined Chabauty--Kim method due to its role in providing instances of the $S$-Selmer Section Conjecture. It is interesting to observe, however, that the same method for proving the refined conjecture for $S=\{2\}$ can be used to prove the original (non-refined) Kim's Conjecture in the case $S= \emptyset$. This is \Cref{thm:main_unrefined} from the introduction, which we shall prove now.

As explained in \Cref{rem: unrefined CK}, if $\pi_1^{\et,\bQ_p}(Y_{\overline{\bQ}}, b) \twoheadrightarrow \Pi$ is any $G_{\bQ}$-equivariant quotient of the $\bQ_p$-pro-unipotent étale fundamental group, we can use the Bloch--Kato Selmer scheme $\rH^1_{f,S}(G_{\bQ}, \Pi)$ to define the (non-refined) Chabauty--Kim locus $\cY(\bZ_p)_{S,\Pi}$ as the set of points $y \in \cY(\bZ_p)$ such that $j_p(z)$ lies in the scheme-theoretic image of $\rH^1_{f,S}(G_{\bQ}, \Pi)$ under the localisation map~$\loc_p$ \cite{kim:motivic, kim:albanese}. The classical formulation of Kim's Conjecture states that the Chabauty--Kim locus $\cY(\bZ_p)_{S,n}$ for the depth-$n$ quotient of the fundamental group should consist exactly of the $S$-integral points for sufficiently large~$n$:

\begin{conj}[{Kim's Conjecture \cite[Conjecture~3.1]{BDCKW}}]
	$\cY(\bZ_S) = \cY(\bZ_p)_{S,n}$ for $n \gg 0$.
\end{conj}

Now let $S = \emptyset$. In this case there are no $S$-integral points: $\cY(\bZ) = \emptyset$. It is shown in \cite[§6]{BDCKW} that Kim's Conjecture holds in depth~1 for $p=3$ and all primes $p \equiv 2 \bmod 3$; moreover, it is conjectured (and verified numerically for $p < 10^5$) that Kim's Conjecture holds in depth~$2$ for all primes $p \equiv 1 \bmod 3$. Using the same calculation with finite polylogarithms as in \Cref{thm: Li vanishes only on -1} we can show the conjecture for all primes by going into higher depth:

\begin{thm}
	\label{thm: Kim conjecture emptyset}
	Let $p \geq 5$ be any prime. Then Kim's Conjecture for $S = \emptyset$ holds in depth~$p-3$:
	\[ \cY(\bZ_p)_{\emptyset,p-3} = \emptyset. \]
\end{thm}

\begin{proof}
	It is enough to show the conjecture for the polylogarithmic quotient of the fundamental group: $\cY(\bZ_p)_{\emptyset,\PL,p-3} = \emptyset$. Specialising \Cref{thm: polylog localisation map} to $S=\emptyset$, the localisation map $\loc_p\colon \rH^1_{f,\emptyset}(G_{\bQ}, U_{\PL,n}) \to \rH^1_f(G_p, U_{\PL,n})$ corresponds to the homomorphism on coordinate rings
	\[ \loc_p^\sharp\colon \bQ_p[\log,\Li_1,\ldots,\Li_n] \to \bQ_p[(z_{2i+1})_{1\leq i \leq \lfloor (n-1)/2\rfloor}] \]
	given by
	\begin{align*}
		\loc_p^{\sharp}(\log) &= 0,\\
		\loc_p^{\sharp}(\Li_1) &= 0,\\
		\loc_p^{\sharp}(\Li_k) &= 0 \quad \text{for $k \geq 2$ even},\\
		\loc_p^{\sharp}(\Li_k) &= a_{\sigma_k} z_k \quad \text{for $k \geq 3$ odd}.
	\end{align*}
	This implies that all elements of the Chabauty--Kim locus $\cY(\bZ_p)_{\emptyset,\PL,n}$ satisfy the equations
	\begin{equation}
		\label{eq: unrefined equations}
		\log(z) = 0, \quad \Li_1(z) = 0, \quad \Li_k(z) = 0 \; \text{for $2 \leq k \leq n$ even}.
	\end{equation}
	By \Cref{thm: Li vanishes only on -1}, the only element of $\cY(\bZ_p)$ satisfying $\log(z) = 0$ and $\Li_{p-3}(z) = 0$ is $z = -1$. However, that element is not a root of $\Li_1$:
	\[ \Li_1(-1) = -\log(1 - (-1)) = -\log(2) \neq 0. \]
	Thus, there are no solutions to the equations~\eqref{eq: unrefined equations} if $n \geq p-3$.
\end{proof}

In conclusion, Kim's Conjecture holds in depth~1 for $p=2,3$ by \cite[§6]{BDCKW} and in depth~$p-3$ for $p \geq 5$ by \Cref{thm: Kim conjecture emptyset}. This proves \Cref{thm:main_unrefined} from the introduction.

\begin{rem}
	\label{rem: unrefined equations in DCW}
	The equations \eqref{eq: unrefined equations} have previously been derived in \cite[Theorem~1.13]{DCW:mixedtate1}.\footnote{The statement of \cite[Theorem~1.13]{DCW:mixedtate1} contains a typo: it is the $\Li_k(z)$ for $k$ \emph{even} not \emph{odd} which vanish on the Chabauty--Kim locus.} The statement of that theorem contains the assumption that $\zeta^p(n) \neq 0$ for $n \geq 3$ odd (as is conjectured); this assumption is only necessary for saying that the Chabauty--Kim locus $\cY(\bZ_p)_{\emptyset,\PL,\infty}$ is \emph{precisely} the one cut out by the given equations, whereas we only need that the Chabauty--Kim locus is contained in their vanishing set.
\end{rem}

\begin{rem}
	It was shown in \cite[§6]{BDCKW} that the depth 1 locus $\cY(\bZ_p)_{\emptyset,1}$, defined by the equations $\log(z) = 0$ and $\Li_1(z) = 0$, consists precisely of the two primitive sixth roots of unity in $\bZ_p$ if $p \equiv 1 \bmod 3$, and is empty otherwise. Adding in the higher depth equation $\Li_{p-3}(z) = 0$ cuts the set down to the empty set for all~$p$.
\end{rem}

\subsection{Curves of higher genus}
\label{sec: higher genus}

We conclude this paper with the proof of \Cref{thm:main_chabauty}, i.e., the observation that there are also some special higher genus curves where one can prove a Kim-like conjecture, strong enough to deduce the Selmer Section Conjecture. Let~$X/\bQ$ be a smooth projective curve of genus~$\geq2$ with a rational point, and suppose that~$A$ is a quotient of the Jacobian of~$X$ of dimension~$\geq2$, Mordell--Weil rank~$0$ and finite Tate--Shafarevich group. Since~$A$ is a quotient of the Jacobian, its $\bQ_p$-linear Tate module~$V_pA$ is an abelian quotient of the $\bQ_p$-pro-unipotent \'etale fundamental group of~$X$. The assumptions on the Mordell--Weil rank and Tate--Shafarevich group imply that
\[
\rH^1_f(G_\bQ,V_pA)=0 \,,
\]
and hence that the Chabauty--Kim locus~$X(\bQ_p)_{V_pA}$ associated to this quotient is the kernel of the map
\[
j_p\colon X(\bQ_p) \to \rH^1_f(G_p,V_pA)
\]
(when~$p$ is a prime of good reduction).

This can be reinterpreted geometrically. Let
\[
f\colon X \to A
\]
be the composite of the Abel--Jacobi map with the projection from the Jacobian to~$A$. The induced map on $\bQ_p$-pro-unipotent fundamental groups is exactly the quotient map from the fundamental group of~$X$ to~$V_pA$, so the map above factors as the composite
\[
X(\bQ_p) \xrightarrow{f} A(\bQ_p) \xrightarrow{\kappa_p} \rH^1_f(G_p,V_pA)
\]
where~$\kappa_p$ is the $\bQ_p$-linear Kummer map for the abelian variety~$A$. Now the kernel of~$\kappa_p$ is exactly the torsion subgroup~$A(\bQ_p)_{\mathrm{tors}}$, so the Chabauty--Kim locus~$X(\bQ_p)_{V_pA}$ is the set of points~$x\in X(\bQ_p)$ such that~$f(x)$ is a torsion point on~$A$.

Now the Manin--Mumford Conjecture implies that the set of points on~$X_{\Qbar}$ mapping to a torsion point on~$A_{\Qbar}$ is finite, so there is a finite closed subscheme~$Z\subset X$ such that
\[
X(\bQ_p)_{V_pA} = Z(\bQ_p)
\]
for all primes~$p$ of good reduction. This in particular implies Theorem~\ref{thm:main_chabauty}.

Finally, let us explain why the conclusion of Theorem~\ref{thm:main_chabauty}, although weaker than Kim's Conjecture, is nonetheless strong enough to deduce the Selmer Section Conjecture for curves~$X$ as above. The containment $X(\bQ_p)_{V_pA} \subseteq Z(\bQ_p)$ for good primes~$p$ implies by Lemma~\ref{lem:project_into_kim} that any point $x=(x_v)_v\in X(\bA_{\bQ,S})_\bullet^{\fcov}$ in the finite descent locus satisfies $x_p\in Z(\bQ_p)$ for all good primes~$p$. So we deduce from the theorem of the diagonal that~$x\in X(\bQ)$, i.e.\ strong sufficiency of finite descent (Conjecture~\ref{conj:sufficiency_finite_descent}) holds for~$X$. So the Selmer Section Conjecture holds for~$X$.\qed

We remark that strong sufficiency of finite descent for such curves~$X$ (in slightly greater generality) was proven by Stoll \cite[Theorem~8.6]{stoll:finite_descent}. The above argument essentially shows that Stoll's theorem lifts to a statement about Chabauty--Kim loci.

%% file: appendix.tex
\section{Mixed Tate motives vs.\ Galois representations}\label{sec:appendix}

This appendix is devoted to the proof of the following fact, which is broadly known to the experts but for which we have not been able to find a clear reference. 

\begin{prop}\label{prop:etale_realisation_iso_ext}
	Let~$K$ be a number field, $S$ a finite set of places of~$K$, and~$p$ a rational prime not divisible by any element of~$S$. Let~$\MT(\cO_{K,S},\bQ)$ denote the $\bQ$-linear category of mixed Tate motives over~$\cO_{K,S}$ with coefficients in~$\bQ$, and let~$\Rep_{\bQ_p}^{\MT,S}(G_K)$ denote the $\bQ_p$-linear category of mixed Tate representations of~$G_K$, unramified outside $S\cup\{\mathfrak p\mid p\}$ and crystalline at all~$\mathfrak p\mid p$. Then the map
	\begin{equation}\label{eq:etale_realisation_extensions}
		\Ext^1_{\MT(\cO_{K,S},\bQ)}(\bQ(0),\bQ(n)) \otimes_{\bQ} \bQ_p \to \Ext^1_{\Rep_{\bQ_p}^{\MT,S}(G_K)}(\bQ_p(0),\bQ_p(n))
	\end{equation}
	induced by \'etale realisation is an isomorphism for all~$n>0$.
\end{prop}

Since $\Ext^2$'s in $\MT(\cO_{K,S},\bQ)$ vanish \cite[Proposition~1.9]{deligne-goncharov}, this implies by Lemma~\ref{lem:iso_on_tannaka_groups} that the induced map
\[
G_{K,S}^{\MTR} \to G_{K,S,\bQ_p}^{\MT}
\]
on Tannaka groups is an isomorphism, and thus completes the proof of Theorem~\ref{thm:tannaka_group_comparison} from the introduction. 

For the proof of \Cref{prop:etale_realisation_iso_ext} we follow a strategy suggested to us by Annette Huber-Klawitter. We wish to express our gratitude for taking the time to answer our questions during the write-up.

We start by treating the case $n = 1$. In this case the Ext-group on the left-hand side of~\eqref{eq:etale_realisation_extensions} is given by
\[ \Ext^1_{\MT(\cO_{K,S},\bQ)}(\bQ(0),\bQ(1)) = \cO_{K,S}^\times \otimes_{\bZ} \bQ \]
by definition of $\MT(\cO_{K,S}, \bQ)$ \cite[§1.7]{deligne-goncharov}. 
The map~\eqref{eq:etale_realisation_extensions} for $n=1$ can be identified with the map
\[ \cO_{K,S}^\times \otimes_{\bZ} \bQ_p \to \rH^1_{f,S}(G_K, \bQ_p(1)) \]
induced by sending an element of $\cO_{K,S}^\times$ to its étale Kummer class, see the proof of \cite[Prop.~1.8]{deligne-goncharov}. By Kummer theory and \cite[Example~3.9]{bloch-kato:tamagawa_numbers}, this is an isomorphism. 

We now assume $n \geq 2$. In this case, the integrality conditions (unramifiedness and crystallineness) play no role: every extension of $\bQ(0)$ by $\bQ(n)$ in $\MT(K,\bQ)$ automatically lies in $\MT(\cO_{K,S}, \bQ)$, and every extension of $\bQ_p(0)$ by $\bQ_p(n)$ in $\Rep_{\bQ_p}(G_K)$ is automatically unramified away from~$p$ and crystalline at places dividing~$p$ \cite[Example~3.9]{bloch-kato:tamagawa_numbers}. So we have to show that the map
\begin{equation}
	\label{eq:ext-realisation}
	\Ext^1_{\MT(K,\bQ)}(\bQ(0),\bQ(n)) \otimes_{\bQ} \bQ_p \to \Ext^1_{\Rep_{\bQ_p}(G_K)}(\bQ_p(0),\bQ_p(n))
\end{equation}
induced by the étale realisation functor
\begin{equation}
	\label{eq:etale-realisation}
	\rho_{\et}\colon \MT(K,\bQ) \to \Rep_{\bQ_p}(G_K)
\end{equation}
is an isomorphism. Both Ext groups can be computed as Hom groups in suitable triangulated categories. On the one hand, the category $\MT(K,\bQ)$ is, by definition, a full and extension-closed subcategory of Voevodsky's triangulated category of motives $\DM(K)_{\bQ}$, where the subscript~$\bQ$ means that we tensor the Hom groups with~$\bQ$. So the first Ext group in~\eqref{eq:ext-realisation} is given by
\[ \Ext^1_{\MT(K,\bQ)}(\bQ(0), \bQ(n)) = \Hom_{\DM(K)_{\bQ}}(\bQ(0), \bQ(n)[1]). \]
On the other hand, let $D^b(\Spec(K)_{\et}, \bZ_p)$ be Ekedahl's triangulated category of $\bZ_p$-sheaves \cite{ekedahl:adic_formalism}, and let $D^b(\Spec(K)_{\et}, \bQ_p)$ be its isogeny category. The category $\Rep_{\bQ_p}(G_K)$ embeds as a full extension-closed subcategory into $D^b(\Spec(K)_{\et}, \bQ_p)$, so the second Ext group in~\eqref{eq:ext-realisation} is given by
\[ \Ext^1_{\Rep_{\bQ_p}(G_K)}(\bQ_p(0), \bQ_p(n)) = \Hom_{D^b(\Spec(K)_{\et}, \bQ_p)}(\bQ_p(0), \bQ_p(n)[1]). \]

By \cite[§2.1]{huber:realizations}, the étale realisation functor~\eqref{eq:etale-realisation} is the restriction of an exact functor of triangulated categories
\begin{equation}
	\label{eq:realisation-triangulated}
	\rho_{\et}\colon \DM(K)_{\bQ} \to D^b(\Spec(K)_{\et}, \bQ_p).
\end{equation}
Thus, we have to show that the map
\begin{equation}
	\label{eq:realisation-on-homs}
	\rho_{\et,*}\colon \Hom_{\DM(K)_{\bQ}}(\bQ(0), \bQ(n)[1]) \otimes_{\bQ} \bQ_p \to \Hom_{D^b(\Spec(K)_{\et}, \bQ_p)}(\bQ_p(0), \bQ_p(n)[1])
\end{equation}
induced by~\eqref{eq:realisation-triangulated} is an isomorphism.
The Hom group on the left of~\eqref{eq:realisation-on-homs} is, by definition, equal to motivic cohomology:
\[ \Hom_{\DM(K)_{\bQ}}(\bQ(0), \bQ(n)[1]) = \rH^1_{\cM}(\Spec(K), \bQ(n)). \]
The Hom group on the right of~\eqref{eq:realisation-on-homs} computes continuous étale cohomology:
\[ \Hom_{D^b(\Spec(K)_{\et}, \bQ_p)}(\bQ_p(0), \bQ_p(n)[1]) = \rH_{\et}^1(\Spec(K), \bQ_p(n)). \]
Thus, we have to show that the map
\begin{equation}
	\label{eq:motivic-to-etale-cohomology}
	\rho_{\et,*}\colon \rH^1_{\cM}(\Spec(K), \bQ(n)) \otimes_{\bQ} \bQ_p \to \rH_{\et}^1(\Spec(K), \bQ_p(n))
\end{equation}
is an isomorphism.
We reduce this to a result by Soulé by considering the precomposition with the motivic Chern class map. More precisely, consider the $n$-th motivic Chern class map
\begin{equation*}
	\label{eq:chern-class}
	c_n^{\mot} \colon K_{2n-1}(K) \to \rH^1_{\cM}(\Spec(K), \bQ(n))
\end{equation*}
from \cite[Corollary~4.2.2]{huber:realizations}. The composition with the map~\eqref{eq:motivic-to-etale-cohomology} equals the étale Chern class
\[ c_n^{\et}\colon K_{2n-1}(K) \to \rH^1_{\et}(\Spec(K), \bQ_p(n)), \]
see \cite[Corollary~4.2.3]{huber:realizations}. We thus have a commutative diagram as follows:
\[ 
\begin{tikzcd}[row sep=large]
	K_{2n-1}(K) \otimes \bQ_p \arrow[dr,"c_n^{\mot} \otimes \bQ_p"]
	\arrow[rr,"c_n^{\et}"] & & \rH^1_{\et}(\Spec(K), \bQ_p(n)) \\
	& \rH^1_{\cM}(\Spec(K), \bQ(n)) \otimes \bQ_p \arrow[ur, "\rho_{\et,*}"] & 
\end{tikzcd}
\]
The motivic Chern class map induces an isomorphism
\begin{equation}
	\label{eq:motivic-chern-class}
	c_n^{\mot} \otimes \bQ\colon K_{2n-1}(K) \otimes \bQ \cong \rH^1_{\cM}(\Spec(K), \bQ(n)).
\end{equation}
This follows from the fact that for any field, the motivic-to-$K$-theory spectral sequence with $\bQ$-coefficients \cite[Theorem~VI.4.2]{weibel:K-book}
\[ E_2^{i,j} = \rH^{i-j}_{\cM}(\Spec(K), \bQ(-j)) \Rightarrow K_{-i-j}(K) \otimes \bQ \]
degenerates at page~2, and the decomposition of $K_{2n-1}(K) \otimes \bQ$ into Adams eigenspaces induces a splitting \cite[§VI.4]{weibel:K-book}:
\[ K_{2n-1}(K) \otimes \bQ \cong \bigoplus_j \rH^{2j-2n + 1}_{\cM}(\Spec(K), \bQ(j)). \]
The motivic Chern class maps on $K_{2n-1}(K)$ are defined by projecting onto the different components. In our situation $K$ is a number field and thus $K_{2n-1}(K)_{\bQ}$ has only a single Adams eigenspace in weight~$n$ \cite[§1.6]{deligne-goncharov}. It follows that the Chern class map \eqref{eq:motivic-chern-class} and therefore also the map $c_n^{\mot} \otimes \bQ_p$ in the above diagram are isomorphisms, so that we are reduced to proving that the étale Chern class map
\[ c_n^{\et}\colon K_{2n-1}(K) \otimes \bQ_p \cong \rH^1_{\et}(\Spec(K), \bQ_p(n)) \]
is an isomorphism. For $p \neq 2$, this is a theorem of Soulé \cite[Theorem~1]{soule:higher};
for $p = 2$, this is shown in \cite[Lemma~2.19 and Theorem~2.21]{hyperbolic-tesselations}, see also \cite[Theorem~B.4.8]{huber-wildeshaus} and \cite[Remark~2.24]{hyperbolic-tesselations}.